\documentclass{amsart}

\usepackage{
amsfonts,
latexsym,
amssymb,
mathabx,
enumerate,
verbatim,
mathrsfs,
bm,
color,
wasysym,
}

\usepackage[matrix,arrow]{xy}

\newcommand{\labbel}[1]{\label{#1} [[{\bf #1}]]}  
  \newcommand{\bibbitem}[1]{\bibitem{#1} [[{\bf #1}]]}
\renewcommand{\labbel}{\label} \renewcommand{\bibbitem}{\bibitem}  

\numberwithin{equation}{section}

\newtheorem{theorem}{Theorem}[section]
\newtheorem{lemma}[theorem]{Lemma}

\newtheorem{proposition}[theorem]{Proposition}

\newtheorem*{thm26a}{Theorem 2.6.a}

\theoremstyle{definition}
\newtheorem{definition}[theorem]{Definition}

\newtheorem*{def26b}{Definition 2.6.b}
\newtheorem*{prob40a}{Problem 4.0.a}

\newtheorem{problem}[theorem]{Problem}

\theoremstyle{remark}
\newtheorem{remark}[theorem]{Remark}
 
\newtheorem{example}[theorem]{Example}

\newtheorem*{ex32a}{Example 3.2.a}

\newtheorem*{acknowledgement}{Acknowledgement}

\definecolor{reedver}{RGB}{255,0,0}
\definecolor{blueever}{RGB}{0,0,255}
\definecolor{greeenver}{RGB}{0,155,0}

\newcommand{\redver}[1]{{\color{reedver}#1}}  
\newcommand{\bluever}[1]{{\color{blueever}#1}}  
\newcommand{\greenver}[1]{{\color{greeenver}#1}}

\DeclareMathOperator{\On}{On}

\newcommand{\+}{\mathbin{\hash}}

\DeclareMathOperator*{\nsum}
    {\ifdim\displaywidth>0pt {{\mbox{\raisebox{-0.4ex}{\LARGE$\hash$}}}}
     \else{{\mbox{\raisebox{-0.3ex}{\Large$\hash$}}}}
\fi}  

\newcommand{\sumn}{\nsum^N} 

\newcommand{\sums}{\nsum^S}  
\newcommand{\suma}{\nsum^{S2}} 
\newcommand{\sumb}{\nsum^{S1}}

\newcommand{\sumt}{\nsum^T}
\newcommand{\suml}{\nsum^L}

\newcommand*\xbar[1]{
  \hbox{
    \vbox{
      \hrule height 0.5pt 
      \kern0.3ex        
      \hbox{
        \kern-0.3em
        \ensuremath{#1}
      }
    }
  }
}

\newcommand{\sumh}{\sum^{\#}}

\newcommand{\osum}{\bigoplus}

\newcommand{\x}{\mathbin {{\times} \hspace{-6.7 pt}{\times}}}

 \changenotsign  

\DeclareMathSizes{10}{10}{7}{6}

 \allowdisplaybreaks[1]

\begin{document}
 \title{Monotone   Infinitary   Operations   on   Ordinals (extended version)}  

\author{Paolo Lipparini} 
\address{Dipartimento di Mametatica\\Viale della  Ricerca Scientifica\\II Universit\`a di Roma ``Tor Vergata''\\I-00133 ROME ITALY (currently retired)}
\urladdr{http://www.mat.uniroma2.it/\textasciitilde lipparin}

\keywords{Infinitary ordinal sum; infinite mixed sum; Hessenberg natural sum; 
ordinal arithmetics;
strictly monotone operation; 
 surreal number  
}

\subjclass[2010]{Primary 03E10; secondary 06A05}
\thanks{Work performed under the auspices of G.N.S.A.G.A. Work 
partially supported by PRIN 2012 ``Logica, Modelli e Insiemi''.
The author acknowledges the MIUR Department Project awarded to the
Department of Mathematics, University of Rome Tor Vergata, CUP 
E83C18000100006.}

\begin{abstract}
We define and study  an $ \omega $-ary operation $\sums$ 
on the class of the ordinals,
which is strictly monotone in many significant
cases (by an elementary argument, there is no
\emph{fully} strictly monotone infinitary operation
on ordinals). We compare $\sums$ with the finitary 
Hessenberg natural sum, which is the smallest \emph{finitary}
 strictly monotone operation on each argument. 
 We also compare $\sums$
 with  other infinitary generalizations of Hessenberg sum.   

We provide order-theoretical characterizations
of $\sums _{i < \omega} \alpha_{i}$, both as the rank of the sequence 
$( \alpha_i) _{i < \omega} $ 
in an appropriate
well-founded order, and as a mixed (or shuffled)
sum of the ordinals in the sequence.
The latter means that the ordinal $\sums _{i < \omega} \alpha_{i}$
is the largest realization as an order-preserving
 disjoint union of copies of the $\alpha_i$,
under some boundedness restriction.
 The former 
 characterization can be recast in terms of  combinatorial games,
leading to the problem whether $\sums$ can be extended
 to the class of Conway surreal numbers.   
\end{abstract} 
 
\maketitle

 This manuscript is an extended version of 
``A Monotone   Infinitary   Operation   on   Ordinals''  
which has appeared  
in print on Mathematical   Logic   Quarterly 
(Vol. 72, no. 2 (2026), doi:10.1002/malq.70019) 
and which 
has been shortened in the published version
owing to space constraints.
In particular, in Appendix I we recast
some results  in terms of  combinatorial game
theory. In Appendix II we present a more direct proof
of the ``minimality'' part of Theorem \ref{simpth}.

Numberings of theorems, equations etc.
are consistent with numberings in the
published version
(but not so for numbered items in the reference list).

\setcounter{section}{-1}

\section{Introduction} \labbel{intro} 

\subsection{Infinitary natural operations} \labbel{infop} 

In this paper we continue the  study of possible infinitary extensions
of the Hessenberg natural sum on ordinals \cite[Ch. XX]{Hes}. 
Recall that the finitary Hessenberg natural sum on the class of ordinals is 
algebraically better behaved than the more usual
ordinal sum \cite{A,Bac,HP,Sier}. Among other possible definitions, the 
Hessenberg natural sum is the smallest binary operation on the
class of ordinals which is strictly monotone on each argument \cite{T}. 

A transfinite  generalization $\sumh $ 
of the finitary Hessenberg natural  can be defined
 by taking Hessenberg natural sums at successor
stages and suprema at limit stages. 
Chatyrko \cite{Ch} denoted this operation by $(+)$ 
and used it to evaluate various  transfinite dimensions
of products of topological spaces.
The transfinite operation $\sumh $ has been
later extensively studied in \cite{t}.
This operation, together with further ``natural'' infinitary operations,
have been used by Berarducci and Mamino \cite{BM}  in order to get a partial 
solution to an old classical problem by Skolem
concerning the order-type of germs at infinity of a certain class 
of positive real valued functions.

The special case of $\sumh $ in which all summands are equal
has been considered much earlier  in 1909 by  Jacobsthal \cite{J}.
This so called \emph{Jacobsthal product} and further
generalizations \cite{A},
sometimes independently rediscovered \cite{Gonshor},
 have subsequently found
applications to ordinal invariants  associated
to well-partial orderings \cite[Def. 3.6]{dJP} 
and to posets with the finite antichain property
 \cite[Lemma 1.11 and Theorem 2.3]{AB}.

Returning to the case of arbitrary summands,
V\"a\"an\"anen and Wang \cite{VW}
used the countable version of 
  $\sumh$ in order to prove the
 adequacy theorem for an
Ehrenfeucht-Fra\"\i ss\'e game appropriate for
$\mathcal L _{ \omega _1, \omega } $.
The main tool in \cite{VW} 
is a  notion of
``size''  for an 
$\mathcal L _{ \omega _1, \omega } $-formula,
 a  measure of complexity finer than the more
usual quantifier rank.
The size of a formula is defined by 
transfinite induction and $\sumh$ 
appears in the definition of the size
of infinitary disjunctions and conjunctions.
V\"a\"an\"anen and Wang's definition applies 
with no essential modification to get an ordinal  notion of ``size''
for a well-founded tree without infinite
branches and  such that each node has
at most countably many descendants.
See Example \ref{meglvw} for explicit details.  

While, so far, applications 
of infinitary ``natural''   operations 
 can be counted on the fingers of one hand, 
it should be mentioned that the study of such infinite operations 
has sometimes
shed new 
light to the more classical ordinal arithmetics, e.~g.,
\cite{A}, \cite[Remark 3.2 and Section 5]{w},
\cite[p. 524]{t},
or has led to new problems, in particular,  
the problem of generalizing to the surreal numbers some 
infinitary operations on ordinals
\cite[p. 256]{w}, 
 \cite[Section 5.4]{t}.  See also Problem \ref{probgam} here
and subsequent comments. 
As another example, in \cite{infnp} some possibly new techniques
are developed for transferring results about sums
to results about products. In this connection, but limited to 
finitary operations, see also \cite{HP}.   

In the special case of $ \omega$ arguments,
the operation $\sumh $ is defined by taking the supremum
of the partial Hessenberg natural sums.
In this special case, $\sumh $
 shares some good associativity
properties, but  generally fails badly, as far as strict monotonicity
is concerned \cite[p. 251]{w}. In fact, it is elementary to see that
there is no everywhere defined infinitary operation on ordinals which
is strictly monotone on each argument. See Remark \ref{no} below.
Note that the countable case of $\sumh$ has been denoted by $\nsum$
 in \cite{w}.

\subsection{A nearly strictly monotone operation} \labbel{nearl}
 
In this paper we introduce an $ \omega$-ary
operation $\sums $ which is  strictly monotone in 
 all significant cases, and is the smallest
such operation (Theorem \ref{simpth}).
Recall that, as we mentioned, there is
no  infinitary operation on ordinals which is strictly monotone
on each argument.
The motivations for introducing $\sums $
are the same as  Hessenberg's.
Indeed, Hessenberg called his sum   ``natural''
because, besides being 
commutative and
monotone, it is cancellative\footnote{\emph{\dots jeder 
der beiden Summanden durch den anderen eindeutig bestimmt ist, 
Eigenschaften, die die Bezeichnung dieser Summe als der ``nat\"urlichen''  rechtfertigen.} \cite[p. 105]{Hes}.}.
 Now note that a monotone 
operation on the ordinals
is cancellative if and only if it is strictly monotone.
Hence our operation $\sums$ is  the best
possible infinitary approximation of Hessenberg project.
 A proof for the existence of such an infinitary operation
is quite easy and is presented in Theorem \ref{thm}.6.a.
The main theme of this paper is
 to describe explicitly the smallest such operation.
See Definition \ref{simpldef} and Theorem \ref{simpth}.
We briefly hint to some differences
between $\sums$  and 
$\sumh$ in Remark \ref{meglvw}.  
Some auxiliary variants of  
$\sums$ are also introduced, denoted by different
superscripts: see Definition \ref{s1s2} and Remark \ref{snote}.

We provide further arguments
suggesting the naturalness of $\sums$. This operation
has two  order-theoretical characterizations, described in Sections \ref{order}
and \ref{rank}
 and a  
 game-theoretical characterization,
presented in Section \ref{game} (Appendix I). 
Some additional remarks are contained in Section \ref{fur}.   
The order-theoretical characterization of $\sums$ presented in Section \ref{order}
fits well with other similar characterizations of ``natural''
operations. As well-known \cite{Car,neumer}, the
Hessenberg natural sum of a finite set of ordinals is the largest
order-type which can be realized as the disjoint union of 
the summands, preserving the order.
 Similar results hold for infinitary ``natural''
operations \cite{w,t,infnp},
subject to some regularity conditions on the
union.   The condition appropriate for $\sums$ will be described in
Definition \ref{ormf} and called \emph{order respecting modulo finite},
for short, \emph{ormf}.  

As far as the characterization in 
Section \ref{rank} is concerned, 
the fact that the Hessenberg natural sum $\+$ 
is the smallest strictly monotone operation on the ordinals
means exactly that $\+$ has the following inductive definition: 
\begin{equation}\labbel{supd} 
\alpha \+ \beta =
\sup  \{ S( \alpha' \+ \beta) ,  S(  \alpha \+ \beta' ) \mid
\alpha' <  \alpha, \beta' < \beta  \},
  \end{equation}    
where $S$ denotes the successor function.
In turn, \eqref{supd}  means that 
$\alpha \+ \beta$ is the rank of the pair
$(\alpha, \beta)$  in the well founded ordered class 
${\On} \times {\On}$, with componentwise order.
Similarly, we show in Theorem \ref{thmgamr} 
that $\sums _{i < \omega} \alpha_{i}$ 
is equal to the rank of the sequence $( \alpha_i) _{i < \omega} $ 
in an appropriate well-founded order.

 In another direction, and at first glance, the game-theoretical characterization
of $\sums$ 
could appear to be an essentially new phenomenon.
However, it should be mentioned that the standard 
game-theoretical operations on the surreals, when restricted
to the ordinals, are the Hessenberg natural sum and product, so that, again,
our game-theoretical characterization fits well with former results
and methods. 
In the above context, a natural problem arises: the problem of generalizing
$\sums$  to the class of surreal numbers.

Finally, let us mention that there  exists another operation $\sumn$ 
which, in a sense,  retains as much strict monotonicity as possible.
See Theorem \ref{thm}.6.a (A). 
 However, $\sumn$ seems to behave
in a rather weird fashion; see Remark \ref{snote}(c).

In the following subsections we present some more details about the above-hinted
notions.

\subsection{A primer about finitary ordinal operations} \labbel{primer} 

Transfinite ordinals take their name from the fact
that they are a class of indexes which can be used
for the purpose of  carrying over
iterated transfinite constructions.
This fact immediately entails a very natural 
and by now standard definition 
for the sum $+$ of two ordinals: intuitively, $\alpha + \beta $
is the index set appropriate for an iteration performed
 $\alpha$ times followed by an iteration
performed $\beta$ times.
From an order-theoretical point of view,
$\alpha + \beta $ is the order-type
of a copy of $\beta$ juxtaposed after a copy of $\alpha$.
Of course, the ordinal sum has also a classical and
well-known inductive
definition \cite{Bac}  and, when restricted to natural numbers, 
it coincides with the usual sum.

The above-described ordinal sum $+$ is associative, right-cancellative,
right-distri\-butive and weakly monotone on both arguments.
While the above comments suggest that this ``classical'' ordinal sum $+$ is a
very ``natural'' operation, it does not share all desirable
algebraic properties: it
is neither commutative nor left cancellative,
nor strictly monotone on the left.
To remedy the above shortcomings,
Hessenberg \cite[Ch. XX]{Hes} introduced another
ordinal operation $\+$ that he called ``natural''
(``nat\"urliche''). Hessenberg natural sum 
can be computed by expressing summands in Cantor
normal form and then summing ``linearly'';
see Definition \ref{natsumdef} below for 
explicit details. 
Alternative characterizations are recalled in
Proposition \ref{minnat}, some subsequent comments and in Theorem \ref{carr}.  
 See \cite{Bac,Sier} and our  review \cite{Lrev} of \cite{A}
for further details and historical remarks.

Carruth \cite{Car} characterized
the Hessenberg natural sum (denoted by $\sigma( \alpha, \beta )$ in \cite{Car})
 as the smallest  binary operation on the ordinals
which is commutative, associative, strictly monotone and for which
$0$ is a neutral element. Most of the above assumptions are redundant,
since the Hessenberg natural sum is the smallest binary
 operation on the ordinals which is 
 strictly monotone on both arguments, as shown by Toulmin \cite{T}.
We are going to 
give a proof of this result in Proposition \ref{minnat},
since it will be relevant for our discussion later.
See \cite{HP} for other axiomatizations of the Hessenberg natural sum. 

 Of course,  smallest  is intended in the following
sense: an operation $\oplus $ is \emph{smaller} than another operation 
$\oplus'$ if $ \alpha \oplus \beta \leq \alpha \oplus' \beta  $,
for every $\alpha$ and $\beta$. The above definition makes sense
for every operation defined on a partially ordered class, but,
 except for some examples and problems, 
 here we will consider only operations defined
on the class of the ordinals (hence \emph{strictly monotone} necessarily means
\emph{strictly increasing}, so there is no ambiguity. 
We will always mean \emph{monotone} in the sense of
\emph{nondecreasing}, in other words, increasing,  not necessarily 
strictly).

As we mentioned,  Hessenberg sum
is indeed commutative and, moreover, cancellative;
this
is the reason why Hessenberg called it 
``natural''. 
Moreover, Hessenberg sum is strictly monotone on both sides.
The main shortcoming of the Hessenberg sum seems to be that
it is not continuous at limits, for example
$ \sup _{ n< \omega} (1\+n) = \omega  $,
while 
  $ 1\+\sup _{ n< \omega} n = \omega +1 $.
On the other hand, right-continuity of the classical sum $+$
seems to be a fundamental aspect in order to correctly carry 
out transfinite constructions.
It should be mentioned, however, 
that the Hessenberg natural sum, being 
strictly monotone on both arguments, 
is sometimes useful for double 
transfinite induction.

Though, of course, it is a subjective matter to consider some notion
``natural'' or not, the standard terminology in the literature 
uses the expression ``natural sum'' for  the Hessenberg  sum.
In order to avoid any possible misunderstanding, here we will always
use the expression  \emph{Hessenberg natural sum} 
written out in full. By the way, there are also
``classical'' and ``natural'' products, we will recall 
in Definition \ref{hprod}.6.b.
As above, the classical sum is right but not left distributive
with respect to the classical product, while distributivity
always holds for the ``natural'' operations. 
A short list of further mathematical applications
of the natural operations,  
sometimes even outside logic, has been
presented in \cite{infnp}.

\subsection{On the ``right''  choice of some infinitary ordinal operation} \labbel{infin} 
As we have mentioned at the beginning, in the present work we are concerned 
with infinitary generalizations of the Hessenberg natural sum.
As we have briefly recalled, such infinitary generalizations  have been 
occasionally considered with some applications 
\cite{BM,Ch,VW}. 
Jacobsthal product \cite{J}, which is in fact  an infinite Hessenberg natural sum in disguise
(see \cite{A}),
has found some applications, as well \cite{AB,A}.
Note that some authors of the mentioned references were not aware of
the original papers and rediscovered the notions.

The classical ordinal sum itself can be iterated transfinitely and,
as intuitively clear from the first paragraph in the previous subsection, this
transfinite ordinal sum satisfies a very strong form
of infinitary associativity \cite[p. 51]{Bac}.
Similarly, V. A. Chatyrko \cite{Ch} introduced a transfinite
iteration $\sumh $ of the Hessenberg natural sum  by taking 
suprema at limit stages.
This operation has proved useful, but we believe that it does not 
really fit with Hessenberg program.
Considering continuity at limit stages is not in the spirit
of Hessenberg proposal; actually, already Hessenberg  
binary ``natural'' sum
 is generally not continuous at limits. 
In fact, as a result, and as we have already mentioned,
$\sumh $ fails badly, as far as strict monotonicity is
concerned.

Hence it seems to be some good news 
that there is an operation which is strict monotone in a large
number of cases. Though the present paper is not explicitly
devoted to the analysis of applications
of $\sums$, some possible applications are 
 hinted in Example \ref{meglvw}. 

Given the above discussions, the reader
will have no difficulty to convince herself
that there exists no infinitary operation
satisfying at the same time all the possible desirable properties.
See 
\cite[Ch. 5]{Wo} for more arguments.
See also, e.~g., \cite{A},  \cite[p. 370]{VW} 
and Remarks \ref{no},  \ref{perforza} and \ref{limitat}   here.
Thus the choice of some appropriate infinitary operation
must necessarily depend on the context and on the desired
applications. Of course, unless one chooses to adhere
to strict finitism, in which case probably\footnote{We say
``probably'' since, for example, the semiring  of ordinals $< \omega ^ \omega  $ 
with natural operations is isomorphic to the semiring of polynomials
over $ \mathbb N$. Hence, even if considering just a finite set of natural
numbers, one can consider some indeterminate $x$ as sharing
some properties of an ``infinite'' object.} the present paper will
 be completely useless! \smiley

\section{The finitary Hessenberg natural sum and some obstacles to infinitary generalizations} \labbel{fobs} 

 We will always consider the ordinal $0$
to be neither limit, nor successor, that is, the only one of his kind. 
Recall that every nonzero ordinal $\alpha$  can be expressed uniquely in
\emph{Cantor normal form} as
$\alpha  =  \omega ^ {\xi_h} r_h +  \dots+ \omega ^ {\xi_0}r_0$ 
with $\xi_h > \xi_{h-1}> \dots > \xi_0$ ordinals  
and $r_h, \dots, r_0$ strictly positive natural numbers.

\begin{definition} \labbel{natsumdef}   
Explicitly, the \emph{Hessenberg natural sum}
of two ordinals is defined as follows. 
If $\alpha  =
 \omega ^ {\xi_h} r_h + 
\dots
+ \omega ^ {\xi_0}r_0$  
and
$ \beta  =
 \omega ^ { \rho _k} s_k + 
\dots
+ \omega ^ { \rho _0}s_0$  are expressed in 
Cantor normal form,
the summands of $\alpha \+ \beta $
in normal form are  
  \begin{itemize}  
 \item  
$ \omega ^ {\xi_i} r_i $, for those $i$ such that  
 $\omega ^ {\xi_i}$ does not appear in the expression defining
$\beta$,
\item
$ \omega ^ { \rho _j} s_j $,
for those $j$ such that  
 $\omega ^ { \rho _j}$ does not appear in the expression defining
$\alpha$, and
\item
$ \omega ^ {\xi_i} (r_i + s_i) $
for those $i$ such that  
 $\omega ^ {\xi_i}$ does  appear in both the expressions defining
$\alpha$ and $\beta$.
 \end{itemize}
The above definition applies to the ordinal $0$, when
 considered as an empty sum.
Otherwise, set $0 \+ \beta  = \beta  $
and $ \alpha \+ 0 = \alpha $ explicitly.    
 \end{definition}

 As we mentioned in the 
introduction, Hessenberg natural sum is not continuous at limits.
However, the following partial form
of continuity will play an important role
in the present note.

\begin{remark} \labbel{lim}  
If, in the  notation from Definition \ref{natsumdef}, both $\alpha$ and $\beta$ 
are limit ordinals and $\xi_0 \leq \rho _0$, then 
$\alpha \+ \beta = \sup \{ \alpha' \+ \beta ,  \mid  
\alpha' < \alpha \} $. 
  \end{remark}

\begin{proposition} \labbel{minnat}
 The Hessenberg natural sum $\+$  is the smallest binary operation on the ordinals which is 
 strictly monotone on both arguments. 
 \end{proposition} 

 \begin{proof}
It is well-known that $\+$ is strictly monotone.

On the other hand, let $\oplus$ be any operation
on the ordinals which is strictly monotone. We will prove by induction
on the value of $\alpha \+ \beta $ that 
$ \alpha \+ \beta \leq \alpha \oplus \beta  $.

The base step $\alpha \+ \beta =0$ is straightforward, since 
$0$ is the smallest ordinal.

Suppose that  $\alpha \+ \beta $ is a successor ordinal, hence
either $\alpha$ or $\beta$ is a successor ordinal,
say $\alpha= S(\alpha ')$, where 
$S$ denotes \emph{successor}.
Since 
$\alpha' \+ \beta < S (\alpha' \+ \beta)= \alpha \+ \beta $,
we get
$\alpha' \+ \beta \leq \alpha' \oplus \beta < \alpha \oplus \beta$,
where the first inequality follows from the inductive
assumption and the second inequality follows from strict monotonicity
of $\oplus$. But then $\alpha \+ \beta =  S (\alpha' \+ \beta) \leq  \alpha \oplus \beta$.
The case when $\alpha$ is limit and $\beta$ is successor is treated
in a symmetric  way.

Suppose now that $\alpha \+ \beta $ is limit,
hence both $\alpha$ and $\beta$ are limit
(unless either $\alpha$ or $\beta$ is $0$, but in this
case the result is
straightforward).

Let  $\alpha  =
 \omega ^ {\xi_h} r_h + 
\dots
+ \omega ^ {\xi_0}r_0$  
and
$ \beta  =
 \omega ^ { \rho _k} s_k + 
\dots
+ \omega ^ { \rho _0}s_0$  
 be the expressions of $\alpha$ and $\beta$ 
in Cantor normal form.
Suppose 
that $\xi_0 \leq \rho_0 $,
the other case is treated in a symmetrical way.
By Remark \ref{lim},   
$\alpha \+ \beta = \sup \{ \alpha' \+ \beta ,  \mid  
\alpha' < \alpha \} $. 
By strict monotonicity of $\+$, we have that
$\alpha' \+ \beta  < \alpha \+ \beta $,
for  $\alpha' < \alpha $, 
hence we can apply the inductive hypothesis, getting
 $\alpha' \+ \beta \leq \alpha' \oplus \beta  $,
for  $\alpha' < \alpha $.
By the above equations and by  monotonicity of 
$\oplus$, we get 
$\alpha \+ \beta = 
\sup \{ \alpha' \+ \beta  \mid  
\alpha' < \alpha  \} 
\leq
\sup \{ \alpha' \oplus \beta  \mid  
\alpha' < \alpha  \} 
\leq
\alpha \oplus \beta$.

Note that in the limit case we have  used only weak monotonicity
of $\oplus$.
 \end{proof}

The existence of a smallest strictly monotone operation
on the ordinals can be proved directly. 
 If $\Gamma$  is a set of ordinals, 
let $\sup ^+ \Gamma $ be defined as 
$\sup \Gamma$,
if $\Gamma$ has no maximum,
and as 
$(\sup \Gamma) + 1$, if  $\Gamma$ has a maximum
(the definition is intended in the sense that $\sup^+ \emptyset =0$).
In other words, 
$\sup ^+ \Gamma $
is the smallest ordinal strictly larger than all the
elements of $\Gamma$.

Then strict monotonicity of some operation $\oplus$ on the ordinals 
implies
$\alpha \oplus \beta \geq
\sup ^+ \{ \alpha' \oplus \beta, \alpha \oplus \beta' \mid
\alpha' <  \alpha, \beta' < \beta  \} $.
But it is clear that the condition
\begin{equation}\labbel{sup} \tag{NS} 
\alpha \oplus \beta =
\sideset{}{^+}\sup  \{ \alpha' \oplus \beta, \alpha \oplus \beta' \mid
\alpha' <  \alpha, \beta' < \beta  \}
  \end{equation}    
gives a recursive definition 
of some operation $\oplus$,
which is therefore the smallest strictly monotone operation.
We can perform the induction in the definition 
\eqref{sup} by recursion on
$(\max \{ \alpha, \beta  \},\min \{ \alpha, \beta  \} )$,
ordered lexicographically.  
More elegantly, we can perform the recursion  on pairs
of ordinals, ordered componentwise, that is, 
$( \alpha, \beta ) \leq ( \alpha', \beta ')  $ 
if and only if $\alpha \leq \alpha '$ and $\beta \leq \beta '$.
We can also perform the induction this way since this ordering on pairs
is well-founded.
Put in another way, the Hessenberg natural sum of two ordinals
$\alpha$ and $\beta$ is the rank of the pair
$(\alpha, \beta)$ in the well-founded  partial order 
${\On} \times {\On}$.

In fact, it is known 
(and it also follows from the above arguments) that
equation \eqref{sup} 
gives an alternative definition 
of the Hessenberg natural sum; e.~g., Conway \cite{C};
this is also implicit from
 the proof of  de Jongh and  Parikh \cite[Theorem 3.4]{dJP}. 
The proof of Proposition \ref{minnat}
is then essentially a restructuring of the proof
of the equivalence of these two definitions of the Hessenberg natural sum. 

Since $\+$ turns out to be associative, there is no
difficulty in providing a generalization of 
Proposition \ref{minnat} holding for
 $n$-ary operations, with
$n$  finite.

On the other hand, the above ideas, as they stand, cannot be used in order to provide
an infinitary generalization of the Hessenberg natural sum, as we show in the following
 elementary remark.
  
Recall that we compare ordinal operations
in the following way: 
an operation $\oplus $ is \emph{smaller} than another operation 
$\oplus'$ if $ \alpha \oplus \beta \leq \alpha \oplus' \beta  $,
for every $\alpha$ and $\beta$. The above definition makes sense
for every operation defined on a partially ordered class,
but, except for some examples and problems,   
 here we will consider only operations defined
on the class of the ordinals (hence \emph{strictly monotone} necessarily means
\emph{strictly increasing}, so there is no ambiguity. 
We will always mean \emph{monotone} in the sense of
\emph{nondecreasing}, in other words, increasing,  not necessarily 
strictly).

\begin{remark} \labbel{no}    
 There is no strictly increasing (on each argument) and
everywhere defined infinitary operation on
 $ \omega$-indexed sequences of ordinals.
This is witnessed, for example, by the following sequences:
\begin{align*}      
&(2,2,2,2,2,\dots),
\\ 
&(1,2,2,2,2,\dots),
\\ 
&(1,1,2,2,2,\dots),
\\ 
&(1,1,1,2,2,\dots),
\dots
 \end{align*}
Any strictly increasing infinitary operation, when applied
to the above sequences, would produce a descending chain of ordinals.  

 Actually, the argument shows that there is 
no everywhere defined strictly increasing function
which associates  ordinals to infinite sequences
of elements taken from a partially ordered set
having at least two comparable elements.

(b)
As above, assume that we have an everywhere 
defined infinitary operation on
 $ \omega$-indexed sequences of ordinals.

If we impose the quite natural requirement that
a summand with value $0$ does not influence the infinitary sum, we get
that the sequences $(2,2,2,2,2,\dots)$   and $(0,2,2,2, \allowbreak 2,\dots)$  
should receive the same value, so that strict monotonicity is lost.
In passing, note also that if just weak monotonicity holds, 
then $(1,2,2,2,2,\dots)$  will receive the same value, too.

This argument has little to do with ordinals in particular:
it applies to every partially ordered set with an element
$0$ and  further comparable elements. 

We do not even need  the
full assumption that the infinitary operation
is everywhere defined:
it is enough to assume that it is defined for
just one constant sequence in which the same nonzero element ($2$
in the above example) repeats.

(c) Even more generally, assume that $P$ is a partially ordered
set and we have a partially defined operation $S$ 
on countable sequences of elements of $P$.

Assume further that $0$ is a \emph{neutral element} for  $S$.
This means that, for every sequence $\sigma$ with $0$ at the $i$th place,
$S( \sigma ) $ is defined if and only if 
$S( \tau  ) $ is defined, where $\tau$ is the subsequence of $\sigma$ 
obtained by removing the $0$ at the $i$th place
(so the summand at the $i+1$th place shifts to the $i$th place).
Moreover, we require that, if defined, both sequences are given
the same value.

Under the above assumptions,
if the sum of some sequence 
$(a_0, a_1, a_2, \allowbreak \dots)$ is defined
and  $0 < a_0 \leq a_1 \leq a_2 \leq \dots $, 
then $S$ cannot be
both monotone on each argument and strictly monotone
on the first argument.
Indeed, since $0$ is neutral,
 $(a_0, a_1, a_2, \dots)$ is evaluated in the same way as
$(0,a_0, a_1, a_2, \dots)$,
contradicting the required monotonicity 
properties\footnote{Note an asymmetry here: we may indeed
have strict monotonicity on each argument when
$a_0 \geq a_1 \geq a_2 \geq \dots $. Consider, for example,
real numbers with the classical notion of converging
series. Of course, series are not everywhere defined.}.

(d) In addition to the assumptions of (c),
assume further that the infinitary sum is \emph{invariant under permutations},
that is, the sum of the elements of some sequence is defined
if and only if the sum of any permutation of the elements
of the sequence is defined, and, if this is the case, the two
sums are equal.

Suppose that the sum of some sequence is defined
and  $0 < a_{i_0} \leq a_{i_1} \leq a_{i_2} \leq \dots $, 
for certain elements of the sequence, where $i_0, i_1, i_2, \dots$ 
are distinct indexes, appearing in any order.
Then the sum is not strictly monotone on each argument.
Argue as in (c), replacing $a_{i_0}$ with $0$,
replacing   $a_{i_1}$ with $a_{i_0}$ and so 
on\footnote{Assuming only that the sum $S$ is invariant under permutations,
if some sum is defined and two comparable elements, call them $1$ and $2$, 
both occur infinitely many times in a sequence whose sum is defined,
then $S$ is not strictly monotone. Just replace one occurrence of $1$ by $2$;
the resulting sum should be the same, since $S$ is invariant and $1$ and $2$ 
both occur infinitely many times.}.
   
\end{remark}

 See Remark \ref{perforza} below and
\cite[Ch. 5]{Wo} for further elaborations. 

We now return to the special case of
some everywhere defined infinitary sums
of sequences of ordinal numbers.
By the  arguments in Remark \ref{no},
if some infinitary sum is weakly monotone, invariant under permutations,
$0$ is a  neutral element, 
and some ordinal $\alpha$ appears infinitely many times in the sequence,
then all the occurrences of ordinals $< \alpha$ 
do not influence the outcome  of  the sum. 
Moreover,  if $\beta \leq \alpha _i$, for infinitely many $i < \omega$,
then an occurrence of 
$\beta$ does not influence the infinitary sum.  

In view of Remark \ref{no}, if we want to generalize the Hessenberg natural sum
to sequences of length $ \omega$, 
we have to relax the assumption of strict monotonicity\footnote{Of
course, there is also the possibility of working in some
class larger than the ordinals, but in the present work we will 
generally confine ourselves
to infinitary extensions of the Hessenberg natural
 sum which lie within the class of the ordinals.
Another possibility is to consider sums which are only partially
defined, that is, defined only for some special class
of sequences. The examples in Remark \ref{no}
provide negative results even in such more general settings. 
  We refer again to 
\cite{Wo} for a discussion of incompatibilities
for desirable properties of infinitary sums.},   
while, of course,
we would like to retain weak monotonicity.
The hope remains to maintain strict monotonicity
when limited to 
elements of the sequence 
which are  bounded above just by finitely many other
elements of the sequence. It turns out that this
can be actually accomplished and that the resulting
operation satisfies some good and natural properties.

\section{Some preliminary definitions and a lemma} \labbel{prel} 

 Unexplained notions and notations are from \cite{Bac,Je,w,Sier}.
From now on, operations are total, that is, everywhere defined.

\begin{remark} \labbel{nobis}    
Recall from Remark \ref{no}
that there is no strictly increasing (on each argument) and
everywhere defined infinitary operation on
 $ \omega$-indexed sequences of ordinals, as
 witnessed by the following sequences:
\begin{equation*}      
(2,2,2,2,2,\dots),\ \ 
(1,2,2,2,2,\dots),\ \ 
(1,1,2,2,2,\dots),\ \ 
(1,1,1,2,2,\dots),
\dots
 \end{equation*}
 \end{remark}

\begin{definition} \labbel{spec}   
Turning to formal definitions, we compare infinitary operations
on $ \omega$-indexed sequences of ordinals by saying that
$\osum$ is \emph{smaller} than $\osum'$ if 
$\osum _{i < \omega} \alpha_{i} \leq
\osum' _{i < \omega} \alpha_{i}$,
for every sequence $( \alpha_i) _{i < \omega} $ 
of ordinals.

If $\bm \alpha = ( \alpha_i) _{i < \omega} $ is a sequence of ordinals, 
we let 
$ \varepsilon _{ \bm  \alpha} $  be the least ordinal $\varepsilon$ 
such that $\{ i < \omega \mid \varepsilon \leq  \alpha _i\}$ is finite.
Thus there are only finitely many 
\emph{e-special elements} of the sequence which
are $ \geq  \varepsilon _{ \bm  \alpha} $, and, for every
 $\varepsilon' <   \varepsilon _{ \bm  \alpha}  $,
there are infinitely many elements of the sequence
which are  $\geq\varepsilon'$. 

It will be notationally convenient to introduce another invariant.
If $  \bm  \alpha = ( \alpha_i) _{i < \omega} $ is a sequence of ordinals, 
we let 
$ \zeta _{ \bm  \alpha} $  be the least ordinal $\zeta$ 
such that $\{ i < \omega \mid \zeta <  \alpha _i\}$ is finite.
Here the definition involves a strict inequality.
The summands $ \alpha _i$
such that $ \zeta _{ \bm  \alpha}  <  \alpha _i$ will be called   
\emph{z-special 
elements} (relative to the sequence under consideration).
\end{definition}

\begin{remark} \labbel{determ}
 Let $\bm \alpha $ be a sequence of ordinals
and $\varepsilon_ {\bm \alpha} $, $\zeta_ {\bm \alpha} $ be
defined as above. 

(a) By definition, $\varepsilon_ {\bm \alpha}  > 0$, for every
infinite sequence of ordinals. 
    
(b) Note that $\zeta_ {\bm \alpha} $ is determined by $\varepsilon_ {\bm \alpha} $. Indeed, if 
$\varepsilon_ {\bm \alpha} $ is a successor ordinal, then 
$\zeta_ {\bm \alpha} $ is the predecessor 
of $ \varepsilon_ {\bm \alpha} $
(and there are infinitely many $ i < \omega$ such that
$\alpha_i = \zeta_ {\bm \alpha} $). 

(c) On the other hand, if $\varepsilon_ {\bm \alpha} $ is a limit ordinal, then 
$\zeta_ {\bm \alpha} =\varepsilon_ {\bm \alpha} $.

(d) In particular $\zeta_ {\bm \alpha}  \leq \varepsilon_ {\bm \alpha} $ always,
and every z-special element is e-special.

(e) In detail, if
$\varepsilon_ {\bm \alpha} $ is successor, then
the sets of z-special and e-special 
elements coincide.

(f)  
If $\varepsilon_ {\bm \alpha} $ is limit, then
the e-special not z-special  elements
 of the sequence are those $\alpha_i$ 
which are exactly $\varepsilon_ {\bm \alpha} $ (possibly, there is
no such element; in any case, the set of such elements is finite). 

(g) Note also that if $\zeta_ {\bm \alpha} $ is limit,
then there might be possibly infinitely many
$\alpha_i$ equal to $\zeta_ {\bm \alpha} $ (in this case
$ \varepsilon _{\bm \alpha }=\zeta_{\bm \alpha }+1$), but, possibly,
there might be only finitely many, possibly, none,
such elements (and in this case
$ \varepsilon _{\bm \alpha }=\zeta_{\bm \alpha }$).

(h) In particular, if  $\zeta_ {\bm \alpha} $ is limit, the value
of $ \varepsilon _ {\bm \alpha} $ is not determined by $\zeta_ {\bm \alpha} $.
On the other hand, if $\zeta_ {\bm \alpha} $ is successor,
then necessarily $ \varepsilon _{\bm \alpha }=\zeta_{\bm \alpha }+1$.
\end{remark}

Recall that we assume that operations are everywhere defined.  

\begin{definition} \labbel{defop}    
We say  that some 
infinitary operation $\osum$
on $ \omega$-indexed sequences of ordinals is
\emph{(weakly) monotone} if,
for all sequences  $\bm \alpha = ( \alpha_i) _{i < \omega} $,
$\bm \beta  = ( \beta _i) _{i < \omega} $ 
of ordinals, 
  \begin{equation}
\labbel{m}
\text{
 $\osum _{i < \omega } \alpha_i \leq \osum_{i < \omega } \beta _i  $,
   whenever 
 $\alpha_i \leq \beta_i $ for every $  i < \omega$.} 
\end{equation}   

We say that $\osum$
is
\emph{strictly monotone on  e-special elements}      
 if, for all $\bm \alpha $,
$\bm \beta   $,
\begin{equation} \labbel{s2}  \tag{2.2e}
\begin{aligned}   
&\text{$\osum _{i < \omega }\alpha_i \allowbreak <
\allowbreak \osum_{i<\omega} \beta_i  $,
 whenever 
$\alpha_i \leq \beta_i $, for every $  i < \omega$,
and there}
\\
&\text{is at least one $ \bar{\imath} \in \omega$
such that  
$ \varepsilon_ {\bm \alpha}  \leq  \beta _{\bar{\imath}}  $
and $ \alpha _{\bar{\imath}}  
<  \beta _{\bar{\imath}} $,} 
\end{aligned}
\end{equation}

equivalently, assuming monotonicity (\ref{m}),
\begin{equation} \labbel{s2'}    \tag{2.3e}
\begin{aligned} 
&\text{$\osum _{i < \omega } \alpha_i <  \osum_{i < \omega } \beta _i  $, 
whenever}  
\\
&\text{$(\alpha_i) _{i < \omega } $ and
$( \beta _i) _{i < \omega} $ 
are sequences of ordinals
which differ}  
\\
&\text{only at  the $\bar{\imath} ^{th}$-place,
$ \varepsilon_ {\bm \alpha}  \leq   
\beta _{\bar{\imath}} $
and $\alpha _{\bar{\imath}}  
<  \beta _{\bar{\imath}} $.} 
 \end{aligned}
\end{equation} 

Moreover, we say that $\osum$ is
\emph{strictly monotone on  z-special elements}   if,
 for all $\bm \alpha $,
$\bm \beta   $,
\begin{equation}\labbel{s3} 
\begin{aligned}
&\text{$\osum _{i < \omega } \alpha_i 
\allowbreak  < \allowbreak  \osum_{i < \omega } \beta _i  $,
 whenever $\alpha_i \leq \beta_i $, for every $  i < \omega$,
and there} 
\\
&\text{is at least one $ \bar{\imath} \in \omega$
such that  
$ \zeta_ {\bm \alpha}  <  \beta _{\bar{\imath}}  $
and $ \alpha _{\bar{\imath}}  
<  \beta _{\bar{\imath}} $,}
 \end{aligned}
\end{equation} 

equivalently, assuming monotonicity (\ref{m}),
\begin{equation} \labbel{s3'}    
\begin{aligned}
&\text{$\osum _{i < \omega } \alpha_i <  \osum_{i < \omega } \beta _i  $, 
whenever $(\alpha_i) _{i < \omega } $ and
$( \beta _i) _{i < \omega} $ are sequences of}
\\
&\text{ordinals
which differ only at the $\bar{\imath} ^{th}$-place and
$ \zeta_ {\bm \alpha} \leq 
\alpha _{\bar{\imath}}   <  
\beta _{\bar{\imath}} $}.
 \end{aligned}\end{equation} 
\end{definition}

\begin{remark} \labbel{primo}
(a) Assuming \eqref{m}, we have that
\eqref{s2}  and \eqref{s2'} 
 above are indeed  equivalent.
First, \eqref{s2} plainly implies  \eqref{s2'}. On the other hand,
under the assumptions in \eqref{s2}, pick a single 
$ \bar{\imath}$ satisfying the assumption
and let
  $\gamma _{\bar{\imath}} =  
 \beta   _{\bar{\imath}}  $
and
 $\gamma_i = \alpha  _i$, if $ i \not=\bar{\imath}$. 
Then 
$\osum _{i < \omega } \alpha_i 
< ^{\eqref{s2'}}
  \osum_{i < \omega } \gamma  _i  
\leq ^{ \text{\eqref{m}}}
  \osum_{i < \omega } \beta  _i  
$, where the superscript in $  < ^{\eqref{s2'}} $
means that we are applying Clause \eqref{s2'} etc. 

(b) An  analogous remark applies to 
\eqref{s3} and \eqref{s3'}.
Just observe that, assuming  
\eqref{m}, 
the conditions 

\ \ \ (i) $ \zeta_ {\bm \alpha}  <  \beta _{\bar{\imath}}  $
and $ \alpha _{\bar{\imath}}  
<  \beta _{\bar{\imath}} $, 

\ \ \  (ii) $ \zeta_ {\bm \alpha} \leq 
\alpha _{\bar{\imath}}   <  
\beta _{\bar{\imath}} $ 

\noindent are interchangeable both in \eqref{s3} and \eqref{s3'}.
Indeed, (ii)  implies (i), so that if, say, 
\eqref{s3} holds under the assumption (i),
then \eqref{s3} holds under the assumption (ii). On the other hand, 
if (i) holds, set
$  \alpha  ^* _{\bar{\imath}} =   \max \{ \zeta_ {\bm \alpha} ,
  \alpha  _{\bar{\imath}}  \}$ and 
 $ \alpha  ^*_i = \alpha  _i$, if $ i \not=\bar{\imath}$. 
Then (ii) holds with respect to  $ \alpha ^*_i$;
moreover,  $\osum _{i < \omega } \alpha  _i 
\leq  \osum_{i < \omega } \alpha ^* _i  $
by \eqref{m}. If
\eqref{s3} holds under the assumption (ii),
then $\osum _{i < \omega } \alpha^*  _i 
<  \osum_{i < \omega } \beta  _i  $, hence also
$\osum _{i < \omega } \alpha  _i 
<  \osum_{i < \omega } \beta  _i  $, so that 
 \eqref{s3} holds under the assumption (i).

Now we get that \eqref{s3} and \eqref{s3'}
are equivalent arguing as in (a).

(b*) There is a subtle difference between clauses 
\eqref{s2} and \eqref{s3}.
By  Remark \ref{determ}(b), 
if $\varepsilon_ {\bm \alpha} $ is a successor ordinal, then
$\varepsilon_ {\bm \alpha}  = \zeta_ {\bm \alpha}  + 1$, hence in this case 
$ \varepsilon _ {\bm \alpha}    \leq  \beta _{\bar{\imath}}  $
if and only if 
 $ \zeta_ {\bm \alpha}   <  \beta _{\bar{\imath}}  $,
thus if $\varepsilon_ {\bm \alpha} $ is a successor ordinal,
then \eqref{s2} and \eqref{s3} give the same condition.
On the other hand, we will see in Remark \ref{snote} that, in general, 
\eqref{s2} and \eqref{s3} are not equivalent.  

(c) However, \eqref{s2} does imply \eqref{s3};
indeed, if $\varepsilon_ {\bm \alpha} $ is a limit ordinal,
then $\zeta_ {\bm \alpha}  =\varepsilon_ {\bm \alpha} $,
by  Remark \ref{determ}(c). 
If the assumptions in \eqref{s3} hold, namely, 
$ \zeta_ {\bm \alpha}  <  \beta _{\bar{\imath}}  $
and $ \alpha _{\bar{\imath}}  
<  \beta _{\bar{\imath}} $, then 
$    \varepsilon_ {\bm \alpha} = \zeta_ {\bm \alpha} 
 <  \beta _{\bar{\imath}}  $, hence we can apply
\eqref{s2} in order to get 
$\osum _{i < \omega } \alpha_i <  \osum_{i < \omega } \beta _i  $.
\end{remark}

\begin{remark} \labbel{e}
If some infinitary operation $\osum$
satisfies \eqref{s2}, then  
$\osum$ is \emph{strictly monotone on the characteristic
$\varepsilon$},
namely, for every pair $\bm \alpha = ( \alpha_i) _{i < \omega} $,
$\bm \beta  = ( \beta _i) _{i < \omega} $ 
of sequences of ordinals, 
if $\varepsilon_ {\bm \alpha } < \varepsilon_ {\bm \beta } $ 
and 
$\alpha_i \leq \beta _i$, for every $i < \omega$,
then   
$\osum _{i < \omega } \alpha_i <  \osum_{i < \omega } \beta _i  $.
 
Indeed, by the definition of $\varepsilon_ {\bm \beta }$,
for every $\varepsilon' < \varepsilon_ {\bm \beta }$,
 there are infinitely many $i < \omega$ such that 
$ \varepsilon ' \leq \beta_i$,
in particular, there 
are infinitely many $i < \omega$ such that 
$ \varepsilon_ {\bm \alpha  } \leq \beta_i$,
since $\varepsilon_ {\bm \alpha  }< \varepsilon_ {\bm \beta }$,
by assumption. On the other hand,   
by the definition of $\varepsilon_ {\bm \alpha }$,
there are only finitely many $\alpha_i$
which are $\geq\varepsilon_ {\bm \alpha  }$.
Thus there is at least one ${\bar{\imath}} < \omega$
(in fact, infinitely many ${\bar{\imath}}$) such that 
$\alpha_{\bar{\imath}}< \beta _{\bar{\imath}}$ and   
$ \varepsilon_ {\bm \alpha  } \leq \beta_{\bar{\imath}}$.
By \eqref{s2}, $\osum _{i < \omega } \alpha_i <  \osum_{i < \omega } \beta _i  $.

Similarly,
 if some infinitary operation $\osum$
satisfies \eqref{s3}, then  
$\osum$ is \emph{strictly monotone on $\zeta$}.
\end{remark}   

As we are going to see in the proof of the next theorem, 
it is not difficult to show  that there is
at least 
one (and hence also the smallest) monotone $ \omega$-ary operation on ordinals
which is both weakly monotone
and
strictly monotone on  e-special elements
(hence also strictly monotone on z-special elements,
by Remark \ref{primo}(c)).

\begin{thm26a} \labbel{thm}
  \begin{enumerate}[(A)] 
   \item   
 There is the smallest $ \omega$-ary operation  $\sumn$ 
on  ordinals among
those operations $\osum$ which are both weakly monotone
and
strictly monotone on  e-special elements.

Namely, there is the smallest operation satisfying Clauses 
\eqref{m} and \eqref{s2} in Definition \ref{defop}.

\item
 There is the smallest operation  $\sums$ 
on $ \omega$-indexed sequences of ordinals among
those operations $\osum$ which are both weakly monotone
and
strictly monotone on  z-special elements.
That is,  there is the smallest operation satisfying Clauses 
\eqref{m} and \eqref{s3} in Definition \ref{defop}.
\end{enumerate}

Both $\sumn$ and 
$\sums$ are invariant under permutations.  
\end{thm26a}

Before we go on, some remarks are in order.  

We have stated Theorem \ref{thm}.6.a
in the above form for simplicity, but the statements could cause some foundational
issues, since we are dealing with operations which are proper classes.
However, the issues can be solved in a number of standard ways,
for example, working
in some set theory in which quantification over classes is allowed,
or restricting the domain of all the operations 
under consideration to sequences whose members are taken 
from a sufficiently large
\emph{set} of ordinals. Even if we are 
supposed to be working in a ``pure'' 
theory of sets, we point out that we are 
eventually going to provide an
explicit definition of 
$\sums$, hence the corresponding statement in Theorem \ref{thm}.6.a
can be reformulated as theorem schemata---one schema for 
each operation $\osum$ satisfying \eqref{m} and \eqref{s3}---and 
asserting that $\sums$ is smaller than 
such a  $\osum$.

We first present a simple but indirect proof of 
Theorem \ref{thm}.6.a. The operation $\suml$ introduced in the following
proof is extremely much larger than necessary, 
in particular, extremely much larger  than the minimal 
operations given by Theorem \ref{thm}.6.a. One of
these operations
will be explicitly described in Definition \ref{simpldef}.
Moreover, $\suml$ appears to be a really weird operation;
however, it 
greatly simplifies the proof.

\begin{proof}
If 
$  \bm  \alpha = ( \alpha_i) _{i < \omega} $ is a sequence of ordinals, let
\begin{equation*} 
\suml_{i < \omega}  \alpha_i = \aleph_{  \varepsilon  _{ \bm  \alpha}} \+
\alpha _{i_1} \+ \dots \+ 
\alpha _{i_h}
 \end{equation*}     
where $\{ \alpha _{i_1}, \dots,
\alpha _{i_h}\}$  is the (possibly empty, finite by definition)
set of e-special elements of $\bm  \alpha$.
If there is no e-special element we simply put 
$\suml_{i < \omega}  \alpha_i = \aleph_{ \varepsilon  _{ \bm  \alpha} }$.

We now show that 
$\suml$ is
(weakly) monotone.
Suppose that  $\bm  \beta $ is larger than $\bm  \alpha$.
If $  \varepsilon  _{ \bm  \alpha} =   \varepsilon  _{ \bm  \beta } $,
then  
$\suml \bm  \alpha \leq \suml \bm  \beta $
by monotonicity of the Hessenberg natural sum.

Otherwise, 
$ \varepsilon_ {\bm \alpha} <  \varepsilon_ {\bm \beta} $.
In this case let 
 $\alpha_{j_0}, \dots $
enumerate the set of those elements of
$\bm  \alpha$ such that 
$ \varepsilon _ {\bm \alpha}  \leq \alpha_{j_0} < \varepsilon _ {\bm \beta} $
and let 
 $\alpha_{k_0}, \dots $
enumerate the set of those elements such that 
$ \varepsilon _ {\bm \beta}  \leq \alpha_{k_0}  $
(again, one or possibly both sets might be empty).
In particular, 
 $\alpha_{j_0}, \dots, \alpha_{k_0}, \dots  $
enumerates the set of the e-special elements with respect to  
$\bm  \alpha$. Hence
\begin{equation*}   
\begin{aligned}   
\suml _{i < \omega} \alpha_{i} 
=
\aleph_{ \varepsilon_ {\bm \alpha}  } \+
  \alpha_{j_0} \+ \dots \+  \alpha  _{j_h} \+  
\alpha  _{k_0} \+ \dots \+ \alpha  _{k_\ell}
<
\aleph_{ \varepsilon_ {\bm \beta}  } \+
\beta   _{k_0} \+ \dots \+ \beta   _{k_\ell} \leq 
\suml _{i < \omega} \beta _{i}
\end{aligned}     
  \end{equation*}
since the Hessenberg natural sum of a finite sequence of ordinals
$< \aleph_{ \varepsilon_ {\bm \beta}  }$ is still
$< \aleph_{ \varepsilon_ {\bm \beta}  }$; moreover, 
 $\alpha  _{k_0} \leq \beta   _{k_0}$, \dots, by assumption.
There might be further elements $\beta_j$ which contribute to the
sum   $\suml _{i < \omega} \beta _{i}$; anyway the last inequality holds,
 since $\beta   _{k_0}$ does contribute to  $\suml _{i < \omega} \beta _{i}$,
because $ \varepsilon _ {\bm \beta}  \leq \alpha_{k_0}  \leq \beta _{k_0}  $
and similarly for $\beta _{k_1}  $, \dots 

By strict monotonicity of the Hessenberg natural sum, $\suml$ is
strictly monotone on e-special elements, using \eqref{s2'}. Note that,
under the assumptions in \eqref{s2'}, 
it might happen that $a_{ \bar{\imath}}$  
  is not e special but $b_{ \bar{\imath}}$ is
e-special.
As we remarked
in Remark \ref{determ}(a),  $\varepsilon_ {\bm \beta} >0$,
hence, in the above situation,    $b_{ \bar{\imath}} 
\geq \varepsilon_ {\bm \beta} > 0$, so that  $b_{ \bar{\imath}}$
does contribute to a strict increment of the value of  $\suml$.
Since we have showed that $\suml$ is strictly monotone on e-special elements,
it is also strictly monotone on z-special elements,
by Remark \ref{primo}(c). 

We have proved that 
there is at least one operation which is strictly monotone on e-special elements
and strictly monotone on z-special elements.
This is enough to prove the theorem.
Indeed, under the usual foundational caution,
once we can prove that there exists at least one operation satisfying,
say, Clauses  \eqref{m} and \eqref{s2} in Definition  \ref{defop}, 
we get that there is necessarily 
the  smallest 
operation satisfying the clauses.
Indeed, if $ \{ \osum^f  \mid f \in F \}  $
is a nonempty family of operations satisfying \eqref{m} and \eqref{s2},
then, for every sequence $( \alpha_i) _{i < \omega} $,  we can define
$ \osum _{i < \omega}  \alpha_i  =
 \inf \{\osum ^f_{i < \omega}  \alpha_i   \mid f \in F \} $
and $\osum$ satisfies \eqref{m} and \eqref{s2}, too
(the strict inequality is preserved, since there is no
infinite decreasing sequence of ordinals).

Since Clauses 
\eqref{m} - \eqref{s3} do not depend on the ordering of the sequence,
by symmetry, the operations, being the minimal ones, 
are invariant under permutations. 
\end{proof}

\begin{def26b} \labbel{hprod}   
We will also use a ``natural'' product $\x$. It is frequently named
after Hessenberg, but many authors, e.~g. \cite{Car}, attribute it to 
Hausdorff.   In order to avoid any possible misattribution,
we will call it the \emph{H-natural product}. 
Recall that the H-natural product on ordinals can be defined 
by setting $ \omega ^{ \xi_i} \x \omega ^{ \rho _j} = \omega ^{ \xi_i \+  \rho _j} $
and extending the product on normal forms by distributivity.
Of course, we define  the product to be  $0$ if some factor is $0$.    
We will need the following lemma. 
 \end{def26b}

\begin{lemma} \labbel{lemdelkatze}
If  $  \varepsilon $, $  \delta $
are ordinals and $n< \omega$,
then $   (\varepsilon \x \omega ) \+ ( \delta  \x n) \leq
( \varepsilon + \delta ) \x \omega $.  
 \end{lemma}

 \begin{proof}
The conclusion is straightforward  if either 
$\varepsilon=0$ or $\delta=0$.

 Otherwise,  
let  $ \varepsilon   =
 \omega ^ {\xi_h} r_h + 
\dots
+ \omega ^ {\xi_0}r_0$  
and
$ \delta  =
 \omega ^ { \rho _k} s_k + 
\dots
+ \omega ^ { \rho _0}s_0$  
 be expressed
in Cantor normal form.
First suppose that 
$ \xi _\ell \geq \rho_k > \xi _{ \ell -1} $
for some $\ell$ with $ h \geq \ell \geq 1$.
Then 
$\varepsilon' = \varepsilon + \delta =
 \omega ^ {\xi_h} r_h + 
\dots
+ \omega ^ {\xi_\ell}r_\ell +
 \omega ^ { \rho _k} s_k + 
\dots
+ \omega ^ { \rho _0}s_0
$ and
$ \omega ^ { \rho _k} \geq \omega ^ {\xi_{\ell-1}+1}$; moreover,
$ \omega ^ { \rho _k} > \omega ^ { \rho _{k-1}} s _{k-1}n
 \+ \dots \+  \omega ^ { \rho _{0}} s_0 n 
\+  \omega ^ {\xi_{\ell-2}+1} r _{\ell-2}  \+ \dots \+  \omega ^ {\xi_0+1}r_0$.
Let $A$ denote the expression on the right
in the latter inequality; note
that $A$  might be an empty sum.
 Hence 
\begin{align*}
 (\varepsilon \x \omega ) \+ ( \delta  \x n)   
&=
  \omega ^ {\xi_h+1} r_h \+ 
\dots
\+ \omega ^ {\xi_0+1}r_0
\+ \omega ^ { \rho _k} s_k n \+ 
\dots
\+ \omega ^ { \rho _0}s_0 n 
\\
  &=  \omega ^ {\xi_h+1} r_h \+ 
\dots
\+ \omega ^ {\xi_\ell+1}r_\ell 
\+ \omega ^ { \rho _k} s_k n  
\+  \omega ^ {\xi_{\ell-1}+1}r_{\ell-1} \+ A 
\\
  & <\omega ^ {\xi_h+1} r_h \+ 
\dots
\+ \omega ^ {\xi_\ell+1}r_\ell 
\+ \omega ^ { \rho _k} s_k n  
\+  \omega ^ { \rho _k}r_{\ell-1} \+ 
  \omega ^ { \rho _k}
\\
 &= \omega ^ {\xi_h+1} r_h \+ 
\dots
\+ \omega ^ {\xi_\ell+1}r_\ell 
\+ \omega ^ { \rho _k} (s_k n +  r_{\ell-1} + 1)
\\
  &< \omega ^ {\xi_h+1} r_h \+ 
\dots
\+ \omega ^ {\xi_\ell+1}r_\ell 
\+ \omega ^ { \rho _k} \omega 
\\
  &\leq   \omega ^ {\xi_h+1} r_h \+ 
\dots
\+ \omega ^ {\xi_\ell+1}r_\ell 
\+ \omega ^ { \rho _k +1} s_k 
\+ \dots \+ 
\omega ^ { \rho _{0} +1} s_{0}
=
\varepsilon ' \x \omega, 
  \end{align*}
since $s_k >0$. 

The remaining cases are much easier.
If $ \rho_k > \xi _h $, then
$ \varepsilon ' =\varepsilon + \delta = \delta $,
$ \rho_k \geq \xi _h +1$, hence
$\varepsilon \x \omega  < \delta (r_h + 1)
\leq  \delta \x (r_h + 1)$. 
If $ \xi _0 \geq \rho_k $,
then
$ \varepsilon ' = \varepsilon + \delta = \varepsilon \+ \delta  $.
In both cases the conclusion  is straightforward. 
 \end{proof}    

Note that if $\varepsilon ' = \varepsilon + \delta $,
then it is not necessarily the case that 
$   (\varepsilon \x n ) \+ ( \delta  \x \omega ) \leq
\varepsilon' \x \omega $.  
For example, take $\varepsilon=1$
and $\delta =  \omega $,
then
$ (\varepsilon  \x n) \+ ( \delta  \x \omega) =
\omega^2 + n >
\omega ^2 =
\varepsilon' \x \omega $,
for $n>0$.    A fortiori, it is
not necessarily the case that 
$   (\varepsilon \x  \omega ) \+ ( \delta  \x \omega ) \leq
\varepsilon' \x \omega $. 
Note that if $\delta>0$,
then the proof shows that the inequality in Lemma \ref{lemdelkatze}
is strict (we will not need this in what follows).

\section{A strictly monotone operation} \labbel{simp}

Rather than the Cantor normal form, 
in what follows it will be convenient to use the 
(additive) normal form. Simply, the 
additive normal form is obtained from the Cantor normal form
by factoring out the integer coefficients and, conversely,
from the additive normal form we obtain the Cantor normal form
by grouping together all summands corresponding to the same
power of $ \omega$. 
In detail, every nonzero ordinal 
$ \zeta $ 
can be written uniquely
 in \emph{normal form} as $ \omega ^ {\xi_s}  + 
\dots
+ \omega ^ {\xi_1}
+ \omega ^ {\xi_0}$, with
$\xi_s \geq \xi _{s-1} \geq \dots \geq \xi_0  $. 

If $\zeta >0$ is written in normal form as above, 
we let $\hat \zeta =  \omega ^ {\xi_s}  + 
\dots
+ \omega ^ {\xi_1}$, namely, we remove
the last summand.
The convention is intended in the sense that if 
$\zeta=\omega ^ {\xi_0}$ has only one summand in 
normal form, then $\hat \zeta = 0$. 
In the next definition we will use $\hat \zeta $ only in the case
when $\zeta$ is a limit ordinal, thus $\xi_0 >0$;
actually, in most cases already
 $\xi_0 $ will be limit.

\begin{definition} \labbel{simpldef}
Suppose that 
$\bm \alpha = ( \alpha_i) _{i < \omega} $ 
is a sequence of ordinals and recall the definitions of 
$ \zeta_ {\bm \alpha}   $ and 
$ \varepsilon_ {\bm \alpha}  $
from Definition \ref{spec}.
Enumerate the finite  (possibly empty) set
of e-special elements of  
$( \alpha_i) _{i < \omega} $ as 
$\alpha _{i_1}, \dots,  \alpha _{i_h} $.
Let $\alpha _{i_1}^\heartsuit$
be the unique ordinal such that   
$ \zeta_ {\bm \alpha} + \alpha _{i_1}^\heartsuit = \alpha _{i_1}$ 
and similarly for $ \alpha _{i_2}^\heartsuit$ \dots\  
 If $\zeta_ {\bm \alpha} >0$, let $\alpha _{i_1}^\diamondsuit$
be the unique ordinal such that   
$\hat \zeta_ {\bm \alpha}  + \alpha _{i_1}^\diamondsuit = \alpha _{i_1}$ 
and similarly for $ \alpha _{i_2}^\diamondsuit$ \dots
 
Recall that, by the very definition, 
 there are 
only a finite number of e-special elements, and,
by Remark \ref{determ}(d),  they are 
all $\geq \zeta_ {\bm \alpha}  $. In particular, if $h \geq 1$, the ordinals  
 $\alpha _{i_1}^\heartsuit, 
 \alpha _{i_2}^\heartsuit, \dots,
\alpha _{i_1}^\diamondsuit, \alpha _{i_2}^\diamondsuit, \dots$ 
actually exist.

If $\zeta_ {\bm \alpha} >0$, 
express $\zeta_ {\bm \alpha} $ in additive normal form as
$\zeta_ {\bm \alpha}  =
 \omega ^ {\xi_s}  + 
\dots
+ \omega ^ {\xi_1}
+ \omega ^ {\xi_0}$.  
Note that 
$ \hat \zeta_ {\bm \alpha} + \omega ^ {\xi_0} +  \alpha _{i_1}^\heartsuit
= \zeta_ {\bm \alpha} + \alpha _{i_1}^\heartsuit = \alpha _{i_1}=
\hat \zeta_ {\bm \alpha}  + \alpha _{i_1}^\diamondsuit$,
hence  
\begin{equation}\labbel{cuorqua}   
 \omega ^ {\xi_0} + \alpha _{i_1}^\heartsuit = 
\alpha _{i_1}^\diamondsuit
   \end{equation} 
since ordinal addition is left cancellative.

Under the above assumptions, 
define $\sums _{i < \omega} \alpha_{i}$ as 
\begin{equation} \labbel{prima} \tag{S1}         
\begin{aligned}
\sums _{i < \omega} \alpha_{i} &=
(\hat \zeta_ {\bm \alpha}  \x \omega)  \+ \omega ^ {\xi_0}
 \+ \alpha _{i_1}^\diamondsuit \+ \dots \+ 
\alpha _{i_h}^\diamondsuit 
\\&
=
 (\omega ^ {\xi_s+1}  + 
\dots
+ \omega ^ {\xi_1+1}
  + \omega ^ {\xi_0})
 \+ \alpha _{i_1}^\diamondsuit \+ \dots \+ 
\alpha _{i_h}^\diamondsuit 
\end{aligned}
 \end{equation}     
if $\zeta_ {\bm \alpha} >0$, the last exponent $\xi_0$ in the normal form 
of $\zeta_ {\bm \alpha} $ is a limit ordinal and there are only a finite number
of elements of $\bm \alpha$ which are $\geq \zeta_ {\bm \alpha}  $, namely,
$\zeta_ {\bm \alpha} = \varepsilon_ {\bm \alpha}  $.  
Note that there might be a finite number of $\alpha_i$
such that $\alpha_i = \zeta_ {\bm \alpha}$ and in this case 
$\alpha _{i}^\diamondsuit$ does contribute to the sum. 

Otherwise, let
\begin{equation}\labbel{seconda}     \tag{S2}    
\begin{aligned}   
\sums _{i < \omega} \alpha_{i} &=
(\zeta_ {\bm \alpha}  \x \omega ) \+ \alpha _{i_1}^\heartsuit \+ \dots \+ \alpha _{i_h}^\heartsuit 
\\
&=
 (\omega ^ {\xi_s+1}  + 
\dots
+ \omega ^ {\xi_1+1}
  + \omega ^ {\xi_0+1})
\+ \alpha _{i_1}^\heartsuit \+ \dots \+ \alpha _{i_h}^\heartsuit 
\end{aligned}     
  \end{equation}

If $\zeta_ {\bm \alpha} =0$, we apply \eqref{seconda}, namely,
 $\sums _{i < \omega} \alpha_{i}=
\alpha _{i_1}^\heartsuit \+ \dots \+ \alpha _{i_h}^\heartsuit=
\alpha _{i_1}\+ \dots \+ \alpha _{i_h}$.

The definitions are intended in the sense that 
$\sums _{i < \omega} \alpha_{i} =
(\hat \zeta_ {\bm \alpha} \x \omega)  \+ \omega ^ {\xi_0}  $, respectively,
 $\sums _{i < \omega} \alpha_{i} =
\zeta_ {\bm \alpha} \x \omega $ in case the set of e-special elements is empty.
The sum $\sums 0$ of the constantly $0$ 
sequence is 	$0$. 

Note that in case \eqref{seconda} there might be summands 
$\alpha _{i_p}$ such that $\alpha _{i_p}^\heartsuit=0$,
so that they do not contribute to the sum. We have maintained them
for uniformity with the case \eqref{prima}. See item (b)   
 in the following remark  for more details.
\end{definition}   

\begin{remark} \labbel{ze} 
(a) When expressing an ordinal in normal form,
we can equivalently use ordinal addition or Hessenberg natural addition, hence
in the formulas \eqref{prima} and \eqref{seconda}
we can everywhere replace the $+$ signs with
$\+$ signs (on the other hand, the H-natural product $\x$ cannot
generally  be replaced
by ordinal product).  

(b) By Remark \ref{determ}(e),
the sets of z-special and of  e-special 
elements are the same
when $\varepsilon_ {\bm \alpha} $ is successor.
 When $\varepsilon_ {\bm \alpha} $ is limit, summands with value 
$\varepsilon_ {\bm \alpha} $  are the only 
e-special not z-special elements.
Moreover, if $\varepsilon_ {\bm \alpha} $ is limit, 
then $\zeta_ {\bm \alpha} = \varepsilon_ {\bm \alpha}  $,
thus if $\alpha_i = \varepsilon_ {\bm \alpha}  $, then 
  $\alpha _{i}^\heartsuit =0$. 
This means that in case \eqref{seconda}
we can equivalently do  with only z-special elements, rather than
 with e-special elements (in any case, always
using the $\alpha _{i_j}^\heartsuit $s, not the $\alpha _{i_j}^\diamondsuit $s). 
In other words, there might be 
e-special elements which are not z-special (only when 
$\varepsilon_ {\bm \alpha} $ is limit) but 
in case \eqref{seconda} such elements  
do not contribute to the sum  $\sums _{i < \omega} \alpha_{i}$.
On the other hand, such elements 
do contribute in case \eqref{prima}, hence in case 
\eqref{prima}  we actually need
to deal with e-special elements.
 \end{remark}

 \begin{ex32a} \labbel{ex}
We first present some examples. 
For the sake of simplicity, we might write
$ \alpha_0 + \alpha_1 + \alpha_2 + \dots $
for $\sums _{i < \omega} \alpha_{i} $.

(a) For example, 
$ 1 + \omega + \omega^ \omega + 2 +2+2 + \dots =
 \omega^ \omega + \omega 3 $.
Here $\zeta_ {\bm \alpha} =2$, $\varepsilon_ {\bm \alpha} = 3$, 
the summand $1$ does not contribute to the sum since 
$1 < \varepsilon _ {\bm \alpha}  $,
the e-special elements are      
$\omega $ and $  \omega^ \omega$,
clause \eqref{seconda} applies, 
the infinite sequence  of $2$'s contributes 
 $ \omega 2$ to the sum, which adds to the contribution by $ \omega$
to get $ \omega 3$.
Here  $ \alpha _1^ \heartsuit = \omega $,
since $\zeta_ {\bm \alpha} + \omega   =
 2+ \omega = \omega = \alpha _1$.
Similarly, 
$ \alpha _2^ \heartsuit = \omega^ \omega  $.

(b) As another example, 
$ \omega + \omega 5 + \omega^ \omega + 
\omega 2 + \omega 2+ \omega 2 + \dots =
 \omega^ \omega + \omega^ 2  2 + \omega 3  $.
Here $\zeta_ {\bm \alpha} = \omega 2$, $\varepsilon_ {\bm \alpha} = \omega 2 +1$, 
$ \omega $ does not contribute to the sum,
the e-special elements are      
$ \omega 5$ and  $  \omega^ \omega$ and
$( \omega 5)^ \heartsuit = \omega3 $,
 $( \omega^ \omega )^ \heartsuit = \omega^ \omega $.

(c) As another sum to which \eqref{seconda} applies,  
$ \omega + \omega 5 + \omega^ \omega + 
1+  2 + 3+ 4 + 5+ \dots = \omega^ \omega + \omega^ 2  + \omega 4$.
Here $\zeta_ {\bm \alpha} = \varepsilon_ {\bm \alpha} =\omega $
and \eqref{seconda} applies, since the exponent of $ \omega$ 
is $1$, not a limit ordinal.  
The  e-special elements are      
$ \omega$, $ \omega 5$ and  $  \omega^ \omega$, but 
$ \omega$ does not contribute to the sum, since
$ \omega^ \heartsuit = 0 $.

(d) Now we check 
$  \omega^ \omega + \omega^ {\omega ^ \omega }+
\omega  + \omega^ 2+ \omega ^3 + \omega ^4+\omega ^5+\dots =
\omega^ {\omega ^ \omega }+
 \omega^ \omega  2 $.
Here $\zeta_ {\bm \alpha} = \varepsilon_ {\bm \alpha}  =\omega ^ \omega $
and \eqref{prima} applies, since the exponent of $ \omega ^ \omega $ 
is a limit ordinal.  
The  e-special elements are      
$\omega^ \omega $ and  $  \omega^ {\omega ^ \omega }$;
 $\omega^ \omega $ does contribute to the sum, since
$\hat \zeta_ {\bm \alpha} =0$.  

(e) Finally, if $\alpha_0 =\omega^ \omega + \omega $,
$\alpha_1 =\omega^ \omega +  \omega^ \omega $,
$\alpha_2 =\omega^ \omega + \omega^ \omega + \omega  $,
$ \alpha _3  =  \omega^ {\omega ^ \omega } +2$ and
$ \alpha _i =\omega^ \omega   + \omega^ i $, for $i \geq 4$, then
$\sums _{i < \omega} \alpha_{i} =  
\omega^ {\omega ^ \omega }+
 \omega^ {\omega+1} +  \omega^ \omega  3 + \omega + 2$.
Here $\zeta_ {\bm \alpha} = \varepsilon_ {\bm \alpha}  =\omega ^ \omega +
\omega ^ \omega$, thus $\alpha_0$ is not e-special,    
 \eqref{prima} applies again,
$\hat \zeta_ {\bm \alpha}  =\omega ^ \omega $,
$ \alpha_1^ \diamondsuit = \omega^ \omega  $,
$ \alpha_2^ \diamondsuit = \omega^ \omega  + \omega $,
 $ \alpha_3^ \diamondsuit= \omega^ {\omega ^ \omega }+ 2$.
 
Further examples will be presented in the course of the
proof of Theorem \ref{oc}.
\end{ex32a}

\begin{example} \labbel{meglvw}
(a) Recall from the introduction that 
V\"a\"an\"anen and  Wang \cite{VW}
used an infinitary Hessenberg natural sum
in order to define some notion of ``size''
or complexity of an $\mathcal L _{ \omega _{1}, \omega }$-formula.
In more detail,  they set inductively
$ s ( \varphi ) =1$ if $\varphi$  is an atomic formula and, say,
$s( \varphi \wedge \psi ) = s( \varphi) \+ s(\psi ) $ and 
$s \left( \bigvee _{i <  \omega } \varphi _i \right) =
\sumh _{i <  \omega } s(\varphi _i)  $.
Recall that $\sumh _{i <  \omega }$ is computed by taking the supremum
of the partial finite Hessenberg natural sums.    
 
In particular, according to \cite{VW},
if each $\varphi _{ij} $ is atomic, then 
the three infinitary formulas $ \psi_0 =\bigvee _{i <  \omega } \varphi _{i,0} $,
$\psi_1 = \bigvee _{i <  \omega } (\varphi _{i,0}  \wedge \varphi _{i,1})$ and
$\psi_3 = \bigvee _{i <  \omega } (\varphi _{i,0}  \wedge \varphi _{i,1} 
\wedge \dots \wedge \varphi _{i,i})$
 all have the same ``size'' $ \omega$.

On the other hand let us  define $s^*$ in a similar way, but setting 
$s^* \left( \bigvee _{i <  \omega } \varphi _i \right) =
\sums _{i <  \omega } s^*(\varphi _i)  $, instead.
We still get $s^*(\psi_0) = \omega $, but
$s^*(\psi_1) = \omega + \omega $ and 
$s^*(\psi_1) = \omega^2 $.
Thus, using $ \sums $, we get a much finer evaluation of ``sizes'',
since, under any reasonable idea of complexity,
$\psi_1$ and  $\psi_2$ are more compex than $\psi_0$.

(b) Similarly, we can consider a well-founded tree
without infinite branches and such that every node
has at most $ \omega$ children. Then set inductively $s^*(n)=1$,
for every terminal node $n$, and  
$s^*(n)=1 \+ \left(  \sums_ m s^*(m)\right) $,
for $n$ nonterminal, 
where $m$ ranges over the children of $n$.  
Then the \emph{size of the tree}  is defined to be the size 
$s^*(r)$ 
of the root $r$.
It is conceivable that the above definition will have further
applications, besides the related notion used in \cite{VW}.
 \end{example}

\begin{theorem} \labbel{simpth}
The operation $\sums$ 
 is both weakly monotone and 
strictly monotone on z-special elements.
It is the smallest operation which is
strictly monotone on z-special elements.

The operation $\sums$ is invariant
under insertion of any number of $0$'s.

If $\alpha_i = \alpha $, for every $i < \omega$,
then   $\sums _{i < \omega} \alpha_{i} =
\alpha  \x \omega $.
\end{theorem}

\begin{definition} \labbel{s1s2}
The proof of Theorem \ref{simpth}
will be simplified if we establish some properties of the two operations
defined by the conditions \eqref{prima} and \eqref{seconda}
considered alone. In detail, let 
$\suma$ be the operation defined by \eqref{seconda},
for every sequence of ordinals,
and let 
$\sumb$ be defined by \eqref{prima},
for all sequences $\bm \alpha$ of ordinals
with $\zeta_ {\bm \alpha}>0$. 
  
If  $\bm \alpha = ( \alpha_i) _{i < \omega} $
and $\bm \beta  = ( \beta _i) _{i < \omega} $  
are sequences of ordinals, let us write
$\bm \alpha \leq \bm \beta $ to mean that
$\alpha_i \leq \beta _i $ for every $i < \omega$.
Note that, if this is the case,
 then 
$\zeta_ {\bm \alpha} \leq \zeta_ {\bm \beta}$ and
$ \varepsilon_ {\bm \alpha} \leq \varepsilon_ {\bm \beta}$.
\end{definition}  

\begin{lemma} \labbel{lemmo}
Let $\bm \alpha = ( \alpha_i) _{i < \omega} $
and $\bm \beta  = ( \beta _i) _{i < \omega} $  
be sequences of ordinals.
  \begin{enumerate}[(i)]   
 \item 
If $\bm \alpha \leq \bm \beta $,
then 
$\suma _{i < \omega} \alpha_{i} \leq \suma _{i < \omega} \beta _{i} $,
that is, $\suma$ is weakly monotone.
\item
If $\bm \alpha \leq \bm \beta $
and $ \varepsilon_ {\bm \alpha} =
\varepsilon_ {\bm \beta} >1$, 
then 
$\sumb _{i < \omega} \alpha_{i} \leq \sumb _{i < \omega} \beta _{i} $.
\item
If $\zeta_ {\bm \alpha} >0$, then
$\sumb _{i < \omega} \alpha_{i} \leq \suma _{i < \omega} \alpha  _{i} $.
\item
Suppose that  $\bm \alpha \leq \bm \beta $,
$ \zeta_ {\bm \alpha}  < \zeta_ {\bm \beta}$, 
 $ \zeta_ {\bm \beta} = \omega ^ {\rho_s}  + 
\dots
+ \omega ^ { \rho _1}
+ \omega ^ {\rho_0}$ in normal form with 
$\rho_0$ a limit ordinal
and there are only a finite number
of elements of $\bm \beta $ which are $\geq \zeta_ {\bm \beta}$. 
Then 
$\suma _{i < \omega} \alpha_{i} \leq \sumb _{i < \omega} \beta   _{i} $.
  \end{enumerate} 
 \end{lemma}

\begin{proof}
(i)
By Remark \ref{ze}(b), here in (i) we can do working just
with z-special elements. 
Let us divide the finite set 
of those indexes $i$  such that 
$ \zeta_ {\bm \alpha} < \alpha_i$
into two sets:
let $\alpha_{j_1}, \dots $
enumerate those elements such that 
$ \zeta_ {\bm \alpha} < \alpha_{j_1} \leq \zeta_ {\bm \beta} $ 
and let 
 $\alpha_{k_1}, \dots $
enumerate those elements such that 
$ \zeta_ {\bm \beta}  < \alpha_{k_1}  $.
One or possibly both sets might be empty,
in particular, the first set is necessarily empty when
$ \zeta_ {\bm \alpha} = \zeta_ {\bm \beta} $. 
Let $\delta$  be the 
unique ordinal
such that $ \zeta_ {\bm \alpha} + \delta = \zeta_ {\bm \beta} $.
In the case of the $\alpha_j$s  we have 
$ \zeta_ {\bm \alpha} + \alpha   _{j_1}^\heartsuit= \alpha   _{j_1} \leq \zeta_ {\bm \beta} 
= \zeta_ {\bm \alpha} + \delta $,
hence 
$ \alpha   _{j_1}^\heartsuit \leq \delta $. 
 In the  case of the $\alpha_k$s, we have 
$ \zeta_ {\bm \alpha} + \alpha   _{k_1}^\heartsuit= \alpha   _{k_1} \leq \beta  _{k_1} = \zeta_ {\bm \beta} +  \beta _{k_1}^\heartsuit
= \zeta_ {\bm \alpha}  + \delta  +  \beta _{k_1}^\heartsuit$, 
hence 
$\alpha   _{k_1}^\heartsuit \leq  \delta  +  \beta _{k_1}^\heartsuit$
and similarly for the other indexes. 
Here $\alpha   _{k_1}^\heartsuit $ 
is computed with respect to $\zeta_ {\bm \alpha}$, while   
$  \beta _{k_1}^\heartsuit$
is computed with respect to $\zeta_ {\bm \beta} $.
Then 
\begin{equation*}     
\begin{aligned} 
\suma _{i < \omega } \alpha_i &=
 \alpha^\heartsuit  _{j_1} \+ \dots \+  \alpha^\heartsuit  _{j_h} \+  
\alpha^\heartsuit  _{k_1} \+ \dots \+ \alpha^\heartsuit  _{k_\ell} \+
 ( \zeta_ {\bm \alpha}  \x  \omega )
 \\
&\leq (\delta \x  h) \+  (\delta  +  \beta _{k_1}^\heartsuit) \+ \dots \+
(\delta  +  \beta _{k_ \ell}^\heartsuit) \+ 
 ( \zeta_ {\bm \alpha}  \x  \omega )
 \\
&\leq (\delta \x  h) \+  \delta  \+  \beta _{k_1}^\heartsuit \+ \dots \+
\delta  \+  \beta _{k_ \ell}^\heartsuit \+ 
 ( \zeta_ {\bm \alpha}  \x  \omega )
\\
&=  \beta _{k_1}^\heartsuit \+ \dots\+
\beta _{k_ \ell}^\heartsuit  \+  ( \zeta_ {\bm \alpha}  \x  \omega )
\+ (\delta \x (h + \ell))
\\
&\leq ^{\text{Lemma \ref{lemdelkatze}}} 
\beta _{k_1}^\heartsuit \+ \dots \+
\beta _{k_ \ell}^\heartsuit \+  (\zeta_ {\bm \beta}   \x  \omega )
\leq 
\suma _{i < \omega } \beta _i 
\end{aligned}
  \end{equation*}   
where the last inequality follows from the assumption
that
$ \zeta_ {\bm \beta}  < \alpha_{k_1} $, hence 
$ \zeta_ {\bm \beta}  < \beta _{k_1} $, 
thus $\beta _{k_1}$ does contribute to the sum 
defining  $\suma _{i < \omega } \beta _i $ 
and similarly for the other indexes. 
Of course, there might be further indexes $i$ such that 
$\beta _{i}^\heartsuit$ 
contributes to $\suma _{i < \omega } \beta _i $,
but the inequality is  satisfied anyway. 

(ii) 
If $ \alpha _{i}$ is not e-special
wrt $\bm \alpha $, 
 then $ \alpha _i$
does not contribute
to the sum defining 
$\sumb _{i < \omega} \alpha_{i} $, 
while $\beta _{i} $ possibly  contributes. 
Note that the sum is defined, since
$ \varepsilon_ {\bm \alpha} >1$,
hence $ \zeta_ {\bm \alpha} >0$. 

On the other hand, if
$ \alpha _{i}$ is  e-special wrt $\bm \alpha $,
then $ \beta  _{i}$ is  e-special
wrt $\bm \beta  $,
since
$ \alpha  _{i} \leq \beta   _{i}$
and
$ \varepsilon_ {\bm \alpha}=
 \varepsilon_ {\bm \beta}  $.
Again since $ \alpha  _{i} \leq \beta   _{i}$,
 we have
$\alpha ^\diamondsuit_{i}   \leq  
\beta ^\diamondsuit_{i} $.
By (weak) monotonicity of the Hessenberg natural sum we get 
$\sumb _{i < \omega} \alpha_{i} \leq
\sumb _{i < \omega} \beta _{i} $, since, as observed in
Remark \ref{determ}(b)(c), for every sequence,
$\zeta$ is determined by $\varepsilon$,
hence  $\varepsilon_ {\bm \alpha} =
 \varepsilon_ {\bm \beta}  $
implies $ \zeta_ {\bm \alpha} = \zeta_ {\bm \beta} $.

(iii) We get the desired inequality
by observing that 
\begin{equation*}      
\begin{aligned} 
\sumb _{i < \omega} \alpha_{i} &=
 (\omega ^ {\xi_s+1}  + 
\dots
+ \omega ^ {\xi_1+1}
  + \omega ^ {\xi_0})
 \+ \alpha _{i_1}^\diamondsuit \+ \dots \+ 
\alpha _{i_h}^\diamondsuit 
\\
&= ^{\eqref{cuorqua}} (\omega ^ {\xi_s+1}  + 
\dots
+ \omega ^ {\xi_1+1}
  + \omega ^ {\xi_0})
 \+ (\omega ^ {\xi_0} + \alpha _{i_1}^\heartsuit)
 \+ \dots \+ 
(\omega ^ {\xi_0} + \alpha _{i_h}^\heartsuit)
\\
&\leq
(\omega ^ {\xi_s+1}  + 
\dots
+ \omega ^ {\xi_1+1}
  + \omega ^ {\xi_0}(h+1))
 \+  \alpha _{i_1}^\heartsuit
 \+ \dots \+ 
 \alpha _{i_h}^\heartsuit 
\\
&<
(\omega ^ {\xi_s+1}  + 
\dots
+ \omega ^ {\xi_1+1}
  + \omega ^ {\xi_0+1})
 \+  \alpha   _{i_1}^\heartsuit
 \+ \dots \+ 
 \alpha   _{i_h}^\heartsuit 
\leq 
 \suma _{i < \omega} \alpha  _{i}.
 \end{aligned} 
 \end{equation*}

(iv) 
First, we observe that we may assume that 
$\zeta_ {\bm \alpha} \geq \hat \zeta_ {\bm \beta}$,
since otherwise we can introduce another
sequence  $\bm \alpha^*$ by setting $\alpha_i^*= \hat \zeta_ {\bm \beta} $, if
$\alpha_i \leq \hat \zeta_ {\bm \beta} \leq \beta_i $ and
$\alpha_i^*= \alpha _i$, otherwise.
If $\zeta_ {\bm \alpha} < \hat \zeta_ {\bm \beta}$, then
infinitely many summands of 
$\bm \alpha^*$ are equal to $\hat \zeta_ {\bm \beta}$,
since  
infinitely many summands of 
$\bm \beta $ are  $>\hat \zeta_ {\bm \beta}$
and only finitely many summands of 
$\bm \alpha  $ are  $>\zeta_ {\bm \alpha} $,
hence only finitely many summands of 
$\bm \alpha  $ are $> \hat \zeta_ {\bm \beta}$.
Thus $\zeta_{\bm \alpha^*}= \hat \zeta_ {\bm \beta}$. 
 By the already proved item (i),
$\suma _{i < \omega} \alpha_{i} \leq \suma _{i < \omega} \alpha ^* _{i} $,
hence if we prove 
$\suma _{i < \omega} \alpha ^* _{i} \leq \sumb _{i < \omega} \beta   _{i}$,
we also get
$\suma _{i < \omega} \alpha _{i} \leq \sumb _{i < \omega} \beta  _{i}$.
 
So let us assume that 
$\zeta_ {\bm \alpha} \geq \hat \zeta_ {\bm \beta}$.
Let $\delta$ be the unique ordinal
such that 
$ \zeta_ {\bm \beta}  = 
 \zeta_ {\bm \alpha}   + \delta $
and let $\hat \delta $
be the unique ordinal
such that 
$ \hat \zeta_ {\bm \beta} + \hat \delta  = 
 \zeta_ {\bm \alpha}  $. 
Moreover,  divide the finite set 
of the elements 
$ \alpha_i$  which  are e-special wrt 
$\bm \alpha$ into the set
 $\alpha_{j_0}, \dots $
of those elements such that 
$ \alpha_{j_0} < \zeta_ {\bm \beta} $ 
and the set 
 $\alpha_{k_0}, \dots $
of those elements such that 
$ \zeta_ {\bm \beta}  \leq \alpha_{k_0}  $.
Since $\zeta_ {\bm \alpha} \geq \hat \zeta_ {\bm \beta}$,
we get $\delta \leq \omega ^ { \rho _0}$, recalling that
$ \omega ^ { \rho _s}  + 
\dots
+ \omega ^ { \rho _0}$   is the normal form of 
$\zeta_ {\bm \beta} $. 
Since   $ \alpha_{j_0} < \zeta_ {\bm \beta} $
and $\zeta_ {\bm \alpha} + \alpha_{j_0}^\heartsuit = \alpha_{j_0}$,
we have  $ \alpha_{j_0}^\heartsuit <\delta \leq \omega ^ { \rho _0}$
and similarly for the other indexes.
From  
 $ \hat \zeta_ {\bm \beta} + \hat \delta  = 
 \zeta_ {\bm \alpha}  $ and 
$ \zeta_ {\bm \alpha} <  \zeta_ {\bm \beta}$, we get
$\hat \delta < \omega ^ { \rho _0}$, hence
$\hat \delta \x \omega  < \omega ^ { \rho _0}$,
since $ \rho _0$ is limit.
Summing up,
$ \alpha_{j_0}^\heartsuit \+ \dots \+
\alpha_{j_h}^\heartsuit \+ \hat \delta \x \omega  < \omega ^ { \rho _0}$.

Since $ \zeta_ {\bm \beta}  \leq \alpha_{k_0}  \leq \beta _{k_0} $,
$\beta _{k_0} $ is e-special wrt  $\bm \beta $,
by the finiteness assumption (compare Remark \ref{determ}(g)).
Since $ \hat\zeta_ {\bm \beta} \leq \zeta_ {\bm \alpha} 
< \zeta_ {\bm \beta}  \leq \alpha_{k_0}  \leq \beta _{k_0} $,
 $\hat \zeta_ {\bm \beta} + \beta _{k_0} ^\diamondsuit = \beta _{k_0}$
and $\zeta_ {\bm \alpha} + \alpha_{k_0}^\heartsuit = \alpha_{k_0}$,
 we have $\alpha_{k_0}^\heartsuit \leq \beta _{k_0} ^\diamondsuit$ 
and similarly for the other indexes. Hence
\begin{equation*}     
\begin{aligned} 
\suma _{i < \omega } \alpha_i &=
 \alpha^\heartsuit  _{j_0} \+ \dots \+  \alpha^\heartsuit  _{j_h} \+  
\alpha^\heartsuit  _{k_0} \+ \dots \+ \alpha^\heartsuit  _{k_\ell} \+
 ( \zeta_ {\bm \alpha}  \x  \omega )
 \\
&\leq
 \alpha^\heartsuit  _{j_0} \+ \dots \+  \alpha^\heartsuit  _{j_h} \+  
\alpha^\heartsuit  _{k_0} \+ \dots \+ \alpha^\heartsuit  _{k_\ell} \+
 (\hat \zeta_ {\bm \beta}   \x  \omega ) \+ ( \hat \delta \x \omega )
\\
&< 
\beta _{k_0}^\diamondsuit \+ \dots \+
\beta _{k_ \ell}^\diamondsuit \+  
(\hat \zeta_ {\bm \beta}   \x  \omega  \+  \omega ^ { \rho _0})
\leq 
\sumb _{i < \omega } \beta _i 
\end{aligned}
  \end{equation*}   
since $  \zeta_ {\bm \alpha} =
\hat \zeta_ {\bm \beta} + \hat \delta  \leq
 \hat \zeta_ {\bm \beta} \+ \hat \delta $, then using
 distributivity of the H-natural product with respect to  the Hessenberg natural sum.
As usual, there might be more 
$\beta_i^\diamondsuit$ contributing to  
$\sumb _{i < \omega } \beta _i $, but the last inequality still holds. 
 \end{proof}

 \begin{proof}[Proof of Theorem \ref{simpth}]
It is straightforward that $\sums$ is invariant
under insertions of  $0$'s.
The last statement is immediate from the definition,
since, under the assumption, $\zeta_ {\bm \alpha}= \alpha $,
$\varepsilon_ {\bm \alpha} = \alpha +1$, thus  case \eqref{seconda} applies.
Moreover, there is no e-special element.

We now prove the first statement. 
By Remark \ref{primo}(b)
it is enough to prove that
$\sums$ satisfies 
\eqref{m} and \eqref{s3'} in Definition \ref{defop}.

(I) Let us first prove weak monotonicity \eqref{m}.
Suppose that $ \alpha _i \leq \beta _i$, for every $i < \omega$. 
The proof is divided in various cases.

If both $\sums _{i < \omega } \alpha  _i $
and $\sums _{i < \omega } \beta _i $
are given by \eqref{seconda}, then the result
is a consequence of Lemma \ref{lemmo}(i).    

If  $\sums _{i < \omega } \alpha  _i $
is given by \eqref{seconda}
and $\sums _{i < \omega } \beta _i $
is given by \eqref{prima},
then necessarily $ \zeta_ {\bm \alpha} < \zeta_ {\bm \beta} $,
since if \eqref{prima} applies to $\bm \beta $,
then only finitely many elements of $\bm \beta $
are $\geq \zeta_ {\bm \beta} $, by assumption, hence 
only finitely many elements of $\bm \alpha  $
are $\geq \zeta_ {\bm \beta} $. If 
by contradiction
$ \zeta_ {\bm \alpha} = \zeta_ {\bm \beta} $,
then \eqref{prima} should be used in computing  
$\sums _{i < \omega } \alpha  _i $, contradicting our assumption.
Hence  $ \zeta_ {\bm \alpha} < \zeta_ {\bm \beta} $
and the desired inequality follows from
Lemma \ref{lemmo}(iv), since the hypotheses there are satisfied, 
if  \eqref{prima} is used in order
to evaluate $\sums _{i < \omega } \beta _i $.

If both $\sums _{i < \omega } \alpha  _i $
and $\sums _{i < \omega } \beta _i $
are given by \eqref{prima} and 
$ \varepsilon_ {\bm \alpha}=
\varepsilon_ {\bm \beta} $,
 the result
is Lemma \ref{lemmo}(ii).    
So, let us assume that
$ \varepsilon_ {\bm \alpha} <
\varepsilon_ {\bm \beta} $.
We claim that also
$ \zeta_ {\bm \alpha} < \zeta_ {\bm \beta} $.
Indeed, 
by Remark \ref{determ}(b)(c)(d),
$\zeta_ {\bm \alpha}  \leq \varepsilon_ {\bm \alpha}  $ and 
$\varepsilon_ {\bm \beta} \leq \zeta_ {\bm \beta} +1$, 
so that if $ \varepsilon_ {\bm \alpha} <
\varepsilon_ {\bm \beta} $
and, by contradiction,  $ \zeta_ {\bm \alpha} = \zeta_ {\bm \beta} $, we have
$\zeta_ {\bm \alpha} = \zeta_ {\bm \beta} =\varepsilon_ {\bm \alpha} <
 \varepsilon_ {\bm \beta} $.
 Then, by Remark \ref{determ}(b),
 there are infinitely many elements of
$\bm \beta  $ which are equal to 
 $\zeta_ {\bm \beta}$, but this contradicts the assumption that
 $\sums _{i < \omega } \beta _i $
is evaluated using
\eqref{prima}.
Hence $ \zeta_ {\bm \alpha} < \zeta_ {\bm \beta} $ and
we get 
$\sums _{i < \omega } \alpha  _i =
\sumb _{i < \omega } \alpha  _i \leq
\suma _{i < \omega } \alpha  _i  \leq  
\sumb _{i < \omega } \beta _i =
 \sums _{i < \omega } \beta _i $
by Lemma \ref{lemmo}(iii)(iv). 
 
Finally, 
if  $\sums _{i < \omega } \alpha  _i $
is given by \eqref{prima}
and $\sums _{i < \omega } \beta _i $
is given by \eqref{seconda},
then 
$\sums _{i < \omega } \alpha  _i =
 \sumb _{i < \omega } \alpha  _i \leq 
\suma _{i < \omega } \alpha  _i \leq 
\suma _{i < \omega } \beta _i =
\sums _{i < \omega } \beta _i $
by Lemma \ref{lemmo}(iii)(i). 

(II) We now prove \eqref{s3'}.
If $\bm \alpha = ( \alpha_i) _{i < \omega} $
and $\bm \beta  = ( \beta _i) _{i < \omega} $  
are sequences which differ only at place $\bar{\imath}$, then 
$\zeta_ {\bm \alpha} = \zeta_ {\bm \beta}$ and case 
\eqref{prima}, resp., \eqref{seconda} applies to
$\bm \alpha$ if and only if it applies to $\bm \beta $. 
Suppose that $ \zeta_ {\bm \alpha}  \leq 
\alpha _{\bar{\imath}}   <  
\beta _{\bar{\imath}} $.
If  $\zeta_ {\bm \alpha}  = 
\alpha _{\bar{\imath}} $ and 
$  \varepsilon _ {\bm \alpha} = \zeta_ {\bm \alpha} +1$, 
then $\alpha_i$  is not e-special, but $\beta_i$  is
e-special, and $\beta_i$  does contribute to the sum. 
Otherwise, both $\alpha_i$ and $\beta_i$  are e-special and,
according to the case which applies,
$\alpha ^\heartsuit_{\bar{\imath}}   <  
\beta ^\heartsuit_{\bar{\imath}} $ 
or $\alpha ^\diamondsuit_{\bar{\imath}}   <  
\beta ^\diamondsuit_{\bar{\imath}} $. 
In any  case, \eqref{s3'} follows from strict monotonicity of the
(finitary) Hessenberg natural sum.
\renewcommand{\qedsymbol}{$\Box$ (to be continued)}
\end{proof}

The proof of Theorem \ref{simpth} will be completed
in Section \ref{rank}.

\begin{proposition} \labbel{ops2}
The operation $\suma$ is both (i) monotone and (ii) strictly 
monotone on z-special elements.
 \end{proposition} 

 \begin{proof}
From Lemma \ref{lemmo}(i) and
the arguments in part (II) 
in the proof of Theorem \ref{simpth}.  
 \end{proof}

\section{An order-theoretical characterization in terms of mixed sums} \labbel{order}

Let $I$ be a possibly infinite set.
An ordinal $\gamma$ is a \emph{mixed sum}
of some family $(\alpha_i) _{i \in I} $ 
of ordinals
 if there are 
pairwise disjoint
subsets $(A_i) _{i \in I} $ of $\gamma$
such that 
$\bigcup _{i \in I} A_i = \gamma  $
and, for every $i \in I$,
$A_i$ has order-type $\alpha_i$,
with respect  to the order induced on $A_i$ 
by $\gamma$.    
 The idea of a finite mixed sum of ordinals occurs implicitly as early 
as in 1942 in \cite{Car}; see the following theorem.
The first occurrence of the expression ``mixed sum''
we are aware of  appears in \cite{Du},
with a somewhat restricted meaning.
See \cite[Section 4]{t} for more references.   
Note that some authors use the expression \emph{shuffled 
sum}  for what we call here a mixed sum.  
Finite mixed sums provide still another characterization
of the Hessenberg natural sum.

\setcounter{theorem}{-1}
\begin{theorem} \labbel{carr}
\emph{(Carruth \cite{Car}, Neumer \cite{neumer})} 
For every $ n < \omega$ and ordinal numbers
 $\alpha_0, \dots, \alpha _n$,  
the largest mixed
sum of $( \alpha_i) _{i \leq n} $ exists and is 
$\alpha_0 \+ \alpha _1 \+ \dots \+ \alpha _n $.
 \end{theorem}

\begin{proof} (Sketch)
We sketch a proof for later use.
If $\alpha_0= \omega ^ {\xi_h}  + 
\dots + \omega ^ {\xi_0}$   in additive normal form,
consider $\alpha_0$ as the union of the ``blocks''
$[0, \omega ^ {\xi_h})$, 
$[\omega ^ {\xi_h}, \omega ^ {\xi_h}+\omega ^ {\xi_{h-1}})$, \dots,
 $[\omega ^ {\xi_h}  + 
\dots + \omega ^ {\xi_1}, \omega ^ {\xi_h}  + 
\dots + \omega ^ {\xi_1} +  \omega ^ {\xi_0})$, of respective order-types 
$\omega ^ {\xi_h}$, $\omega ^ {\xi_{h-1}}$, 
\dots, $\omega ^ {\xi_0}$. Here we are using the standard interval
notation, e.~g., $[\omega ^ {\xi_h}, \omega ^ {\xi_h}+\omega ^ {\xi_{h-1}})=
\{ \, \gamma < \alpha _0   \mid  \omega ^ {\xi_h} \leq \gamma < 
\omega ^ {\xi_h} + \omega ^ {\xi_{h-1}} \, \} $.

Represent each $\alpha_i$ as a union of blocks as above and 
arrange all such blocks in decreasing order, with respect to length.
We get exactly the order-type of $\alpha_0  \+ \dots \+ \alpha _n $.
Thus $\alpha_0  \+ \dots \+ \alpha _n $ can be realized as  a mixed sum
of the $\alpha_i$s. 

Suppose by contradiction that some ordinal $\gamma$ strictly larger
 than $\alpha_0  \+ \dots \+ \alpha _n $ can be realized as a mixed sum
of $( \alpha_i) _{i \leq n} $.
Choose a counterexample with $\gamma$ minimal,
 thus $\gamma$ is the disjoint union
of $A_0$, \dots, $A_n$, such subsets have order-type, respectively,
$\alpha_0$, \dots, $\alpha_n$, and     
$ \gamma ' = \alpha_0  \+ \dots \+ \alpha _n < \gamma $.
Thus $\gamma' = (A_0 \cap \gamma ') \cup \dots \cup (A_n \cap \gamma ')$,
hence $\gamma'$ is a mixed sum of the ordinals 
$\alpha'_0$, \dots, $\alpha'_n$, where $\alpha'_0$
is the order-type of  $A_0 \cap \gamma '$, etc.
In particular,
$\alpha'_0 \leq \alpha_0$, etc. 
Since $\gamma' < \gamma = A_0 \cup \dots \cup A_n$, we have
$\gamma' \in A_i$, for some $i \leq n$, hence $\alpha'_i < \alpha _i$,
since $\alpha'_i$ is the order-type of  $A_i \cap \gamma '$.
By strict monotonicity of the Hessenberg natural sum,
$\alpha'_0  \+ \dots \+ \alpha' _n < \alpha_0  \+ \dots \+ \alpha _n 
= \gamma '$. This contradicts the minimality of $\gamma$. 
\end{proof}

\begin{prob40a} \labbel{refer}   
An anonymous
referee asked the following interesting problem. 
Suppose that $\gamma$ is realized as a mixed sum of 
 $\alpha_0, \alpha _1, \dots$. Study the possible topologies
induced by the topology of $\gamma$ on  
$\alpha_1, \alpha _2, \dots$ through $ A_0, A _1, \dots$
as subspaces of $\gamma$.
\end{prob40a}

There is no immediate generalization of Theorem \ref{carr}
for mixed sums of an infinite family of ordinals.
In fact, every countable ordinal is a mixed sum of a countable family of $1$s.  
On the other hand, some infinitary sums allow a characterization
in terms of mixed sums satisfying some special properties
\cite{w,t}. We show that the same applies to $\sums$. 

\begin{definition} \labbel{ormf}    
Suppose that $\gamma$ is a mixed sum of
$(\alpha_i) _{i \in I} $ as realized by some
family $(A_i) _{i \in I} $ of subsets of $\gamma$
as above, that is, the $A_i$ are 
pairwise disjoint,
$\bigcup _{i \in I} A_i = \gamma  $
and each 
$A_i$ has order-type $\alpha_i$. For every $i \in I$,
let $h_i$ be the order-preserving bijection 
from $A_i$ to $\alpha_i$.
  
Let us say that some mixed sum is \emph{order respecting modulo finite},
\emph{ormf}, for short, if, for every $i \in I$   and $\eta \in A_i$,
the set 
\begin{equation}\labbel{or}  
    \{ \, j \in I \mid  \text{there is }\theta \in A_j 
\text{ such that } \theta < \eta  
\text{ and }  h_j ( \theta  ) \geq h_i( \eta  ) \,\}
 \end{equation}  
is finite.
Essentially, this means that in $\gamma$  any element $ \eta $ with some rank
$\rho=h_i( \eta  )$, with respect to the set $A_i$ to which 
$ \eta $ belongs, 
is allowed to be larger than an arbitrary number of elements with smaller rank,
with respect to their relative  $A_j$s. On the contrary,
$ \eta  $ is allowed to be larger  than a set of elements with 
larger or equal rank taken only from a finite set of $A_j$s. 

It is 
interesting to compare the above notion
with a parallel notion used in \cite{w}.
In \cite{w} we defined a mixed sum to be \emph{left-finite}
if  $  \{ \, j \in I \mid  \text{there is }\theta \in A_j 
\text{ such that }   \theta  < \eta   \,\}$ is finite.
In this latter definition we take into account also those
$\theta$ such that $ h_j ( \theta  ) \leq h_i( \eta  )$; in particular, every 
left-finite realization   is ormf.
\end{definition}

The next lemma is almost immediate from those parts of Theorem \ref{simpth}
proved so far 
and the definition of an ormf mixed sum.
Recall that an ordinal is the set of all the smaller ordinals, hence
if $A \subseteq \gamma $ and $\eta < \gamma $, then
$A \cap \eta = A \cap [0, \eta ) =
\{ \theta \in A \,  \mid  \theta < \eta  \, \} $.
We sometimes maintain the interval notation, as  in 
$A \cap [0, \eta )$ for clarity. 

\begin{lemma} \labbel{orf}
Suppose that $\gamma \geq \sums _{i < \omega} \alpha _{i}$
and $\gamma$  is an order respecting modulo finite mixed sum of
$( \alpha _i) _{i < \omega} $ realized by  $(A_i)_{i < \omega} $.
Suppose further that $\eta < \gamma $ and, for every $i < \omega$, 
let $ \beta _i$
be the order-type of   
$B_i = A_i \cap [0, \eta )  $.

Then   $\sums _{i < \omega} \beta  _{i}
< \sums _{i < \omega} \alpha _{i} \leq \gamma $.
 \end{lemma} 

We will soon prove  that 
if $\gamma$  is an ormf mixed sum of
$( \alpha _i) _{i < \omega} $,
 then $\gamma \leq \sums _{i < \omega} \alpha _{i}$,
so that the only actually occurring case in the lemma
is when $\gamma = \sums _{i < \omega} \alpha _{i}$.
However, we need the more general statement
in the subsequent proof of Theorem \ref{oc}.

\begin{proof} 
Straightforwardly, $\beta_i \leq \alpha _i$, for every $i < \omega$.  
By the definition of a mixed sum,
there is a unique $\bar \imath  < \omega$ such that
 $\eta \in A_{\bar \imath }$,
in particular, the order-type of $B_{\bar \imath }$ is strictly less than
the order-type of $A_{\bar \imath }$, that is,
 $ \beta _{\bar \imath } < \alpha  _{\bar \imath }$.
Since the realization is ormf, 
there exist only a finite number of indexes $j$  such that
the order-type of $B_j$ is $\geq h _{\bar \imath }( \eta)$,
that is, $ \varepsilon _ {\bm \beta } \leq h _{\bar \imath }( \eta)$,
in particular, $\zeta _ {\bm \beta } \leq h _{\bar \imath }( \eta)$
(actually, $\zeta _ {\bm \beta } < h _{\bar \imath }( \eta)$,
except possibly when $h _{\bar \imath }( \eta)$ is limit). 
Since, by construction,  $ h _{\bar \imath }( \eta) < \alpha _{\bar \imath } $,
we get  $\zeta _ {\bm \beta } < \alpha _{\bar \imath } $.
 Thus  
 $\sums _{i < \omega} \beta  _{i}
< \sums _{i < \omega} \alpha _{i}$,
since $ \beta _{\bar \imath } < \alpha  _{\bar \imath }$
and $\sums$ is strictly monotone
on z-special elements \eqref{s3}, by the proved part of Theorem \ref{simpth}
(note that here $\alpha$  and $\beta$  are swapped). 
\end{proof}

\begin{theorem} \labbel{oc}
For every sequence $( \alpha _i) _{i < \omega} $ of ordinals, 
the ordinal $\sums _{i < \omega} \alpha _{i}$
is the largest mixed sum of $( \alpha _i) _{i < \omega} $
which is order respecting modulo finite.
 \end{theorem}

 \begin{proof}
We first show that $\sums _{i < \omega} \alpha_{i}$
can be realized as an  ormf mixed sum of the $\alpha_i$.
Recall that \emph{ormf}  is a shorthand for  order respecting modulo finite.
The proof will be obtained by  combining three 
basic cases, which 
we now describe.

(a)  We first consider the case when, for some ordinal $\xi$,
the sequence $ \bm \alpha  = ( \alpha_i) _{i < \omega} $ 
is such that $\alpha_i = \omega ^ \xi$
for infinitely many $i < \omega$, and
 $\alpha_i <  \omega ^ \xi $
for the remaining indexes (possibly, an empty set).
In this case $\zeta_ {\bm \alpha } =  \omega ^ \xi$,
\eqref{seconda} applies and   
$\sums _{i < \omega} \alpha_{i} = \omega ^ {\xi +1}$.
In order to realize $\omega ^ {\xi +1}$, simply put a copy
of each $\alpha_i$ one after the other, in the order
given by the indexing of the sequence.

In more detail, recall that we denote intervals of ordinals in the 
usual way, e.~g., $[ \alpha _1, \alpha _2 ) = 
\{ \, \varepsilon \mid 
\alpha _1 \leq \varepsilon < \alpha _2   \, \}$.
Thus we set 
 $A_0= [0, \alpha _0)$,  $A_1= [ \alpha _0, \alpha _0 + \alpha _1)$,
$A_2= [\alpha _0 +  \alpha_1 , \alpha _0 +  \alpha_1 + \alpha _2)$, etc.  

The (order-types of) those $\alpha_i  $ which are $< \omega ^ \xi $
are absorbed by some subsequent $\alpha_j$
with order-type $ \omega ^ \xi $, since there are infintely
many such $\alpha_j$.
We thus get the order-type of an $ \omega$-indexed chain
of ``blocks'' each having order-type $ \omega ^ \xi $,
thus the full chain has order-type
$ \omega ^ \xi \omega = \omega ^ {\xi+1} $.  
The realization is clearly ormf, actually, left-finite in the sense from \cite{w}.

(b)   For some ordinal $\xi>0$,
the sequence $ \bm \alpha  = ( \alpha_i) _{i < \omega} $ 
is such that $\alpha_i < \omega ^ \xi$,
for every $i < \omega$, but, for every $\delta<  \omega ^ \xi$,
there is $i < \omega$ such that  
 $ \delta \leq \alpha_i $.
Furthermore, in the present case we assume that $\xi$ is a limit ordinal.
Again,  $\zeta_ {\bm \alpha } =  \omega ^ \xi$, 
but in this case
\eqref{prima} applies and   
$\sums _{i < \omega} \alpha_{i} = \omega ^ {\xi}$.
As in (a), if we put a copy
of each $\alpha_i$ one after the other in their order,
we realize $\omega ^ \xi$. Indeed, we have 
arbitrarily large $\alpha_i$s less than
the limit ordinal $\omega ^ \xi$, hence the order-type
of the full chain is $\geq \sup \{ \,  \delta \mid \, \delta < \omega ^ \xi \} 
= \omega ^ \xi$, since $\xi >0$, hence  $\omega ^ \xi$ is limit.
On the other hand, 
each finite initial segment of the 
chain has order-type  $< \omega ^ \xi$,
since $\omega ^ \xi$ is additively indecomposable,
hence the full chain cannot have order-type  $> \omega ^ \xi$.
Again, the realization is obviously ormf, actually,
 left-finite.

(c) The assumptions in the final case are
 as in (b), but in this case we assume that 
$\xi$ is a successor ordinal, instead, say, $\xi = \sigma +1$.
In detail, we also assume that
 $\alpha_i < \omega ^ \xi$,
for every $i < \omega$, but, for every $\delta<  \omega ^ \xi$,
there is $i < \omega$ such that  
 $ \delta \leq \alpha_i $.

In this case 
\eqref{seconda} applies and   
$\sums _{i < \omega} \alpha_{i} = \omega ^ {\xi +1}$.
Here a slightly more elaborate construction is needed.

We first exemplify the construction in the special case
$\xi =1$, say, for the sequence $(1,2,3, \dots, n, \dots)$.  
Set $A_0 = \{ 0 \} $, 
$A_1 = \{ 1 , \omega \} $,
 $A_2 = \{ 2 , \omega+1, \omega 2 \} $,
 $A_3 = \{ 3 , \omega+2, \omega 2+1, \omega 3 \} $, \dots\ 
Each $A_i$ has order-type $\alpha_i$ and 
$\bigcup _{i < \omega } A_i = \omega^2 = \omega ^ {\xi +1}$,
thus we obtain the result in this special case.   

In the general case here in (c), construct a chain by using
a copy of the initial segment of order-type
 $\omega ^ \sigma $ of each long enough $\alpha_i$
(note that in the previous paragraph, $\sigma=0$, $\omega ^ \sigma =1$);
if $\alpha_i < \omega ^ \sigma$, instead, take the whole of $\alpha_i$.
Since we have infinitely many indexes such that 
$\alpha_i \geq \omega ^ \sigma$, this initial chain
contributes to a set of order-type   
 $\omega ^ \sigma \omega =  \omega ^ {\sigma+1} = 
 \omega ^ \xi $.
Going on, 
append atop the above-described first chain,
a second chain constructed in a similar way.
In detail, merge a copy of the  segment 
 $[\omega ^ \sigma, \omega ^ \sigma+\omega ^ \sigma)  $
 of each long enough $\alpha_i$;
if $\alpha_i $ is not long enough, take the whole of the 
remaining part of $\alpha_i$; if $\alpha_i$ has already been exhausted when
constructing the first chain, simply do not consider it.
This construction adds a second chain of order-type
$ \omega ^ \xi$ to the first chain, hence we get 
$  \omega ^ \xi 2$ in total. 
Since there are arbitrarily large $\alpha_i$ less than
 $\omega ^ \xi$, we never fall short of elements,
that is, we can iterate the construction $ \omega$ times,
thus realizing $  \omega ^ \xi \omega =   \omega ^ {\xi+1} $.

Explicitly, if, say, $\alpha_0 = \omega ^ \sigma 3 + 10$,
then 
\begin{equation*}  
A_0 = [0, \omega ^ \sigma) \cup 
[\omega ^ {\sigma+1}, \omega ^ {\sigma+1}+\omega ^ \sigma) \cup
[\omega ^ {\sigma+1}2, \omega ^ {\sigma+1}2+\omega ^ \sigma)
\cup  [\omega ^ {\sigma+1}3, \omega ^ {\sigma+1}3+10).
\end{equation*}     
If $\alpha_1 \geq \omega ^ \sigma 5 $, then   
\begin{align*} 
A_1  &= [\omega ^ \sigma, \omega ^ \sigma 2) \cup 
[\omega ^ {\sigma+1}+\omega ^ \sigma,
 \omega ^ {\sigma+1}+\omega ^ \sigma 2) \cup
[\omega ^ {\sigma+1}2+\omega ^ \sigma,
 \omega ^ {\sigma+1}2+\omega ^ \sigma 2)
\\
&\cup  [ \omega ^ {\sigma+1}3+10 , \omega ^ {\sigma+1}3 +\omega ^ \sigma)
\cup [\omega ^ {\sigma+1}4, \omega ^ {\sigma+1}4+\omega ^ \sigma)
\cup \dots 
 \end{align*} 
 and so on.
In the informal terminology from the above paragraphs, 
the first two ``blocks'' of the first chain are 
$[0, \omega ^ \sigma)$ and  $[\omega ^ \sigma, \omega ^ \sigma 2) $,
the first two blocks of the second chain are  
  $[\omega ^ {\sigma+1}, \omega ^ {\sigma+1}+\omega ^ \sigma)$
and $[\omega ^ {\sigma+1}+\omega ^ \sigma,
 \omega ^ {\sigma+1}+\omega ^ \sigma 2) $, etc. 

In the above construction we apply in a significant way the condition in the
definition of an ormf realization.
For every $i < \omega$ and each element $\eta$ of, say, the second chain
(provided $\alpha_i$ were long enough), we have 
 $h_i(\eta) \geq \omega ^ \sigma$.
It is true that there are infinitely many indexes $j$  such that  
$\eta$ is $\geq$ some element $\theta \in A_j$
(thus the realization is not left-finite in the sense from \cite{w}),
but either $\theta$ belongs to the first chain,
hence  $h_i(\eta) \geq \omega ^ \sigma > h_j(\theta)$,
so that $j $ does not belong to the set in  \eqref{or},
or, by construction, $j \leq i$; then note that  there are only  a finite number
of indexes $\leq i$. Arguing in this way
for every element of the chain, we see that the realization is
ormf. This concludes the discussion of case (c).

(c$'$) We also discuss the variation of case (c) in which 
the sequence is allowed to contain a finite number of summands equal to 
$\omega ^ \xi$. In detail, we  assume that
 $\alpha_i = \omega ^ \xi$ for finitely many indexes,
and  $\alpha_i < \omega ^ \xi$ for the remaining indexes.
Moreover,  for every $\delta<  \omega ^ \xi$,
there are infinitely many $i < \omega$ such that  
 $ \delta \leq \alpha_i $.
Again, $\xi$ is assumed to be a successor ordinal,
hence case \eqref{seconda} applies; moreover, those summands 
$\alpha_{i_j}$ which are equal to  $\omega ^ \xi$ are 
e-special, but are such that $\alpha_{i_j}^\heartsuit =0$,
hence they do not contribute to 
 $\sums _{i < \omega} \alpha_{i}$. Thus, as in (c),
$\sums _{i < \omega} \alpha_{i} = \omega ^ {\xi +1}$. 

A realization is obtained (this is not the only possible way)
by putting all the summands of length  $\omega ^ \xi$
at the beginning (hence the ormf condition is preserved,
since there are only finitely many such elements)
and then realizing $ \omega ^ {\xi +1}$ as in (c) using the 
remaining elements. 
The elements of length  $\omega ^ \xi$ are absorbed by
$ \omega ^ {\xi +1}$, hence the total length of the realization
is $ \omega ^ {\xi +1}$.

 Another possibility 
is to work  essentially in the same way as in case (c); the only difference
is that summands  equal to  $\omega ^ \xi$ are divided
into infinitely many blocks, which nevertheless are cofinal in 
the realization of 
 $\sums _{i < \omega} \alpha_{i}$, hence they do not modify the outcome.
For example, in the case of the sequence $( \omega, 1, 2, \omega ,3,
4, 5, 6,  \dots, n, \dots)$,  
$A_0 = \{ 0, \omega, \omega 2, \omega 3, \dots \} $, 
$A_1 = \{ 1  \} $,
 $A_2 = \{ 2 , \omega+1\} $,
 $A_3 = \{ 3 , \omega+2, \omega 2 +1, \omega 3+1, \dots \} $,
 $A_4 = \{ 4 , \omega+3, \omega 2+2 \} $, \dots\ 

\emph{Remark.} In passing, we mention that the 
last variation in (c$'$) works also if there
are infinitely many summands equal to 
$\omega ^ \xi$. In other words, \emph{in the case when $\xi$ 
is successor}, if there
are infinitely many summands equal to 
$\omega ^ \xi$, we can equivalently apply  the construction
in (a) or the construction in (c$'$). 
This is not always the case; actually, when $\xi$ is limit
the construction in (c$'$) does not even make sense.
 Of course, we can devise a similar construction,
when $\xi$ is limit, say, $\alpha_i= \omega ^ \omega $, for every $i < \omega$,
and we may set 
$A_0 = [0, \omega ) \cup [ \omega ^2 , \omega ^2 2, )
\cup [ \omega ^3 , \omega ^3 2), \dots  $,
$A_1= [ \omega, \omega 2 ) \cup [ \omega ^2 2 , \omega ^2 3, )
\cup [ \omega ^3 2, \omega ^3 3) $ \dots, ie,
divide each summand in ``blocks'' of increasing
lengths $ \omega$, $ \omega^2$, $ \omega^3$ \dots\    
 But, in such a way, we only realize $ \omega^ \omega $,
not the wanted outcome $ \omega ^{ \omega +1} $.  

 The above remark   shows that the constructions in (a) - (c$'$)
cannot be made uniform.  

Having dealt with the above basic cases, the proof of the theorem 
can be roughly worked out along the lines of the proof of 
Theorem \ref{carr} \cite{Car,neumer,T}.
In any case, we will give full details.
If $\zeta_ {\bm \alpha }=0$,
then all but a finite number of the summands $\alpha_i$
are equal to $0$, hence the result is given
by Carruth  Theorem \ref{carr} itself.
The case when $\zeta _ {\bm \alpha }$ has just one summand
and there is no
special element follows from cases (a) - (c) above.
If   $\zeta _ {\bm \alpha }=   \omega ^ {\xi_s}  + 
\dots + \omega ^ {\xi_1}+ \omega ^ {\xi_0}$ has more than one summand,
we always have countably many elements as long as 
 $ \omega ^ {\xi_s}  + 
\dots + \omega ^ {\xi_1}$,
so we can realize each $\omega ^ {\xi_p +1} $ 
by the methods in (a). We can realize
$\omega ^ {\xi_0}$ or  $\omega ^ {\xi_0+1}$  as in (a) - (c$'$),
according to the corresponding case.
If there are e-special elements, their contributions
should be inserted in the appropriate place, in decreasing order according
to the length of each contribution, 
as in the proof of 
 \cite[Theorem 1, II]{Car}. 

 For example, if 
$  \redver{ \alpha _0 = \omega^7 +1} $,
$  \bluever{ \alpha _1 = \omega ^4+ \omega } $,
$  \greenver{ \alpha_2 = \omega ^3+ \omega ^2 }$, 
$  \alpha _i= \omega ^3+ (i-2) $, for 
$ i \geq 3$, with
$\zeta _ {\bm \alpha }= \omega ^3+ \omega $,
$  \redver{ \alpha_0^ \heartsuit = \omega^7 +1}$,
 $  \bluever{ \alpha_1^ \heartsuit = \omega^4 + \omega } $,
  $\greenver{\alpha_2^ \heartsuit = \omega^2}$, 
we can realize   $\sums _{i < \omega} \alpha _{i}
=   \redver{ \omega ^7} +   \bluever{ \omega ^4} +  \omega ^4 + 
\greenver{\omega^2} + \omega ^ 2+
  \bluever{ \omega } +  \redver{ 1}$ by
\begin{align*}  
\redver{A_0} & = \redver{[0, \omega^7 ) \cup \{\omega ^7 + \omega ^4 2 + \omega ^2 2 + \omega \} },
\\
 \bluever{A_1} &=  \bluever{ [\omega^7, \omega ^7 + \omega ^4 )
 \cup [\omega ^7 + \omega ^4 2 + \omega ^2 2,
\omega ^7 + \omega ^4 2 + \omega ^2 2 + \omega )}
\\ 
\greenver{ A_2} &=
\greenver{ [\omega ^7 + \omega ^4, \omega ^7 + \omega ^4 + \omega ^3 ) \cup 
[\omega ^7 + \omega ^4  2, \omega ^7 + \omega ^4 2 + \omega ^2)}
\\ 
A_3& = [ \omega ^7 + \omega ^4 + \omega ^3 , 
 \omega ^7 + \omega ^4 + \omega ^3 2) \cup
 \{ \omega ^7 + \omega ^4 2 + \omega ^2 \},
\\ 
A_4 &= [ \omega ^7 + \omega ^4 + \omega ^3 2, 
 \omega ^7 + \omega ^4 + \omega ^3 3) \cup 
\{ \omega ^7 + \omega ^4 2 + \omega ^2 +1, 
\omega ^7 + \omega ^4 2 + \omega ^2  + \omega \} , \dots  
\\ 
A_5 &= [ \omega ^7 + \omega ^4 + \omega ^3 3, 
 \omega ^7 + \omega ^4 + \omega ^3 4) \cup  
\\
& \phantom{=\ \ }
\{ \omega ^7 + \omega ^4 2 + \omega ^2 +2, 
\omega ^7 + \omega ^4 2 + \omega ^2  + \omega +1,
\omega ^7 + \omega ^4 2 + \omega ^2  + \omega 2  \},
\\
& \dots  
\end{align*} 
\begin{equation*}
\xymatrix{
 \ar[rr]^{ \omega ^7}_{\redver{A_0}} & & 
 \ar[rr]^{ \omega ^4}_{ \bluever{A_1}} &&
\ar[rr]^{ \omega ^4= \omega ^3 +\omega ^3 + \dots }
_{\greenver{A_2}\ A_3 \ A_4 \dots}
 &&
\ar[rr]^{ \omega ^2}_{\greenver{A_2}} &&
\ar[rr]^{ \omega ^2= \omega + \omega + \dots}_{A_3A_4A_5\dots\  A_4A_5\dots} &&
\ar[r]^{ \omega }_{ \bluever{A_1}}&\ar[r]^{1 }_{\redver{A_0}} &&
}
   \end{equation*}    

A technical detail needs to be mentioned. 
First, observe that if we write $\sums _{i < \omega} \alpha_{i}$
in normal form as  
 $ \omega ^ { \rho _k}  + 
\dots + \omega ^ { \rho _0}$,
some exponents might turn out to be equal, say, $\rho_\ell = \rho _{\ell+1}$.  
In case, say, $ \omega^{ \rho _\ell}$ can be obtained 
both from  some e-special element and from 
$\zeta _ {\bm \alpha }$, say
 $ \omega^{ \rho _\ell} =  \omega ^ {\xi_p+1} $, 
or $   \omega^{ \rho _\ell}= \omega ^ {\xi_0}$, then 
it is preferable, and sometimes necessary,
to realize $\sums _{i < \omega} \alpha_{i}$
by using first the interval coming from the e-special element,
and only after this, use the interval coming from
$\zeta _ {\bm \alpha }$ through the constructions in (a) - (c).
 
For example, if $\bm \alpha $ is
$ \omega ^ \omega , \omega , \omega ^2, \omega ^3, \dots, \omega ^n, \dots$,
we obtain an ormf, actually, left-finite realization of 
$\sums _{i < \omega} \alpha_{i}= \omega ^ \omega 2$
by setting
$A_0= [0, \omega ^ \omega )$,
 $A_1= [ \omega ^ \omega, \omega ^ \omega + \omega  )$,
 $A_2= [\omega ^ \omega + \omega, \omega ^ \omega + \omega + \omega ^2  )=
[\omega ^ \omega + \omega, \omega ^ \omega + \omega ^2  )$,
$A_3=[ \omega ^ \omega + \omega ^2,  \omega ^ \omega + \omega ^3)$, etc.
On the other hand, we have no way to put \emph{all}
the elements of $A_0$ atop all the elements of the $A_i$, for $i >0$,
if we want to maintain the ormf property. In any ormf realization, the
first element of $A_0$, corresponding to $0 \in  \omega^ \omega $,
must be larger than only finitely many elements coming from the
 $A_i$, for $i >0$. With finitely many $A_i$ with $i >0$
we can realize only ordinals strictly smaller than $ \omega^ \omega $,
so that the first element of $A_0$ must have  rank 
$< \omega^ \omega $ in the realization. Repeating the argument for every element of
$A_0$, we get that in an ormf realization, all the elements of $A_0$
must have rank $< \omega^ \omega $. 
We could have put, say,
$A_0= [ \omega ^2, \omega ^ \omega )$,
$A_1= [0, \omega)$,
$A_2= [ \omega ,  \omega^2 )$,
$A_3= [ \omega ^ \omega,  \omega ^ \omega + \omega ^3 )$, \dots,
or even intersperse $A_0$ with the other sets, but in any case 
we cannot go too far, and anyway we get no advantage.

Before giving a formal proof of the general case, we
present a few more examples which show that some subtleties
are necessary in the proof. The reader who, at this point, feels sufficiently
comfortable, might skip the examples and go to the
complete proof below. We first elaborate a bit on examples
similar to the example presented in (c$'$). 

\emph{Example} (1). 
First, consider the sequence
$ \bm \alpha  = ( \omega , 1, 2, 3, \dots)$.
We have $\zeta_{\bm \alpha } = \omega $
and $\alpha_0$ is the only e-special element, but 
$ \alpha _{0}^\heartsuit = 0 $,
hence $\sums _{i < \omega} \alpha _{i} = \omega ^2  $.  
As in the first example
presented in (c) above, we can use the summands from the second
place on, in order to realize $ \omega^2$.  
Where can we insert the elements of the first summand?
The first element of $\alpha_0$ cannot be inserted at the top, if we want
the ormf condition to be satisfied; actually,
the first element of $\alpha_0$ cannot be inserted after
we have used infinitely many elements from the
other summands, thus the first element of $\alpha_0$
must be inserted before we realize  a copy of $ \omega$ in
the realization of  $\omega ^2  $.
Similarly the second element of $\alpha_0$
must be inserted before we realize  a copy of $ \omega 2$. 
The argument shows that, since we want an ormf
realization, we cannot use $\alpha_0$ in order to provide
a realization longer than $ \omega^2$. 

In any case, in order to 
realize $ \omega^2$, the simplest thing to do is to insert all the elements of 
$\alpha_0$ at the beginning, as follows: 
$A_0= [0, \omega ) $,
$A_1=  \{  \omega \} $,
$A_2=  \{  \omega + 1, \omega 2 \} $,
$A_3=  \{  \omega + 2, \omega 2+ 1, \omega 3 \} $\dots \ 
However, we have some freedom on where to put 
the   elements of $\alpha_0$; for example,
$A_0= \{ 3 \} \cup [ \omega  , \omega 2  ) $,
$A_1=  \{ 1 \} $,
$A_2=  \{  2 + 1, \omega 2 \} $,
$A_3=  \{  4, \omega 2+ 1, \omega 3 \} $\dots\  
could have been an alternative possibility. 
We can also let the elements of $\alpha_0$ be cofinal in 
the realization, as in the example in case (c$'$).  

\emph{Example} (1a). 
Now consider the sequence
$ \bm \alpha  = ( \omega + \omega , 1, 2, 3, \dots)$.
We have $\zeta_{\bm \alpha } = \omega $
and $\alpha_0$ is the only e-special element, with 
$ \alpha _{0}^\heartsuit = \omega $,
thus $\sums _{i < \omega} \alpha _{i} = \omega ^2 + \omega $.  
As in (c) above, we can use the summands from the second
place on, in order to realize $ \omega^2$.  
Where can we insert the elements of the first summand
in this case?
As above, the first $ \omega$  elements
 of $\alpha_0$ cannot be inserted at the top, if we want
the ormf condition to be satisfied.
 On the other hand, the last $ \omega$ elements
of $\alpha_0$ can be inserted at the top, respecting ormf.
Thus we realize $ \omega ^2 + \omega$ by letting
$A_0= [0, \omega ) \cup [\omega ^2 , \omega ^2 + \omega)$,
$A_1=  \{  \omega \} $,
$A_2=  \{  \omega + 1, \omega 2 \} $,
$A_3=  \{  \omega + 2, \omega 2+ 1, \omega 3 \} $\dots 
\ Note that the last $ \omega$ elements of $\alpha_0$
(or, at least, an infinite cofinal tail  of them)
\emph{must} be inserted at the end, since otherwise they
are absorbed by (the copy of) $ \omega^2$ and we get a realization of 
$ \omega^2$, not of $ \omega ^2 + \omega$.
  Of course, as in Example (1),
we have some freedom on where to put 
the first $ \omega$ elements of $\alpha_0$; for example,
$A_0= \{ 3 \} \cup [ \omega  , \omega 2  ) \cup [\omega ^2 , \omega ^2 + \omega)$,
$A_1=  \{ 1 \} $,
$A_2=  \{  2 + 1, \omega 2 \} $,
$A_3=  \{  4, \omega 2+ 1, \omega 3 \} $\dots\  
could have been an alternative possibility. However,
we will not need this remark,
in what follows.

\emph{Example} (2). 
The case of  the sequence
$ \bm \alpha  = ( \omega ^2 + \omega  , 1, 2, 3, \dots)$
is simpler.
Again, $\zeta_{\bm \alpha } = \omega $,
 $\alpha_0$ is the only e-special element, but in this case 
$ \alpha _{0}^\heartsuit = \omega ^2 + \omega  = \alpha_0  $,
since $\zeta_{\bm \alpha } + \omega ^2 + \omega = 
\omega + \omega ^2 + \omega = \omega ^2 + \omega
= \alpha _0 $. 
In this case $\sums _{i < \omega} \alpha _{i} =
 \omega ^2  +\omega ^2  +  \omega  $,
we realize the summand $ \omega^2$  coming from   
$ \zeta_{\bm \alpha } \omega $ as usual, and put the first
block of length $ \omega^2$ from $\alpha_0$ at the beginning, 
and the last block of length $ \omega$ at the top.    
  
\smallskip 

We now present the  complete proof
in the general case.
So $\sums _{i < \omega} \alpha_{i} $ is either
\begin{align} \labbel{alli1}  
& \omega ^ {\xi_s+1}  \+ 
\dots \+ \omega ^ {\xi_1+1}
  \+ \omega ^ {\xi_0}
 \+ \alpha _{i_1}^\diamondsuit \+ \dots \+ 
\alpha _{i_h}^\diamondsuit 
&&\text{from \eqref{prima}, or}
\\ \labbel{alli2}  
  & \omega ^ {\xi_s+1}  \+ 
\dots
\+ \omega ^ {\xi_1+1}
  \+ \omega ^ {\xi_0+1}
\+ \alpha _{i_1}^\heartsuit \+ \dots \+ \alpha _{i_h}^\heartsuit  
&&\text{from \eqref{seconda}}  
 \end{align}
from Definition \ref{simpldef}, where we have used Remark \ref{ze}(a).

In both cases,
write $\sums _{i < \omega} \alpha_{i}$
in normal form as  
 $ \omega ^ { \rho _k}  + 
\dots + \omega ^ { \rho _0}$. Note that, since
we are dealing with additive---not Cantor---normal
form, some exponents might be equal, say, $\rho_1 = \rho _0$.  
In order to distinguish the
$\omega ^ { \rho _ \ell}$ in the above finite sum 
from the elements of the sequence $\bm \alpha $,
the $\omega ^ { \rho _\ell}$ will be called S-summands.  
Note that, writing $\alpha _{i_1}^\diamondsuit$\dots\ 
and  $\alpha _{i_1}^\heartsuit$\dots\  in normal form, we get that
each S-summand is either some  $\omega ^ {\xi_\ell+1}$,
possibly $ \omega ^ {\xi_0}$ in case \eqref{prima},
or can be taken to come  from some   $\alpha _{i_p}^\diamondsuit$ 
or from some
$\alpha _{i_p}^\heartsuit$.

The proof is by a finite induction on the value of $k$ 
in $ \omega ^ { \rho _k}  + 
\dots + \omega ^ { \rho _0}$
(the case when $\sums _{i < \omega} \alpha_{i} =0$
is straightforward).

The case $k=0$ has been treated in (a) - (c$'$)
above, with the exception when  
$ \omega ^ { \rho _0}$ is given  by
 some 
$\alpha _{i_p}^\heartsuit$
(if $k=0$ and \eqref{prima} applies, 
 we are necessarily in case (b)).
The case when $k=0$ and 
$ \omega ^ { \rho _0}=\alpha _{i_p}^\heartsuit$
is straightforward, too,
since if 
$\sums _{i < \omega} \alpha_{i} =\omega ^ { \rho _0}$
and $\alpha _{i_p}^\heartsuit =\omega ^ { \rho _0}$
in case \eqref{seconda}, then
 $\zeta_ {\bm \alpha }=0$, hence
 $\alpha_i=0$, for every $i \neq i_p$, since otherwise
there would be more S-summands.  Note that, differently
from case (c), if we modify (a) and (b) by adding some e-special element
to the sequence,
then such an element does contribute to $\sums _{i < \omega} \alpha_{i}$,
hence we can have $k=0$ only in the above-treated cases.

So assume that the theorem holds for all sums
whose value has the normal form $ \omega ^ { \sigma  _{k-1}}  + 
\dots + \omega ^ { \sigma  _0}$, for some fixed $k \geq 1$
and assume that $\sums _{i < \omega} \alpha_{i}=
 \omega ^ { \rho _k}  + \dots + \omega ^ { \rho _0}$. 

We now need to deal with various cases.
Recall that if $\zeta >0$ and 
$\zeta$ is written in normal form as $ \omega ^ {\xi_s}  + 
\dots
+ \omega ^ {\xi_1}
+ \omega ^ {\xi_0}$,  
we have set $\hat \zeta =  \omega ^ {\xi_s}  + 
\dots+ \omega ^ {\xi_1}$.

(i) First, assume that $\sums _{i < \omega} \alpha_{i}$ is 
given by \eqref{alli1} (in particular, $\zeta_ {\bm \alpha } >0$)
 and, furthermore, 
$\omega ^ { \rho _0} = \omega ^ {\xi_0}$,
that is,  $\zeta_ {\bm \alpha }$ contributes to an
S-summand with the smallest exponent.
 Construct another sequence $\bm \beta $ by changing to
$ \hat \zeta_ {\bm \alpha }$ every summand $\alpha_i$
with  $ \hat \zeta_ {\bm \alpha } < \alpha _i <  \zeta_ {\bm \alpha }$
and leaving all the other summands unchanged.
Thus $  \zeta_ {\bm \beta  } =  \hat \zeta_ {\bm \alpha }$.
Indeed, there are arbitrarily large 
(hence, infinitely many) $\alpha_i$ strictly below 
the limit ordinal $\zeta_ {\bm \alpha }$,
since \eqref{prima} applies to $\sums _{i < \omega} \alpha_{i}$,
 hence there are
infinitely many $\beta_i$ equal to  $\hat \zeta_ {\bm \alpha }$.
Moreover,  recall from Definition \ref{simpldef}
that in case \eqref{prima}   there are only finitely many
$\alpha_i \geq \zeta _ {\bm \alpha }  $, 
hence there are only finitely many
$ \beta _i > \hat \zeta_ {\bm \alpha }$.
We have showed that $ \zeta_ {\bm \beta  } = \hat \zeta_ {\bm \alpha } $.
Since in case \eqref{prima} 
$\zeta _ {\bm \alpha } = \varepsilon _ {\bm \alpha }$,
the indexes of the e-special elements are the same
for both sequences,
the elements are equal, and
 $ \beta  _{i_0}^\heartsuit = \alpha _{i_0}^\diamondsuit$, \dots,
$ \beta  _{i_h}^\heartsuit = \alpha _{i_h}^\diamondsuit$,
since the definitions refer 
to, respectively,  $  \zeta_ {\bm \beta  } $ and $ \hat \zeta_ {\bm \alpha }$,
which are equal.
So case \eqref{seconda} applies to $\bm \beta $,
giving  $\sums _{i < \omega} \beta _{i}=
 \omega ^ { \rho _k}  + \dots + \omega ^ { \rho _1}$,
since the only S-summand which gets lost is 
$ \omega ^ {\xi_0}$, equal by assumption  to $\omega ^ { \rho _0}  $.

By the inductive assumption we have an ormf realization 
$R$ of
$\sums _{i < \omega} \beta _{i}$ by the $\beta_i$. 
Consider $R$ as an initial segment of the realization of 
$\sums _{i < \omega} \alpha  _{i}$ we want to construct.
The elements we are left with are the segments
$I_i = [\hat \zeta_ {\bm \alpha }, \alpha _i)$,
for those $\alpha_i$ such that  
$ \hat \zeta_ {\bm \alpha } < \alpha _i <  \zeta_ {\bm \alpha }$.
Since such $\alpha_i$ are unbounded below $  \zeta_ {\bm \alpha }$,
the above segments have lengths   unbounded below 
$ \omega ^ {\xi_0} = \omega ^ { \rho _0}  $, since 
$ \zeta_ {\bm \alpha } = \hat  \zeta_ {\bm \alpha } +  \omega ^ {\xi_0}$. 
Since we are in case \eqref{prima}, $\xi_0$ is limit,
hence, as in (b), we can realize the order-type of   $ \omega ^ {\xi_0}$
by using the intervals $I_i$. Now it is enough to put a copy of this
last realization atop  the realization $R$ given by the inductive hypothesis,
getting a realization of   
$\sums _{i < \omega} \beta _{i} +  \omega ^ {\xi_0}=
 \omega ^ { \rho _k}  + \dots + \omega ^ { \rho _1}+ \omega ^ { \rho _0} 
= \sums _{i < \omega} \alpha  _{i}  $.
The realization is still ormf, since all the elements of the 
$I_i $s are $ \geq  \hat \zeta_ {\bm \alpha }$, while, except for finitely
many indexes, all the elements contributing to  the realization $R$ 
are $ <  \hat \zeta_ {\bm \alpha }$.

(ii) Now assume again that $\sums _{i < \omega} \alpha_{i}$ is 
given by \eqref{alli1}, but this time assume that  
$\omega ^ {\xi_0} > \omega ^ { \rho _0} $.
Note that we allow $\rho _0 = 0$, namely,  $\sums _{i < \omega} \alpha_{i}$
might be a successor ordinal; however, there will be no difference
in the proof between the cases when $\sums _{i < \omega} \alpha_{i}$
is successor or limit.
Since $\omega ^ {\xi_0} > \omega ^ { \rho _0} $, a smallest
S-summand
 $\omega ^ { \rho _0}$ comes from 
 some e-special element $\alpha _{i_p}^\diamondsuit$,
say, $\alpha _{i_p}^\diamondsuit = \alpha _{i_p}^\bullet + \omega ^ { \rho _0} $,
hence
$\alpha _{i_p} = \hat \zeta_ {\bm \alpha } + \alpha _{i_p}^\diamondsuit
 =\hat \zeta_ {\bm \alpha }+ \alpha _{i_p}^\bullet + \omega ^ { \rho _0} $.
Note that we must have 
$\alpha _{i_p} \geq \zeta_ {\bm \alpha }$,
since  $\alpha _{i_p}$ is e-special and, as remarked shortly after
the definition, in case \eqref{prima}
   $ \varepsilon _ {\bm \alpha } = \zeta_ {\bm \alpha }$. 
By the above inequalities, 
$\hat \zeta_ {\bm \alpha }+ \alpha _{i_p}^\bullet + \omega ^ { \rho _0} =
\alpha _{i_p} \geq \zeta_ {\bm \alpha } =
 \hat \zeta_ {\bm \alpha } + \omega ^ {\xi_0}$,
hence
$\alpha _{i_p}^\bullet + \omega ^ { \rho _0}  \geq 
  \omega ^ {\xi_0}$.
Since $\omega ^ {\xi_0} > \omega ^ { \rho _0} $,
necessarily 
$\alpha _{i_p}^\bullet  \geq 
  \omega ^ {\xi_0}$, hence 
$\hat \zeta_ {\bm \alpha }+ \alpha _{i_p}^\bullet \geq
\hat \zeta_ {\bm \alpha }+  \omega ^ {\xi_0}=\zeta_ {\bm \alpha }$. 

Let $ \beta _{i_p}= \hat \zeta_ {\bm \alpha }+ \alpha _{i_p}^\bullet$
and $\beta_i = \alpha _i$ for all $i \neq i_p$.
Case \eqref{prima} still applies,
$\zeta_ {\bm \beta } =  \zeta_ {\bm \alpha }$ and
$ \beta _{i_p}$ is still e-special by the last inequality in the
above paragraph. 
Moreover, by construction, 
$ \beta  _{i_p}^\diamondsuit = \alpha _{i_p}^\bullet$,
so that  
$\sums _{i < \omega} \beta _{i} =
 \omega ^ { \rho _k}  + \dots + \omega ^ { \rho _1}$.
By the inductive assumption we have an ormf realization $R$ 
of $\sums _{i < \omega} \beta _{i}$ by the $\beta_i$.
Just append atop   $R$ a copy of the segment 
$[\hat \zeta_ {\bm \alpha }+ \alpha _{i_p}^\bullet, \alpha _{i_p})= 
\alpha _{i_p} \setminus \beta _{i_p}$ in order to get a
realization of   
$\sums _{i < \omega} \beta _{i} +  \omega ^ {\xi_0}=
 \omega ^ { \rho _k}  + \dots + \omega ^ { \rho _1}+ \omega ^ { \rho _0} 
= \sums _{i < \omega} \alpha  _{i}  $.
The realization is ormf, since we have showed that 
$\hat \zeta_ {\bm \alpha }+ \alpha _{i_p}^\bullet \geq
\zeta_ {\bm \alpha }$, while 
 all the elements contributing to  the realization $R$ 
are $ <   \zeta_ {\bm \alpha }$, except for  
 the finite set of e-special elements.

(iii) Next, assume that $\sums _{i < \omega} \alpha_{i}$ is 
given by \eqref{alli2} and that 
$\omega ^ { \rho _0} = \omega ^ {\xi_0+1}$,
that is,  $\zeta_ {\bm \alpha } \x \omega $ contributes to an
S-summand with the smallest exponent.
In the present case, assume further that 
$\zeta_ {\bm \alpha }$ is successor, thus 
$\xi_0 =0$, 
$\zeta_ {\bm \alpha }=  \hat \zeta_ {\bm \alpha } +1$ and
there are 
infinitely many $\alpha_i$ equal to  $ \zeta_ {\bm \alpha }$.
Let $\beta_i =  \hat \zeta_ {\bm \alpha }$ if
 $ \alpha _i =   \zeta_ {\bm \alpha }$ and $\beta_i = \alpha _i$ otherwise,
thus $ \zeta_ {\bm \beta  } = \hat \zeta_ {\bm \alpha }$
and   \eqref{alli2} still applies to $\sums _{i < \omega} \beta _{i}$. 
The indexes of the e-special elements are the same, 
since no $\beta_i$ is equal to $ \zeta_ {\bm \alpha }$
and $ \varepsilon _ {\bm \alpha } = \zeta_ {\bm \alpha }+1 $,
$ \varepsilon _ {\bm \beta } = \zeta_ {\bm \beta  }+1 =
\hat \zeta_ {\bm \alpha } +1 = \zeta_ {\bm \alpha }$,
by Remark \ref{determ}.
Moreover, 
if $\alpha _{i_p}$ is e-special, then
$\alpha _{i_p} = \zeta_ {\bm \alpha } + \alpha _{i_p}^\heartsuit$,
holding by the definition of   $\alpha _{i_p}^\heartsuit$, 
thus
$ \beta  _{i_p} = \alpha _{i_p} = \zeta_ {\bm \alpha } + \alpha _{i_p}^\heartsuit
= \zeta_ {\bm \beta  } + 1 + \alpha _{i_p}^\heartsuit =
\zeta_ {\bm \beta  } +  \alpha _{i_p}^\heartsuit $. The
last identity follows from the observation that
 $ \alpha _{i_p}^\heartsuit $ is infinite, since 
$ \rho _0 = \xi_0+1$, hence
$\omega ^ { \rho _0} $ is infinite, thus 
 $ \alpha _{i_p}^\heartsuit  \geq \omega ^ { \rho _0}$,
since $ \alpha _{i_p}^\heartsuit$ does contribute to 
$\sums _{i < \omega} \alpha  _{i}$
(as noticed in Remark \ref{ze}(b), the only case in which
some $ \alpha _{i_p}^\heartsuit$ might be $0$ is when 
$ \varepsilon _ {\bm \alpha }$ is limit).

Thus $\sums _{i < \omega} \beta _{i}=
 \omega ^ { \rho _k}  + \dots + \omega ^ { \rho _1} $
and, by the inductive assumption, we have an ormf realization 
$\sums _{i < \omega} \beta _{i}$ by the $\beta_i$.
Just append atop the above realization an
$ \omega$-chain obtained from the singletons remained 
from those infinitely many
$\alpha_i$ equal to  $\zeta_ {\bm \alpha }= \hat \zeta_ {\bm \alpha } +1$,
 which have been
changed to $\beta_i = \hat \zeta_ {\bm \alpha }$.
In this way we ormf realize  $\sums _{i < \omega} \alpha  _{i}=
 \omega ^ { \rho _k}  + \dots + \omega ^ { \rho _1} + \omega ^ { \rho _0}$
by the $\alpha_i$.

(iv) Still  assume that $\sums _{i < \omega} \alpha_{i}$ is 
given by \eqref{alli2}, that 
$\omega ^ { \rho _0} = \omega ^ {\xi_0+1}$,
but in this case
$\zeta_ {\bm \alpha }$ is limit.
Thus we have only finitely many $\alpha_i$
which are   $ > \zeta_ {\bm \alpha }$, and,
possibly, infinitely many $\alpha_i=  \zeta_ {\bm \alpha }$
(hence $ \varepsilon _ {\bm \alpha } = \zeta_ {\bm \alpha }+1$,
by Remark \ref{determ}(g)), or, 
possibly, only finitely many (perhaps none)
such $\alpha_i$, but in this latter case
we have arbitrarily large $\alpha_i < \zeta_ {\bm \alpha }$
(and $ \varepsilon _ {\bm \alpha } = \zeta_ {\bm \alpha }$).  
Note that in the latter case $\xi_0$ is a successor ordinal,
otherwise case \eqref{prima} would apply.  

Construct a sequence $\bm \beta $ by changing to
$ \hat \zeta_ {\bm \alpha }$ every summand $\alpha_i$
with  $ \hat \zeta_ {\bm \alpha } < \alpha _i \leq  \zeta_ {\bm \alpha }$
and leaving all the other summands unchanged,
thus $  \zeta_ {\bm \beta  } =  \hat \zeta_ {\bm \alpha }$, since
  there are only finitely many
$\alpha_i > \zeta_ {\bm \alpha }$
 and, in any case, infinitely many $\alpha_i$ have been moved to 
$ \hat \zeta_ {\bm \alpha }$.
In particular,  
 \eqref{alli2} applies, as well, to the computation of 
$\sums _{i < \omega} \beta _{i}$.
 In fact, the same construction has been performed in case
(iii) above, but we have considered (iii) as a separate case
for simplicity.   

If $ \varepsilon _ {\bm \alpha } = \zeta_ {\bm \alpha } + 1$,
the indexes of the e-special elements are the same for both sequences.
If $ \varepsilon _ {\bm \alpha } = \zeta_ {\bm \alpha }$ 
and $\alpha_i =  \zeta_ {\bm \alpha }$, then $\alpha_i$ is 
e-special with respect to $\bm \alpha$, but    
 $\beta_i$ is not 
e-special with respect to $\bm \beta $, since
$ \beta _i = \hat  \zeta_ {\bm \alpha }=  \zeta_ {\bm \beta  }<
 \varepsilon _ {\bm \beta  }$; however, in this case,
$ \alpha  _{i}^\heartsuit= 0 $, thus $\alpha_i$ does not contribute
to   $\sums _{i < \omega} \alpha  _{i}$, and we can discard
such elements.
In conclusion, the set of those indexes such that 
 $\alpha_i$ or $\beta_i$ do contribute to the 
corresponding sums are the same.
So, assuming that $\alpha  _{i_p}^\heartsuit$ does contribute to 
 $\sums _{i < \omega} \alpha  _{i}$, expressing
$\alpha  _{i_p}^\heartsuit$ in normal form, we have 
$\alpha _{i_p}^\heartsuit = \alpha _{i_p}^\bullet + \omega ^ { \rho _q} $,
for some ordinal $\alpha _{i_p}^\bullet$, possibly, 
$\alpha _{i_p}^\bullet=0$, and with   $\rho _q \geq \xi_0+1$,
by the assumption that 
$\omega ^ { \rho _0} = \omega ^ {\xi_0+1}$.
Thus $\alpha  _{i_p} =
  \zeta_ {\bm \alpha } + \alpha  _{i_p}^\heartsuit > \zeta_ {\bm \alpha }$, 
hence
$ \beta  _{i_p} = \alpha  _{i_p} =
  \zeta_ {\bm \alpha } + \alpha  _{i_p}^\heartsuit=
\hat \zeta_ {\bm \alpha } +  \omega ^ {\xi_0} +
 \alpha _{i_p}^\bullet + \omega ^ { \rho _q}=
 \zeta_ {\bm \beta  } + 
 \alpha _{i_p}^\bullet + \omega ^ { \rho _q}$, since 
in the above \emph{ordinal}  sum $ \omega ^ {\xi_0}$
is absorbed by either  $ \alpha _{i_p}^\bullet $
or $ \omega ^ { \rho _q}$ because
 $\rho _q \geq \xi_0+1 > \xi_0$. 
From $ \beta  _{i_p} = \zeta_ {\bm \beta  } + 
 \alpha _{i_p}^\bullet + \omega ^ { \rho _q}$
we get 
$ \beta   _{i_p}^\heartsuit=\alpha _{i_p}^\bullet + \omega ^ { \rho _q}
=  \alpha  _{i_p}^\heartsuit$.
In conclusion, elements with the same indexes contribute to 
the corresponding sums, and they contribute in the same way.
This means that
$\sums _{i < \omega} \beta _{i}=
 \omega ^ { \rho _k}  + \dots + \omega ^ { \rho _1} $,
since $\zeta_ {\bm \alpha   } = \zeta_ {\bm \beta  } + \omega ^ {\xi_0}$
and $\omega ^ { \rho _0} =\omega ^ {\xi_0+1}$.

Now the proof proceeds in a way similar to case (i).
By the inductive assumption we have an ormf realization 
$R$ of
$\sums _{i < \omega} \beta _{i}$ by the $\beta_i$. 
If $ \varepsilon _ {\bm \alpha } = \zeta_ {\bm \alpha }+1$,
we have  infinitely many $\alpha_i $ 
equal to $   \zeta_ {\bm \alpha }$
which have been changed to $\beta_i = \hat \zeta_ {\bm \alpha }$.
Thus we are left with countably many ``tails''
$[\hat \zeta_ {\bm \alpha }, \alpha _i)$    of length $ \omega ^ {\xi_0}$
which can be put together as in (a) in order to get a chain
of length $ \omega ^ {\xi_0+1}$.
Adding this chain atop $R$ we get a realization of 
$\sums _{i < \omega} \alpha  _{i}$.
If, on the other hand, $ \varepsilon _ {\bm \alpha } = \zeta_ {\bm \alpha }$,
we have  at most finitely many $\alpha_i $ 
equal to $   \zeta_ {\bm \alpha }$, but
 arbitrarily large $\alpha_i < \zeta_ {\bm \alpha }$,
thus we are left 
with segments $[\hat \zeta_ {\bm \alpha }, \alpha _i)$ 
   of arbitrary large lengths $ < \omega ^ {\xi_0}$.
As mentioned above, $\xi_0$ is successor, otherwise we are
in case \eqref{prima}, contrary to our assumption;
thus we can construct the final chain of the desired 
realization arguing as in (c$'$) above.  In both cases
the final realization is ormf, since all but finitely many 
$\beta_i$ are $\leq   \zeta_ {\bm \beta }=   \hat \zeta_ {\bm \alpha  }$,
while in the   added chain we always use elements
$\geq \hat \zeta_ {\bm \alpha  }$.

(v) The final case is when $\sums _{i < \omega} \alpha_{i}$ is 
given by \eqref{alli2} and  
$\omega ^ {\xi_0+1} > \omega ^ { \rho _0} $, thus a smallest
S-summand
 $\omega ^ { \rho _0}$ comes from 
 some e-special element $\alpha _{i_p}^\heartsuit$,
say, $\alpha _{i_p}^\heartsuit = \alpha _{i_p}^\bullet + \omega ^ { \rho _0} $,
hence
$\alpha _{i_p} =  \zeta_ {\bm \alpha } + \alpha _{i_p}^\heartsuit
 = \zeta_ {\bm \alpha }+ \alpha _{i_p}^\bullet + \omega ^ { \rho _0} $,
possibly, with $ \alpha _{i_p}^\bullet =0$.
Let $ \beta _{i_p}=  \zeta_ {\bm \alpha }+ \alpha _{i_p}^\bullet$
and $\beta_i = \alpha _i$ for all $i \neq i_p$.
Case \eqref{prima} still applies,
$\zeta_ {\bm \beta } =  \zeta_ {\bm \alpha }$ and
$ \beta _{i_p}$ is still e-special,
unless possibly when $ \alpha _{i_p}^\bullet =0$. 
Since $\zeta_ {\bm \beta } =  \zeta_ {\bm \alpha }$,
$ \beta  _{i_p}^\heartsuit = \alpha _{i_p}^\bullet$,
so that  
$\sums _{i < \omega} \beta _{i} =
 \omega ^ { \rho _k}  + \dots + \omega ^ { \rho _1}$.
As custom by now, by the inductive assumption we have an ormf realization 
 $R$ of $\sums _{i < \omega} \beta _{i}$ by the $\beta_i$.
If we append atop $R$   a copy of the segment 
$[ \zeta_ {\bm \alpha }+ \alpha _{i_p}^\bullet, \alpha _{i_p})= 
\alpha _{i_p} \setminus \beta _{i_p}$, we get a
realization of   
$\sums _{i < \omega} \beta _{i} +  \omega ^ { \rho _0}=
 \omega ^ { \rho _k}  + \dots + \omega ^ { \rho _1}+ \omega ^ { \rho _0} 
= \sums _{i < \omega} \alpha  _{i}  $.
As in the above cases, the realization is ormf, since  
$ \zeta_ {\bm \alpha }+ \alpha _{i_p}^\bullet \geq
\zeta_ {\bm \alpha }$, while 
 all the elements contributing to  the realization $R$ 
are $ <   \zeta_ {\bm \alpha }$, except for the elements belonging
to the finite set of e-special elements.

With the previous case we have covered all
the possible cases, hence we have concluded the proof that
 $\sums _{i < \omega} \alpha_{i}$
can be realized as an  ormf mixed sum of the $\alpha_i$.

It remains to show that 
that if  $\gamma$
is a realization of $( \alpha _i) _{i < \omega} $ 
as an ormf mixed sum, then
 $ \gamma \leq \sums _{i < \omega} \alpha _{i}$.
This is almost immediate from Lemma \ref{orf}. 

Suppose by contradiction that 
$\gamma> \eta =\sums _{i < \omega} \alpha _{i} $
and $\gamma$ is an ormf realization of
$( \alpha _i) _{i < \omega} $.
Choose the counterexample in such a way that $\eta$ 
is minimal and define the $B_i$ and the
$\beta_i$   as in Lemma \ref{orf}. 
By construction, 
$\eta =\sums _{i < \omega} \alpha _{i} $
is realized as an ormf mixed sum
of $( \beta_i) _{i < \omega} $.  
By  Lemma \ref{orf}, 
$\sums _{i < \omega} \beta  _{i}
< \sums _{i < \omega} \alpha _{i} = \eta  < \gamma $
and this contradicts the minimality of $\eta$,
since we have got a counterexample with 
$\gamma'= \eta > \eta' = \sums _{i < \omega} \beta  _{i}$. 
 \end{proof}

\begin{remark} \labbel{megl}
Recall that we have denoted by $\suma$ the operation
obtained by always using Clause \eqref{seconda} in 
Definition \ref{simpldef}.
 
The first situation in which we apply \eqref{prima},
instead, that is,  the first case in which $\suma$ and $\sums$
differ, occurs when $\zeta _{\bm \alpha }= \omega ^ \omega  $, say,
for the sequence given by $\alpha_i = \omega ^i$,
for which $\sums _{i < \omega} \alpha_{i}= \omega ^ \omega $,
while $\suma _{i < \omega} \alpha_{i} = \omega ^{ \omega +1} $.   
 
As we have seen in Theorem \ref{oc},    $\omega ^ \omega$
can be realized as a mixed sum of the $\alpha_i$s
 in such a way that some loose finiteness condition---the
ormf condition---is satisfied.
As a matter of fact, also  $\omega ^ {\omega+1}$
can be realized as a mixed sum of the $\alpha_i$s:
just partition $ \omega$ into countably many disjoint
countable subsets and, for each such subset, use those $\alpha_i$
whose indexes belongs to the subset in order to realize
a copy of $ \omega^ \omega $, as in case  (b)
in the proof of Theorem \ref{oc}.
If we append such realizations one after the other in
a chain of length $ \omega$, we get  
 $\omega ^ {\omega+1}$.
However, we believe that such a realization is not
sufficiently ``natural''; in fact, by rearranging the above 
partial realizations into a chain of length $\delta$,
for $\delta$ an infinite countable ordinal,  
we get a realization of $ \omega^ \omega \delta $, that is,
we get arbitrarily large countable ordinals.
In any case, all the above realizations do not satisfy the ormf 
property, which itself is a rather weak request.
Moreover, as we will see in Appendix I,
$\sums$ has also a rather simple game-theoretical characterization
(essentially, a reformulation of the other order-theoretical 
characterization we are going to present in the next section). 
 \end{remark}

\section{A characterization of $\sums $ 
as the rank in  a well-founded order} \labbel{rank}

\begin{definition} \labbel{rankd}
We introduce a strict  partial order
$\prec$ on $ \omega$-indexed sequences of ordinals.  
Recall that we write $\bm \alpha$ for $( \alpha  _i) _{i < \omega} $.
We let $\bm \beta \prec \bm\alpha$ 
if 
there are some $\bar \imath < \omega$, 
some finite  set $F \subseteq  \omega$ 
 and some $\beta $ 
such that 
\begin{equation} \labbel{defgg}
\begin{aligned} 
\beta_{\bar \imath} &\leq \beta < \alpha _{\bar \imath},
\\ 
\beta_i &\leq  \alpha _i, && \text{for every   $i \in  \omega $,} 
\\
\beta_i & \leq  \beta,   && \text{for every $ i  \notin F$.}
\end{aligned}
\end{equation}      
Roughly, all the ``large'' elements of 
$\bm\alpha$ are made $ \leq \beta$,
except possibly for a finite set of elements;
moreover, at least one element---corresponding to the index 
 $\bar \imath $---strictly decreases to an ordinal $ \leq \beta$. 
\end{definition}   

\begin{lemma} \labbel{wf}
The relation $\prec$ is a strict well-founded order
on the class of $ \omega$-indexed ordinal sequences. 
 \end{lemma} 

\begin{proof}
If $\bm \beta \prec \bm \alpha$, then, for every $i < \omega$,  
 $\beta_i\leq \alpha_i $. Moreover, 
$\beta_ {\bar \imath} < \alpha_{\bar \imath} $.
This implies that $\prec$ is irreflexive. 
As for transitivity, let $\bm \gamma  \prec \bm \beta $
be witnessed by 
$\bar \jmath $, 
 $G $ and
 $ \gamma  < \beta  _{\bar \jmath}$.
Then 
$\bm \gamma  \prec \bm \alpha  $ is witnessed
by  $F \cup G$, $\min \{\beta, \gamma  \}$
and $\bar \imath$, if $\beta \leq \gamma $, or
$\bar \jmath$, if $ \gamma  \leq \beta  $.

Finally, we check that $\prec$ is well-founded.
Indeed, if $\bm \beta   \prec \bm \alpha  $
is witnessed by some $\beta$ and $F$, then  
$ \zeta_{\bm \beta} \leq \beta $, since $F$ is finite.
This implies that if the chosen $\bar{\imath}$ 
is such that $\alpha_ { \bar{\imath} } \leq \zeta_{\bm \alpha }$, 
then   $ \zeta_{\bm \beta } < \zeta_{\bm \alpha } $.
Since there are only a finite number of z-special elements
(that is, summands $> \zeta_{\bm \alpha }$) and 
there is no infinite decreasing sequence on an ordinal,
if we are given an infinite $\prec$-decreasing chain to which
$\bm \alpha$ belongs, then in a finite number of steps below
$ \bm \alpha$ there is a $\prec$ inequality  
witnessed by some 
 $ \delta _{ \bar{k} } \leq \zeta_{\bm \alpha }$. As we have 
mentioned, this implies that 
the value of  $\zeta$
 strictly decreases. Since $\zeta$ is itself an ordinal and, again since
there is no infinite decreasing sequence of ordinals, 
after a finite number of steps  we reach a sequence with
$\zeta =0$. This means that $ \zeta $ cannot be decreased further
and that the sequence is constantly $0$, with the possible exception
of a finite number of elements. But each $\prec$-down step
decreases at least one such element, hence in the end we 
necessarily reach
the constantly $0$ sequence after a finite number of steps.  
This shows that there is no infinite $\prec$-decreasing sequence. 
 \end{proof}    

Since $\prec$ is well-founded, each sequence $\bm \alpha$
of ordinals has a rank (relative to $\prec$). We will denote
this rank by   $v(\bm \alpha)$. 

\begin{lemma} \labbel{ranklem}
Suppose that $\sumt$ is an infinitary operation which 
is strictly monotone on z-special elements, that is, 
$\sumt$ satisfies Clause \eqref{s3}.

Then, for every 
sequence $ \bm \alpha = ( \alpha_i) _{i < \omega} $ of ordinals,
 $v(\bm \alpha) \leq \sumt _{i < \omega} \alpha_{i}$.
\end{lemma}

 \begin{proof} 
By induction on $v(\bm \alpha)$.
 (The induction can be equivalently carried
over on $\sumt _{i < \omega} \alpha_{i}$.)   
 If 
$v(\bm \alpha)=0$, then  
the conclusion is straightforward, since $0$ is the smallest ordinal. 
 
Otherwise, fix some nonzero sequence $\bm \alpha$ and 
suppose that the lemma holds for every 
sequence $\bm \beta $ with 
$v(\bm \beta ) < v(\bm \alpha)$.
As we have remarked in the proof of Lemma \ref{wf},
if $\bm \beta \prec \bm \alpha$, as witnessed by 
$\bar \imath$, then either $ \alpha _{\bar \imath}$
is z-special, or   $ \zeta_{\bm \beta } < \zeta_{\bm \alpha } $.
This implies that  
$\sumt _{i < \omega} \beta _{i} < \sumt _{i < \omega} \alpha_{i}$,
since, by assumption,  $\sumt$ is 
 strictly monotone on z-special elements,
hence also   strictly monotone on $\zeta$, by Remark \ref{e}.

Thus, by the inductive hypothesis, 
if $\bm \beta \prec \bm \alpha$, then
$ v(\bm \beta ) \leq
\sumt _{i < \omega} \beta _{i} < \sumt _{i < \omega} \alpha_{i}$.
By the definition of rank, $v(\bm \alpha )$ is the smallest ordinal which is
strictly larger than every $v(\bm \beta )$ such that   $\bm \beta \prec \bm \alpha$.
Thus $ v(\bm \alpha ) \leq \sumt _{i < \omega} \alpha_{i}$.
\end{proof}

\begin{theorem} \labbel{thmgamr}
For every sequence $ \bm \alpha = ( \alpha_i) _{i < \omega} $ of ordinals,
$v(\bm \alpha) = \sums _{i < \omega} \alpha_{i}$.
 \end{theorem}

 \begin{proof} 
The inequality $v(\bm \alpha) \leq \sums _{i < \omega} \alpha_{i}$
follows from Lemma \ref{ranklem}, since $\sums$
is strictly monotone on z-special elements, by the proved part of Theorem \ref{simpth}.

In order to prove the converse inequality, we will use Theorem \ref{oc}.
In particular,  
$\sums _{i < \omega} \alpha_{i}$ has an ormf realization
by certain sets  $( A_i) _{i < \omega} $.
Recall that this means that  $\sums _{i < \omega} \alpha_{i}$
is the disjoint union of the $ A_i $,
each $A_i$ has order-type $\alpha_i$, with the 
order-preserving bijection denoted
by $h_i$,  and furthermore  the realization is ormf, namely,
 for every $ \eta  < \sums _{i < \omega} \alpha_{i}$,
$\eta \in A_i$,  
there is a finite set $F_ \eta  \subseteq \omega $ such that if 
$ \theta  < \eta  $,
$ \theta   \in A_j$, 
 and $h_j( \theta  ) \geq h_i ( \eta  )$,
then $j \in F_ \eta $
(it is no loss of generality to assume that $i\notin F_ \eta$).
Note that, given $\eta$ as above, 
$D_i =\eta \cap A_i$ ($i < \omega$) is an ormf realization of $\eta$ 
itself, so that, by Theorem \ref{oc}, if 
we let $ \delta _i$  be the order-type of  
$\eta \cap A_i$, for $i < \omega$,
then $\eta \leq \sums _{i < \omega} \delta  _{i}$.
Compare Lemma \ref{orf}.

The proof that $\sums _{i < \omega} \alpha_{i} \leq v(\bm \alpha)$,
for every sequence $\bm \alpha$, 
is by induction on
$\sums _{i < \omega} \alpha_{i}$. As above, the base case  
$\sums _{i < \omega} \alpha_{i} =0$
is straightforward.

Otherwise, fix some nonzero sequence $\bm \alpha $
and assume that $ \sums _{i < \omega} \beta _{i} \leq v(\bm \beta )$,
for every sequence $\bm \beta $ such that 
$ \sums _{i < \omega} \beta _{i} < \sums _{i < \omega} \alpha_{i}$.  
We need to show that for every $ \eta  < \sums _{i < \omega} \alpha_{i}$,
there is a sequence $ \bm \beta \prec \bm \alpha $ 
such that   $ \eta  \leq v(\bm \beta ) $.
Given $ \eta $ and an ormf realization of 
$\sums _{i < \omega} \alpha_{i}$ as above,
let $\bar{\imath}$    be the unique index 
such that  $\eta  \in A_ { \bar{\imath} }$.
Let $\bm \beta$ be defined by setting
$ \beta  _{\bar \imath} =  \beta = h_{\bar{\imath}} ( \eta  )$,
$\beta_i =  \alpha _i$, if  $i \in F_\eta$, and 
$\beta_i =  \min\{ \beta, \alpha _i \} $, otherwise.
In particular $\bm \beta \prec \bm \alpha$. 
Moreover, by the argument in the proof of
Lemma \ref{ranklem}, we have  
$\sums _{i < \omega} \beta _{i} < \sums _{i < \omega} \alpha_{i}$,
since
$\sums$ is strictly
monotone on z-special elements.

Recall that, for $j < \omega$,  $ \delta _j$  is the order-type of  
$\eta \cap A_j$.
For every $j < \omega$, we have $\delta_j \leq \beta _j$,
since if  $j \in F_ \eta $, then $  \delta _j \leq  \alpha_j =  \beta _j $.
Otherwise,  if
$ \theta   \in \eta \cap  A_j$, 
 then $h_j( \theta  ) < h_{\bar{\imath}} ( \eta  )$,
thus $\delta_j $, being the order-type of 
$\eta \cap A_j$, is $\leq h_{\bar{\imath}}( \eta) = \beta $;
but also $\delta_j  \leq \alpha _j$, thus $\delta_j \leq \beta _j$.
By the already proved weak monotonicity of $\sums$,
$\sums _{i < \omega} \delta _{i} \leq \sums _{i < \omega} \beta _{i}$,
hence 
$\eta \leq \sums _{i < \omega} \delta  _{i} \leq 
\sums _{i < \omega} \beta _{i} = v(\bm \beta )$,
by the inductive assumption, since 
$\sums _{i < \omega} \beta _{i} < \sums _{i < \omega} \alpha_{i}$,
as already remarked.  
\end{proof}

\begin{proof}[Proof of Theorem \ref{simpth} (continued)]
So far,  we have proved that 
 $\sums$ is  
strictly monotone on z-special elements. 
If $\sumt$ is another operation satisfying this
property, then
$\sums _{i < \omega} \alpha_{i} =
v(\bm \alpha) \leq  \sumt _{i < \omega} \alpha_{i}$,
for every sequence $\bm \alpha $, by Theorem \ref{thmgamr}
and Lemma \ref{ranklem}.
Hence $\sums$ is the smallest operation with the
property.    
\end{proof}

The ordinal $v(\bm \alpha)$ can be interpreted as the value of some
combinatorial game 
\cite{Sieg} 
played on sequences of ordinals, where
the Left player moves on $\bm \alpha $
 choosing some sequence $\bm \beta \prec \bm \alpha$.  
This leads to the problem of extending  $\sums$ 
 to sequences of Conway surreal numbers or, possibly, even 
sequences of combinatorial games.
See Appendix I below.

\section{Further Remarks} \labbel{fur}

We have seen that $\sums$ is 
strictly monotone on
z-special elements. 
However, as we have showed in
Theorem \ref{thm}.6.a, it is possible
to have strict monotonicity on
e-special elements, and this is 
a slightly stronger  condition.
In the next example we show that
$\sums$
does not satisfy this condition.

 \begin{remark} \labbel{snote}   
(a) The operation $\sums$
is not  
strictly monotone on
e-special elements (recall Condition \eqref{s2}).

Indeed,
$ \sums (0,1,2,3, \dots ) =  \sums ( \omega ,1,2,3, \dots )
= \omega ^2 $.
For both sequences 
$\varepsilon = \omega $,
but $0= \alpha_0 < \varepsilon  = \beta _0 $,
where the elements of the first sequence are $\alpha_i$
and the elements of the second sequence are $ \beta _i$.
The point is that 
$\beta_0 = \omega $
in the second sequence is 
e-special but not
z-special.     
 
Similarly
$\sums$
is not  
strictly monotone on $\varepsilon$  (compare
Remark \ref{e}).
Indeed,
$ \sums (0,1,2,3, \dots ) =  \sums ( \omega , \omega , \omega  \dots )
= \omega ^2 $, but 
$\varepsilon _ {\bm \alpha }= \omega < \omega +1 =\varepsilon _ {\bm \beta  }$
(though 
$ \zeta  _ {\bm \alpha }= \zeta  _ {\bm \beta  } =  \omega $). 

(b) The above counterexamples apply also to $\suma$, as introduced
shortly before Lemma \ref{lemmo}. In particular,  Proposition \ref{ops2}
cannot be improved.  

(c) As explained in Section \ref{fobs}, under some reasonable requests,
we cannot expect
strict monotonicity for an infinitary sum operation;
on the other hand, again by the arguments
in Section \ref{fobs},  $\sums$ appears to be quite natural,
since it retains many instances of strict monotonicity.
Nevertheless, one could argue that the other operation
$\sumn$, whose existence has been proved in Theorem \ref{thm}.6.a,
might be the most natural operation, under the aspects taken into 
consideration here. We have not devised an explicit
construction of  $\sumn$ yet, but the arguments  
in (a)  imply that  
$ \sumn ( \omega ,1,2,3, \dots )
> \omega ^2 $. Actually, it is not difficult to check that  
$\sumn$ and  $\sums$ agree on any sequence $\bm \alpha$ 
with finite $\varepsilon _ {\bm \alpha }$.
This implies that  
$ \sumn ( \omega ,1,2,3, \dots )
= \omega ^2 +1 $,
$ \sumn ( \omega , \omega ,2,3, \dots )
= \omega ^2 +2 $, \dots,
$ \sumn ( \omega , \omega , \omega , \dots )
= \omega ^2 + \omega  $, outcomes which at
first sight do not appear very natural, so that we are still
convinced that, in the above respects, $\sums$ 
is the most convenient operation.

While the explicit definition of $\sums$ 
is somewhat involved, the fact that $\sums$
has   order-theoretical  and
 game-theoretical characterizations suggests
the naturalness of the operation. 
 \end{remark}

In any case, a thorough study of $\sumn$ is missing.

\begin{problem} \labbel{probn}
Give an explicit definition of $\sumn$
and study this operation. 
 \end{problem}

\begin{remark} \labbel{compar}
(a) Comparing $\sums$  with the operation $\nsum$  studied in \cite{w}
(the countable case of $\sumh $ as presented in the introduction here),
we may say that $\sums$ retains most instances  of 
strict monotonicity  at the expense of instances of associativity,
 while, on the other hand, $\nsum$
 retains as much of commutativity-associativity as possible
at the expense of instances of strict monotonicity.

For example, 
we can weakly associate 
$\nsum _{i < \omega} 1 = 1 +1 +1 +1 \dots  = (1 +1) +(1 +1) +  \dots 
= 2+ 2+ \dots = \nsum _{i < \omega} 2$.
See \cite[Proposition 2.4(6)]{w} for a formal statement. 
However, the above identities strongly break strict 
monotonicity.
On the other hand,
$\sums _{i < \omega} 1  =
 \omega \ne \omega 2 =
\sums _{i < \omega} 2 $. 

(b) As a related remark, 
$\sums$ satisfies many instances of distributivity.
Suppose that  $\delta$ is any ordinal, 
$\bm \alpha $ is a sequence such that  there are infinitely many   
$i < \omega$ such that $ \alpha _i = \zeta _{ \bm \alpha } $
(thus \eqref{seconda} applies) and 
$ \alpha _{i_\ell} =\zeta_ {\bm \alpha} \+ \alpha _{i_ \ell}^\heartsuit $
(natural sum), for every e-special element $\alpha _{i_ \ell}$. 
Then 
$\sums _{i < \omega} ( \delta \x \alpha  _i) =
\delta \x \sums _{i < \omega} \alpha  _i$. 

However, distributivity with respect to the H-natural product
 does not always hold:
for example, evaluating using $\sums$, 
$1+2+3+4+ \dots = \omega ^2$, but also 
$2+4+6+8+ \dots \dots = \omega ^2$.
Similar counterexamples
necessarily present themselves under minimal assumptions,
as we will show in  Remark \ref{perforza}.
\end{remark}   

\begin{remark} \labbel{annon} 
As another property which fails for
 $\sums$,
it is not always the case that 
\begin{equation}\labbel{sump}    
  \sums (\alpha_i \+ \beta_i) = 
\big(\sums \alpha_i \big)  \+ \big(\sums \beta_i\big ) 
   \end{equation}
Again, taking
$ \alpha_i=\beta_i=i$, we have 
  $\sums (\alpha_i \+ \beta_i) = \omega ^2  \neq \omega ^2 2
= \left( \sums \alpha_i  \right ) \+ \left( \sums \beta_i \right )$.
As another counterexample, 
if $  \alpha _0 =  \omega^3 $,  
$ \beta _0=1$,  $ \alpha_i= \beta_i = \omega$,
for $i > 0$, then
$\left( \sums \alpha_i  \right ) \+ \left( \sums \beta_i \right ) =
(\omega^3 + \omega^2) \+ \omega^2 \neq
\omega^3 + \omega^2 2 +1 =
\sums (\alpha_i \+ \beta_i)$.

Probably it is worthwhile to investigate
those special cases under 
 which
\eqref{sump} or distributivity 
with respect to products actually hold. 
\end{remark}

\begin{remark} \labbel{perforza}
Suppose that $\bigoplus$
is an everywhere defined
 weakly monotone  
 operation on countable sequences of ordinals
with the property that enlarging a 
sequence by inserting any number, possibly infinite, of $0$'s
does not change the value of  
$\bigoplus$
(if this happens, we might say that $0$ is
\emph{fully neutral}; note that this is a stronger notion
with respect  to being neutral, as defined in Remark \ref{no}(c)).   

If 
$( \alpha_i) _{i < \omega} $ 
and $( \beta _i) _{i < \omega} $ 
are mutually cofinal sequences,
then necessarily
$\bigoplus_{i < \omega} \alpha_{i}= 
\bigoplus _{i < \omega} \beta _{i}$.
Indeed, if necessary, add a sufficient number of $0$'s 
at the beginning of the second sequence
in such a way that the first non-null
element of $( \beta _i) _{i < \omega} $, now at place, say,
$i_0$,   
is bounded by $\alpha_{i_0}$. Afterwards, if necessary, 
 add further $0$'s in such a way  
that the second non-null
element of $( \beta _i) _{i < \omega} $, now moved at place
$i_1$,   
is bounded by $\alpha_{i_1}$
and go on the same way.
By monotonicity and invariance under adding $0$'s,
we get  
$\bigoplus_{i < \omega} \alpha_{i}\geq 
\bigoplus _{i < \omega} \beta _{i}$
and the symmetric argument provides the converse inequality.
 
Note that the present argument has little to do with 
ordinals, it works for every partially ordered set 
with some specified element  $0$ 
and for sequences whose members are $\geq 0$.

Thus, under the assumption that 
$\bigoplus$ is monotone and invariant under inserting $0$'s,
we get the quite counterintuitive fact that $\bigoplus$
is necessarily evaluated the same way
at the sequences, say, $(1,2,2,3,3,3,4,4,4,4, \dots)$
and $(2^2,3 ^{3^3}, 4 ^{4^{4^4}}, \dots )$.  
 \end{remark}

\begin{problem} \labbel{transf}
Find a generalization of the definitions and of the results
presented here to sequences indexed by
ordinals $ > \omega $ or, possibly, 
indexed by unordered sets. Compare \cite{t}.
 \end{problem}

\begin{acknowledgement} 
We thank anonymous referees of the present paper and of 
 \cite{w,t,infnp}
for many interesting suggestions.
We thank Harry Altman for stimulating discussions.
 \end{acknowledgement}

\section{Appendix I. A game-theoretical rewording} \labbel{game} 

We now provide a characterization of $\sums$ in terms of
combinatorial game theory. We refer to 
\cite[Chapter VIII]{Sieg} for   basic notions, though here we will need
really little prerequisites and we will try to make the paper as 
self-contained as possible. The reader familiar with
combinatorial game theory might skip the next three subsections.
To make a long story short,
what we will use is that  
a move on the ordinal $\alpha$ simply amounts to choosing some
ordinal $\gamma < \alpha $. If $\alpha=0$, no move is possible and the 
player looses the play at that point. 

\subsection{Brief hints to various notions of games} \labbel{comparg} 
Classical combinatorial games such as Go, Draughts,
Chess, Tic-Tac-Toe have been played since ancient times,
but the general theory is quite recent. We refer to 
\cite{Sieg} for some history.
In short, a \emph{combinatorial game} is a possibly infinite,
win-lose, two-player alternating partisan
game with perfect information, no chance element, no draw
and no infinite run. See below for more details.
A special characteristic of combinatorial games,
which make them different from other kinds of games studied
in the classical setting,
 is that in combinatorial games
there are no special payoffs, apart from winning and losing.
However, this is no essential loss of generality,
since J. Conway has proved that each game can be assigned a 
\emph{value}, which can be, say, a real number, or even an ordinal.
In general, values of (possibly infinite)
 combinatorial games include all the \emph{surreal numbers}, 
 a Field which extends the real numbers and
contains the class of all the ordinals.
There are also values which are ``not  numbers'',
but here we will deal with these  only marginally.
 
Combinatorial games are different also from those games
profitably studied in set theory,
e.~g. \cite{KM} in the reference list below.
One difference is that in combinatorial games the losing player
is the player who has no available move (this is the \emph{normal
play};   in the \emph{mis\'ere} play the losing player is the ome who
makes the last move; we will always deal with the normal play).
On the other hand, generally, in  
set theoretical games there is a set
of winning positions, and say, the First player wins if the succession
of moves belongs to such a winning set.
Moreover, set theoretical games generally include the possibility
of making an infinite sequence of moves, while, in the classical
situation, any combinatorial game necessarily ends after a finite
(but not necessarily bounded in advance) sequence
of moves (anyway, there is a somewhat deep combinatorial theory
of \emph{loopy games}, but we will not need it here.
 We refer again to  \cite{Sieg}).

Note that, strictly formally, some games used in Model
Theory (for example, classical finite Ehrenfeucht-Fra\"\i ss\'e
games) are for all intents combinatorial games; however, they are 
generally  treated in
a different spirit. By the way, some authors have 
 attempted to apply some notions
from combinatorial game theory to Ehrenfeucht-Fra\"\i ss\'e 
games, see e.~g., \cite{MV}.
It should be also mentioned that some ideas used in 
combinatorial game theory have independently evolved in a different
setting \cite{Bl}  with subsequent semantical applications
 \cite{Hyl}.

\subsection{A really concise introduction to combinatorial games} \labbel{brieintr} 
As we mentioned, a \emph{combinatorial game} is a possibly infinite,
win-lose, two-player alternating partisan
game with perfect information, no chance element, no draw
and no infinite run. \emph{Partisan} means that 
the sets of  moves are possibly distinct for the two players.

For our purposes, combinatorial games are better introduced as
some kind of  ``bilateral'' sets\footnote{For an explicit axiomatization,
see \cite{CK}.
See \cite{But} for still another implementation.}.
Namely, one constructs games in a way similar
to set construction in the von Neumann hierarchy,
  but now we have two membership relations,
``left''  membership and  ``right''  membership.
Combinatorial games are considered as such ``bilateral'' sets.
In a way similar to the standard set-theoretical notation, games can be denoted
by the notation $\{ \, \dotsc  \mid \dots  \,\}$, where the left members
of the game are listed on the left, and the right members of the game 
are listed on the right.
The two players are conventionally called \emph{Left}  and \emph{Right};
the left elements of some game are considered as the possible moves for the 
Left player, and 
the right elements of the game are considered as the possible moves for the 
Right player.

``At the beginning'' (called ``Day 0'' by J. Conway,
and corresponding to rank $0$ in the von Neumann
hierarchy) there is only the empty game $0= \{ \,  \mid \, \} $,
in which no player has any move. Then at ``Day 1''
we have $1= \{ \, 0 \mid  \,\} $, 
$-1 = \{ \,  \mid 0 \, \} $    and $* = \{ \, 0 \mid 0 \, \} $.
All the combinatorial games are constructed by a transfinite
iteration as above in a way similar to
the von Neumann hierarchy. 

As usual in combinatorial
game theory, games and positions are identified. 
Combinatorial games are two-player
alternating games, but, as we are going to explain soon,
 the development of the theory requires 
to consider the possibility
that some player makes an arbitrarily long finite sequence
of consecutive moves on the same run (a \emph{run} 
is a sequence of moves, a \emph{play} is an alternating run). In technical terms,
one needs to  consider runs 
 which are not necessarily alternating. 
In any case, in the basic theory, infinite runs are not allowed,
though game themselves might be infinite.
For example, $ \omega = \{ \, 0, 1, 2, \dots \mid  \,\}$
is infinite, but no infinite run is possible 

There is a natural operation of \emph{(Conway) sum}
on games: in the sum $G+H$ the player who has to  move
should make a move in exactly one of  $G$ and $H$ of her choice,
say, she makes the move $H^L$ on $H$. Then the play continues
on (the transfinitely recursively defined) sum $G+H^L$. 
So, for example, in some run Left might always go on---but
she is not forced to---moving
on the second component, while Right always moves on
the first component; this is the
reason why in general games---in the example at hand, $G$ 
and $H$---we must allow the possibility of the same player
making consecutive moves.
A \emph{(Conway) zero} game is a game in which the 
player moving Second, be her Left or Right, 
has a winning strategy under the alternating rule.
It is easily seen that if $H$ is a zero game, then, for every
game $G$, the games $G$ and $G+H$ have the same outcome.
This means that if Left is to move first, she has a winning strategy
on $G+H$ if and only if she has a winning strategy on $G$;
and symmetrically if Right is to move first.    
The \emph{opposite} $-G$ of some game $G$ 
is the game in which Left and Right options are recursively
exchanged. Two games $G$ and $H$ are \emph{Conway
equivalent} if $G-H$ is Conway zero.    

 Games  are frequently considered up to
Conway equivalence  and the equivalence classes are frequently called
\emph{values}. The value of some game $G$ will be denoted
by $v(G)$. Note that games
are frequently  identified directly with values, but 
we will not use such a convention 
 here.

\subsection{Connections with sets, ordinals and partially ordered sets} \labbel{conne} 
One can interpret sets  as games on which only  Left
 can move in each subposition. In other words, one can consider a
set as a game for which the ``right membership'' relation is always 
(=hereditarily) false.  
 It turns out that each set $x$ 
is Conway equivalent to exactly one ordinal, which in fact 
is  the von Neumann
rank of $x$ (we will not need this here).
 In the case of (games Conway equivalent to) ordinals,
equivalence means that the supremum of the possible moves 
(up to equivalence) is the same,
so that, say,  $n= \{ \, 0, 1, \dots, n-1 \mid  \,\}$ 
 is equivalent to $\{  n-1 \mid  \,\}$, and 
$ \omega = \{ \, 0, 1, \dots, n, \dots \mid  \,\}$ 
 is equivalent to 
$\{ 3 ^{3^3}, 4 ^{4^{4^4}}, \dots, n^{n^{n^{\dots}}}
, \dots \mid  \,\}$.
 Recall that in the above
``bilateral'' set-theoretical-like notation, 
 the left side lists all the possible moves of the
Left  player. As we mentioned, the Right player has no available move 
in the games under consideration, hence the right-hand side
will be always empty. 

At first sight, it might appear strange that
only one player has possible moves in a two-player game.
However, the setting has some use, since we are then allowed to  apply
Conway theory\footnote{By the way, we know no explicit
description of the operation of Conway sum
among sets (we mean, disregarding Conway equivalence.
 Up to Conway equivalence,
the situation is completely clear: the sum is the Hessenberg natural sum 
of the von Neumann ranks). Some marginal comments appear
in \cite{sergam} and in \emph{A Ring structure on the Class of
 Combinatorial Games}, in preparation. A discussion about
some other operations on sets appear in \cite{Kir}.}. 
In any case, let us point out
that this setting is not essentially different from the computation
of the rank of a well-founded partially ordered set $P$.
If we associate to $P$ a game such that a Left move
amounts to choosing some element $p$ of $P$ and the 
outcome of the move is the ordered set $ \downarrow _< p$ 
of all the elements strictly smaller than $p$, then the Conway value
of this game associated to $P$ is exactly the rank of $P$.  
On the other hand, for a reader already familiar with combinatorial game theory,
we feel that the game theoretical approach looks slightly more natural and
somewhat simpler than the order theoretical approach which, in the case at hand,
we have presented in Section \ref{rank}. 

Indeed, the above point of view is confirmed by works by 
A. Blass and Y. Gurevich \cite{BG}, who implicitly
introduced the above game on $P$ (they called 
\emph{height} what we call rank here)
and then compared $P$ with another well-founded 
poset $Q$  by considering, in different terminology, the Conway
difference of the two associated games. See \cite[Section 3, in particular,
Proposition 23]{BG}. The use of games simplified many 
arguments in \cite{BG}.

\subsection{A game-theoretical characterization of $\sums$} \labbel{subsecgame} 

\begin{definition} \labbel{gamesum}
We associate to every sequence 
$\bm \alpha = ( \alpha_i) _{i < \omega} $ 
of ordinals the following game $G(\bm \alpha)$.

A Left move on $G(\bm \alpha)$ consists of choosing
some $\bar \imath < \omega$, some finite, possibly empty, $F \subseteq  \omega$ 
with $\bar \imath \notin F$  
 and some $\beta < \alpha _{\bar \imath}$ (so that no move
is possible when $ \alpha _i=0$, for every $i < \omega$). The Right
player has never  available moves, since $\bm \alpha$ is a  sequence of ordinals
(in what follows we will briefly discuss the possibility
of extending the ruleset to sequences of surreal numbers). 
The outcome of the above Left move is the game 
$G(\bm \beta )$ associated to the sequence 
 $ \bm \beta = ( \beta _i) _{i < \omega}$ defined by
\begin{equation} \labbel{defg}
\begin{aligned} 
\beta_i &= \alpha _i, && \text{if  $i \in F$, or $\alpha_i \leq \beta $,} 
\\
\beta_i & = \beta ,   && \text{otherwise}.
\end{aligned}
\end{equation}      

A simple way to justify the above definition is to observe
that the following partial order $\precsim $ between sequences of ordinals
is  well-founded:
 $ \bm \gamma  \precsim \bm \alpha $ if 
$G(\bm \gamma ) $ can be reached from $G(\bm \alpha)$ after a finite
number of moves, as given by the definition.
Indeed, if Left chooses $\beta$ in a move on 
$G(\bm \alpha)$, getting $G(\bm \beta )$,
 then $ \zeta_{\bm \beta} \leq \beta $, since $F$ is finite.
This implies that if the chosen $\bar{\imath}$ 
is such that $\alpha_ { \bar{\imath} }$ is
$\leq \zeta_{\bm \alpha }$, then   $ \zeta_{\bm \beta } < \zeta_{\bm \alpha } $.
Since there are only a finite number of z-special elements
(that is, summands $> \zeta_{\bm \alpha }$), and 
there is no infinite decreasing sequence on an ordinal,
sooner or later, after a finite number of moves, the player is forced to 
choose some $\bar{\imath}$ 
with $\alpha_{ \bar{\imath} } \leq \zeta_{\bm \alpha }$ and, as we have shown,
such a move strictly decreases the value of 
$\zeta$. The value of $\zeta$ is an ordinal itself and, again since
there is no infinite decreasing sequence of ordinals, 
after a finite number of moves, the player will reach 
(the game associated to) a sequence with
$\zeta =0$. This means moving on a sequence whose summands are
all $0$, except for finitely many of them, and any run on such a sequence 
necessarily  ends
after a finite number of moves. 

The above argument shows that
$\precsim $
is  well-founded, hence 
the definition of $G(\bm \alpha)$ can be correctly given by induction
on the rank 
 of  the sequence $\bm \alpha$ in the well-founded partial order  $\precsim $.
The argument also shows that no infinite run is possible, 
so that each $G(\bm \alpha)$ is actually a combinatorial game
in the strict sense. The game  $G(\bm \alpha)$ is equivalent to an 
ordinal, since only the Left player has available moves on each subposition;
actually, the value of  $G(\bm \alpha)$ is the rank of the sequence 
 $\bm \alpha$ in 
the  well-founded order $\precsim $.
Indeed, the value of a game
for which  only  Left  has available moves on each subposition
 is the smallest
ordinal which is strictly larger than the value of all possible moves.
Compare \cite[VII, Theorem 1.14]{Sieg}.
\end{definition}

\begin{lemma} \labbel{gamlem}
Suppose that $\sumt$ is an infinitary operation which 
is strictly monotone on z-special elements, that is, 
$\sumt$ satisfies Clause \eqref{s3}.

Then, for every 
sequence $ \bm \alpha = ( \alpha_i) _{i < \omega} $ of ordinals,
 $v(G(\bm \alpha)) \leq \sumt _{i < \omega} \alpha_{i}$.
\end{lemma}

 \begin{proof} 
By induction on $v(G(\bm \alpha))$
(the induction can be equivalently carried
over on $\sumt _{i < \omega} \alpha_{i}$). 
 If 
$v(G(\bm \alpha))=0$,
the conclusion is straightforward, since $0$ is the smallest ordinal. 
 
Otherwise, fix some nonzero sequence $\bm \alpha$ and 
suppose that the lemma holds for every 
sequence $\bm \beta $ with 
$v(G(\bm \beta )) < v(G(\bm \alpha))$.
As we have remarked in the comment
justifying Definition \ref{gamesum},
any move on $G(\bm \alpha)$  either involves some
z-special element $\alpha_ { \bar{\imath} }$  of $\bm \alpha$, 
or strictly decreases the value of $\zeta$,
hence any resulting sequence $\bm \beta $ is such that 
$\sumt _{i < \omega} \beta _{i} < \sumt _{i < \omega} \alpha_{i}$,
since, by assumption,  $\sumt$ is 
 strictly monotone on z-special elements,
hence also   strictly monotone on $\zeta$, by Remark \ref{e}.
Since  $v(G(\bm \alpha ))$  is an ordinal and $G(\bm \beta )$
  is a move on $G(\bm \alpha )$, $v(G(\bm \beta ))<v(G(\bm \alpha ))$. 
By the inductive hypothesis, for any such
$\bm \beta $,  $v(G(\bm \beta )) \leq  \sumt _{i < \omega} \beta _{i} $,
 hence  $ v(G(\bm \beta )) < \sumt _{i < \omega} \alpha  _{i}$.
Now $ v(G(\bm \alpha )) \leq  \sumt _{i < \omega} \alpha  _{i}$
follows, since $v(G(\bm \alpha ))$ is the smallest ordinal
strictly larger than $v(G(\bm \beta ))$,  
for all possible moves $G(\bm \beta )$ 
on $G(\bm \alpha )$.
 \end{proof}

\begin{theorem} \labbel{thmgam}
For every sequence $ \bm \alpha = ( \alpha_i) _{i < \omega} $ of ordinals,
$v(G(\bm \alpha)) = \sums _{i < \omega} \alpha_{i}$.
 \end{theorem}

 \begin{proof} 
The inequality $v(G(\bm \alpha)) \leq \sums _{i < \omega} \alpha_{i}$
follows from Lemma \ref{gamlem}, since $\sums$
is strictly monotone on z-special elements, by the proved part of Theorem \ref{simpth}.

In order to prove the converse inequality, we will use Theorem \ref{oc}.
In particular,  
$\sums _{i < \omega} \alpha_{i}$ has an ormf realization
by certain sets  $( A_i) _{i < \omega} $.
Recall that this means that  $\sums _{i < \omega} \alpha_{i}$
is the disjoint union of the $ A_i $,
each $A_i$ has order-type $\alpha_i$, with the 
order-preserving bijection denoted
by $h_i$,  and furthermore  the realization is ormf, namely,
 for every $ \eta  < \sums _{i < \omega} \alpha_{i}$,
$\eta \in A_i$,  
there is a finite set $F_ \eta  \subseteq \omega $ such that if 
$ \theta  < \eta  $,
$ \theta   \in A_j$, 
 and $h_j( \theta  ) \geq h_i ( \eta  )$,
then $j \in F_ \eta $
(it is no loss of generality to assume that $i\notin F_ \eta$).
Note that, given $\eta$ as above, 
$D_i =\eta \cap A_i$ ($i < \omega$) is an ormf realization of $\eta$ 
itself, so that, by Theorem \ref{oc}, if 
we let $ \delta _i$  be the order-type of  
$\eta \cap A_i$, for $i < \omega$,
then $\eta \leq \sums _{i < \omega} \delta  _{i}$.
Compare Lemma \ref{orf}.

The proof that $\sums _{i < \omega} \alpha_{i} \leq v(G(\bm \alpha))$
is by induction on
$\sums _{i < \omega} \alpha_{i}$. If 
$\sums _{i < \omega} \alpha_{i} =0$,
then
$\sums _{i < \omega} \alpha_{i} \leq v(G(\bm \alpha))$,
since $v(G(\bm \alpha))$ is an ordinal and 
$0$ is the smallest ordinal.  

Otherwise, fix some nonzero sequence $\bm \alpha $
and assume that $ \sums _{i < \omega} \beta _{i} \leq v(G(\bm \beta ))$,
for every sequence $\bm \beta $ such that 
$ \sums _{i < \omega} \beta _{i} < \sums _{i < \omega} \alpha_{i}$.  
We need to show that for every $ \eta  < \sums _{i < \omega} \alpha_{i}$ 
the Left player has a move $G(\bm \beta )$  on $G(\bm \alpha)$ 
such that  $ \eta  \leq v(G(\bm \beta )) $.
Given $ \eta $ and an ormf realization of 
$\sums _{i < \omega} \alpha_{i}$ as above,
let $\bar{\imath}$    be the unique index 
such that  $\eta  \in A_ { \bar{\imath} }$ 
and let the player move according to Definition \ref{gamesum},
choosing $\bar{\imath}$,  
$\beta= h_{\bar{\imath}} ( \eta ) $ and  $F=F_ \eta $,
getting a sequence $( \beta _i) _{i < \omega} $
such that  $\sums _{i < \omega} \beta _{i} < \sums _{i < \omega} \alpha_{i}$,
by the argument in the proof of
Lemma \ref{gamlem}, since
$\sums$ is strictly
monotone on z-special elements.

Recall that, for $j < \omega$,  $ \delta _j$  is the order-type of  
$\eta \cap A_j$.
For every $j < \omega$, we have $\delta_j \leq \beta _j$,
since if  $j \in F_ \eta $, then $  \delta _j \leq  \alpha_j =  \beta _j $;
otherwise,  if
$ \theta   \in \eta \cap  A_j$, 
 then $h_j( \theta  ) < h_{\bar{\imath}} ( \eta  )$,
thus $\delta_j $, being the order-type of 
$\eta \cap A_j$, is $\leq h_{\bar{\imath}}( \eta) = \beta = \beta _j$.
By weak monotonicity of $\sums$ (Theorem \ref{simpth}),
$\sums _{i < \omega} \delta _{i} \leq \sums _{i < \omega} \beta _{i}$,
hence 
$\eta \leq \sums _{i < \omega} \delta  _{i} \leq 
\sums _{i < \omega} \beta _{i} = v(G(\bm \beta ))$,
by the inductive assumption, since necessarily
$\sums _{i < \omega} \beta _{i} < \sums _{i < \omega} \alpha_{i}$,
as already remarked.  
\end{proof}

\begin{remark} \labbel{notall}
It follows from Lemma  \ref{gamlem} that any move on 
$G(\bm \alpha )$ has value strictly less than 
$ \sums _{i < \omega} \alpha_{i}$.

In general, we cannot always get every 
$ \gamma <  \sums _{i < \omega} \alpha_{i}$
as the value of some subposition $G(\bm \beta )$
(but, as we have shown, we have arbitrarily large values 
$<  \sums _{i < \omega} \alpha_{i}$).
For example, if $\alpha_i= \omega $ for every $i < \omega$,
any move on   
$G(\bm \alpha )$ produces infinitely many
summands equal to some $n$ (the chosen $\beta$)
together with finitely many summand equal to $ \omega$ 
(those with the index in $F$), giving the value 
$ \omega n + \omega m = \omega (n+m)$.    
 
If we want to get every 
$ \gamma <  \sums _{i < \omega} \alpha_{i}$
as a subposition, we need to modify the rules in 
Definition \ref{gamesum}, to the effect that Left
is also allowed to move at the same time on further
summands. 
 \end{remark}

\subsection{The problem of generalizing $\sums$ to the surreals} \labbel{probsurrsec} 

\begin{problem} \labbel{probgam}
If possible, find a suitable generalization of Definition \ref{gamesum}
to the case when  $ \bm \alpha  $ is a sequence of surreal
numbers or, possibly, a sequence of combinatorial games.

In case the above problem has an affirmative answer,
study the  infinitary operation which associates to 
some sequence the corresponding game (or its value).
 \end{problem}

\begin{remark} \labbel{limitat}   
Some ideas which might help solving 
the above problem can be borrowed from \cite{sergam}. 
However,
there are serious limitations to possible solutions of the above problem.
In the present remark   we assume that the reader is familiar
with surreal numbers \cite{Go,Sieg}.

(1) There is no (possibly, partial) infinitary operation $\sum^*$ 
on the class of surreal numbers such that 
  \begin{enumerate}[(a)]   
\item
$\sum^*$ is weakly monotone, and
 \item 
for converging (in the standard sense from classical analysis)
 series of real numbers, $\sum^*$ is defined and coincides with the usual series operation,
possibly,  only modulo an infinitesimal, and
\item
if all the summands are equal to some surreal $\alpha$, then $\sum^*$
is defined and 
the outcome of $\sum^*$ is the surreal product $  \omega  \cdot \alpha  $.
  \end{enumerate} 

Indeed, if such an operation exists, then
$1=  \omega  \frac{1}{ \omega } =^{\rm (c)}
 \sum ^{*} \frac{1}{ \omega } \leq ^{\rm (a)}
 \sum ^{*} \frac{1}{2^{n+2}} =  ^{\rm (b)} \frac{1}{2}  \pm \varepsilon  $,
absurd. 

(1$'$) A slightly more involved argument produces a more significant
negative result.
There is no everywhere defined infinitary operation $\sum^\triangle$ 
on the class of surreal numbers such that 
  \begin{enumerate}[(a$'$)]   
\item
$\sum^\triangle$ is weakly monotone, and
 \item 
for every surreal number $s$ and every sequence 
$( r_i) _{i < \omega} $ of real numbers 
such that $\sum _{i < \omega} r_{i}$ 
classically
converges to $r$,
we have
$\sum ^\triangle_{i < \omega} sr_i = sr$.
  \end{enumerate} 

Actually, the proof will show that it is enough to assume (b$'$)   
 when $s$ has the form $ \omega^a$, for some surreal $a$
and  it is enough to assume weak monotonicity
when the first (the smaller) sum consists only of
$0$s and $1$s; we do not need the full assumption
that $\sum^\triangle$ is everywhere defined, either:
it is enough to assume that  $\sum^\triangle_{i < \omega} 1$
is defined (and, of course, also all the sums given by
(b$'$) are defined). 
Recall the definition of $ \omega^a$ from \cite[p. 55]{Go}.   

Indeed, suppose by contradiction that 
$\sum^\triangle$ is such an operation. 
By taking $s=1$ in (b$'$), we get 
$n= 1+1+\dots +1+0+0+ \dots$, where the sum has
exactly $n$ nonzero summands.
Let $t= 1+1+1+ \dots$, computed according to  
$\sum^\triangle$. By weak monotonicity,
$t \geq n$, for every $n< \omega$.
By a classical result, eg. \cite[Theorem 5.6]{Go},
$t$ has a normal form   
$\omega ^{a} r + \omega ^{a_1 } r_1 + \dots$
with $a>a_1> \dots$ surreals and $ r, r_1 \dots$ nonzero reals 
(the normal form might contain transfinitely many
summands, but here it will be enough to consider 
the leading monomial).
The condition $t \geq n$, for every $n< \omega$,
means that $a>0$ and $r>0$, since the surreal order
is given by the lexicographic order of the coefficients, 
considering  larger leading exponents first.
Of course, $n$ has normal
form $ \omega^0 n$.
This means that   also
$\omega ^{a} r^* > n$, for every $n< \omega$ 
and $r^* >0$.
Then take a sequence $( r_i) _{i < \omega} $   
of strictly positive real numbers whose sum is
$r/2$.
We get 
$t= \sum^\triangle_{i < \omega} 1 \leq ^{\rm (a')}
 \sum^\triangle _{i < \omega} \omega ^{a}r_i = ^{\rm (b')}
\omega ^{a} \frac{r}{2}  $, but this is a contradiction,
since the coefficient of the leading monomial
in the expression of $t$ is $r>0$,
thus $t > \omega ^{a} \frac{r}{2}$.

(2) 
Since surreal numbers are linearly ordered, we can apply
Definition \ref{gamesum} to sequences of surreal numbers;
however, in this way we do not get a combinatorial game, since then infinite
runs are possible. 

For example, if $\alpha_n = -n - \frac{1}{2} $, for every $n < \omega$, 
Left can move on  $\alpha_0 =-\frac{1}{2}$, turning it to $-1$.
If we apply  \eqref{defg} literally,  then, no matter the choice of $F$,
all the subsequent $\alpha_i$ remain unchanged, so that  
Left can continue moving on  $\alpha_1 =-1-\frac{1}{2}$, turning it to $-2$,
and so on\footnote{A definition along the above lines might be 
justified some way, though not all the consequences are 
completely clear, so far. See \cite{Li}.}.

This shows that new ideas are necessary, when trying to 
solve Problem \ref{probgam}.

(3) However, we can modify Definition \ref{gamesum} in such a way 
that we get a combinatorial game without infinite runs
even when we deal with surreal numbers.

Recall that a \emph{surreal number} $s$, when defined by means
of a \emph{sign sequence},  is a function from some
ordinal $ \ell (s)  $ to the set $\{ -, +\}$
(thus $ \ell (s)  $ is the first ordinal $\gamma$ such that 
$s(\gamma)$ is not defined). 
Left moves on $s$ by removing some $+$ sign,
together with all the signs that follow, and Right
moves in a similar way by removing a $-$ sign.
More formally, if $\gamma \leq  \ell (s)$,
we define the \emph{$\gamma$-initial segment of $s$}, in symbols, 
$ s { \restriction  \gamma }  $  as the surreal 
$t$ such that $\ell(t) = \gamma $
and $t(\delta) = s( \delta )$,  
for every $\delta< \gamma $.  
With the above terminology, Left moves on $s$
by choosing some $\gamma < \ell (s)$
such that $s( \gamma ) ={+}$. The outcome of the move is
$ s { \restriction  \gamma }  $.

To a sequence $\mathbf s =( s_i) _{i < \omega} $ of surreal numbers,
let us associate a game $G(\mathbf s)$ as follows. 
A Left move on $G(\mathbf s)$ consists of choosing
some $\bar \imath < \omega$, some finite $F \subseteq  \omega$ 
with $\bar \imath \notin F$  
 and some $\gamma < \ell (s _{\bar \imath } )$
such that $s_{ \bar \imath } ( \gamma ) ={+}$.  
The outcome of the above Left move is the game 
$G(\mathbf t)$ associated to the sequence 
$\mathbf t =( t_i) _{i < \omega} $ defined by
\begin{equation*} \labbel{defgt}
\begin{aligned} 
t_i &= s _i, && \text{if  $i \in F$, \  or  \ $ \ell (s_i) \leq \gamma  $,
\ or  \ $s_i( \delta )= -$, for every $\delta \geq \gamma $,} 
\\
t_i & = s_i { \restriction  \gamma _i} ,   && \text{otherwise, where 
$\gamma_i$ is the smallest ordinal $\geq \gamma $
  such that $s( \gamma _i) = {+}$.}
\end{aligned}
\end{equation*}      
Right moves are defined in the symmetrical way, with the signs
$+$ and $-$ exchanged. 

The arguments presented in Definition \ref{gamesum}
can be adapted to show that if $\mathbf s$ is a sequence
 of surreal numbers, then
 $G(\mathbf s)$ has no infinite run. 
 
(4) Another tentative definition, which works for arbitrary 
combinatorial games, can be obtained by merging some ideas
from the definition of the Conway product,
since (c) in (1) above is a  desirable property.

In detail, let $ \mathbf H= ( H_i) _{i < \omega} $ be a sequence of combinatorial
games. Let us define $G(\mathbf H)$ as follows.
The Player supposed to move chooses some $\bar{\imath} < \omega $,
some position $H_{ \bar{\imath} }^P$ on $H_{ \bar{\imath} }$
he can move to, and some finite set $ F \subseteq  \omega$     
such that $\bar{\imath} \notin F$. Define another sequence 
$ \mathbf K= ( K_i) _{i < \omega} $ by
\begin{equation} \labbel{defgtt}
K_i = H_{ \bar{\imath} }^P,
\text{ if  $H_{ \bar{\imath} }^P$ is a Player's option in $K_i$}; 
\end{equation}      
 otherwise, $K_i$ is the game in which the Player has only those
options of $H_i$ which have birthday $<$ than the birthday
of    $H_{ \bar{\imath} }^P$, while the options of the Other
player are exactly her options in $H_i$.

The outcome of such a move is the game
$ H_{i_1} - K_{i_1}+ H_{i_2} - K_{i_2} + \dots + H_{i_n} - K_{i_n}
+ G ( \mathbf K) $, where $F= \{ i_1, \dots, i_n \} $. 
Of course, the outcome is simply $G ( \mathbf K)$
if the player chooses $F= \emptyset $.  
\end{remark}

\begin{problem} \labbel{!}
One can define the Ring $\mathbf R$ of the (Hahn) formal series of the type
\begin{equation} \labbel{!!}  
\sum _{  \beta < \alpha  }  x ^ {a_ \beta} r_ \beta   
  \end{equation}     
for $\alpha$ an ordinal,  $(r_ \beta ) _{ \beta < \alpha } $  a sequence
of nonzero real numbers and $(a_ \beta ) _{ \beta < \alpha } $ 
a sequence of surreal numbers such that $a _ \beta > a_ \gamma $,
for $ \beta > \gamma $. This is a dual (order reversed) notion with respect to 
Conway normal forms, when considered as formal  series.

Can we assign a surreal value to all the series as above, when $x$
is interpreted as $ \omega$? In detail,  is there a ring homomorphism
from $\mathbf R$ as defined above to the class of surreal numbers,
such that $x^a$ is sent to $ \omega^a$, for every surreal number $a$?   
Of course, counterintuitive outcomes would necessarily emerge, for example
(identifying $x$ with $w$)
$1+ \omega + \omega ^2 + \dots $ should be sent to $  \frac{1}{1- \omega } $. 
But, above all, equation \eqref{!!} is not sufficient to define
infinite sums for \emph{all} sequences of surreal numbers.   
 \end{problem}

\section{Appendix II. A direct proof of Theorem \ref{simpth}} \labbel{appii} 

We have obtained the ``minimality''
part of  Theorem \ref{simpth} at the end of Section \ref{rank} 
in a rather indirect way,  
as a consequence of  the
rank-theoretical 
characterization of $\sums$ obtained in 
Section \ref{rank}   
(equivalently,  game-theoretical characterization
in Section \ref{game}).
The proofs in both sections rely heavily on the characterization in 
Theorem \ref{oc}.

For the reader's convenience, in this appendix we present a direct proof
that if $\sumt$ is an  operation
which is weakly monotone and 
strictly monotone on z-special elements, then
$\sums _{i < \omega} \alpha_{i} \leq \sumt _{i < \omega} \alpha_{i}$
for every sequence $\bm \alpha$. The proof is by induction on
$\sums _{i < \omega} \alpha_{i}$ and
 is very similar to  the arguments in Proposition \ref{minnat}.

The base case is straightforward, 
since $0$ is the smallest ordinal. 

 For the successor step,
suppose that $\sums _{i < \omega} \alpha_{i} \leq \sumt _{i < \omega} \alpha_{i}$
for every sequence with 
$ \gamma = \sums _{i < \omega} \alpha_{i}$.
If $\bm \beta $ is a sequence with 
$ \gamma+1 = \sums _{i < \omega} \beta _{i}$,
then, if case \eqref{seconda} applies,
  at least one  $ \beta  _{i_j}^\heartsuit$ is a successor ordinal
and, 
if case \eqref{prima} applies,
  at least one  $ \beta  _{i_j}^\diamondsuit$ is a successor ordinal.
Indeed, in case \eqref{seconda},
$\zeta_ {\bm \beta }  \x \omega$ is always a limit
ordinal or $0$.  
In case \eqref{prima},
$(\hat \zeta_ {\bm \beta }  \x \omega)  \+ \omega ^ {\xi_0}$ is always a limit
ordinal, since $\xi_0$ is assumed to be limit, in particular,
$\xi_0 > 0$. 

We first deal with case \eqref{seconda}.
Let    $\bm \alpha $ be the sequence in which
$ \beta  _{i_j}$ is changed to its predecessor $\alpha _{i_j}=
\beta  _{i_j}-1$ and all the other elements
are unchanged. Then
$\sums _{i < \omega} \alpha_{i} < \sums _{i < \omega} \beta _{i}$ 
by strict monotonicity of the finitary natural sum
and,  by the inductive assumption,
  $\sums _{i < \omega} \alpha_{i} \leq \sumt _{i < \omega} \alpha_{i}$.
Changing the value of just one summand does not alter
$\zeta$ and $\varepsilon$, hence the e-special elements are the same
for $\bm \beta $ and for $\bm \alpha  $.
Note that 
$ \beta  _{i_j}^\heartsuit \geq 1$,
since it is a successor ordinal, hence
$\beta  _{i_j}  > \zeta_ {\bm \beta}$
(thus also $\beta  _{i_j}  \geq \varepsilon_ {\bm \beta} $), since  
$ \beta  _{i_j}  = \zeta_ {\bm \beta} + \beta  _{i_j}^\heartsuit $,
by construction.
Hence $ \beta  _{i_j}$ is z-special and 
$\sumt _{i < \omega} \alpha_{i} < \sumt _{i < \omega} \beta _{i}$,
if $\sumt$ satisfies \eqref{s3'}. Then   
\begin{equation*} 
\sums _{i < \omega} \beta _{i} = 
\big (\sums _{i < \omega} \alpha_{i}\big) +1  
\leq \big ( \sumt _{i < \omega} \alpha_{i} \big ) +1 \leq
\sumt _{i < \omega} \beta _{i}
 \end{equation*}     

The case in which \eqref{prima} applies is entirely similar,
dealing with  $ \beta  _{i_j}^\diamondsuit$ instead.
Note that if \eqref{prima} applies, then 
 $\varepsilon_ {\bm \alpha}   = \zeta_ {\bm \alpha} $
is a limit ordinal. Hence, if 
$\beta  _{i_j}= \alpha   _{i_j}+1$,
then $\beta  _{i_j}$ is e-special if and only if
$ \alpha   _{i_j}$    is e-special.

It remains to deal with
 the limit step.
Suppose that 
$\gamma$ is a limit ordinal and
$\sums _{i < \omega} \alpha_{i} \leq \sumt _{i < \omega} \alpha_{i}$
for every sequence $\bm \alpha $ such that  
$ \sums _{i < \omega} \alpha_{i} < \gamma $.

If $\bm \beta $ is a sequence such that 
$ \sums _{i < \omega} \beta _{i} = \gamma $,
then all the summands in the expression giving
$ \sums _{i < \omega} \beta _{i}$, either from
\eqref{prima} or from \eqref{seconda},
are limit ordinals (or possibly $0$), since $\gamma$ 
is a limit ordinal and in
\eqref{prima} and \eqref{seconda} we are dealing with
natural sums.
In order to avoid confusion with the summands
of the sequence, we will use the expression
\emph{S-summand} for a summand  in the formulas
\eqref{prima} or  \eqref{seconda} giving $ \sums _{i < \omega} \beta _{i}$.

Choose a nonzero S-summand with the smallest exponent
in its normal form. 
This might be 
$ \omega ^{\xi_0}$, or possibly come from
the normal form of
some $ \beta  _{i_j}^\diamondsuit$
in case \eqref{prima}, or it might be 
$ \omega ^{\xi_0+1}$
or come from some 
$ \beta  _{i_j}^\heartsuit$
in case \eqref{seconda}.
There might be two or more S-summands
 with the same smallest exponent
in  normal form;
if $ \omega ^{\xi_0}$ appears among such summands
give the preference to  $ \omega ^{\xi_0}$.
We now need to consider all the above cases.

(a) First, consider the case when an
S-summand with the smallest exponent comes from
$ \beta  _{i_j}^\diamondsuit$.
By the choice of $ \beta  _{i_j}^\diamondsuit$
 and by a finitary generalization  of   \ref{lim},
$\sums _{i < \omega} \beta _{i} = 
(\hat \zeta_ {\bm \beta }  \x \omega)  \+ \omega ^ {\xi_0} \+
\beta  _{i_1}^\diamondsuit \+ \dots \+ \beta  _{i_h}^\diamondsuit $
is the supremum of 
the set of ordinals of the form
\begin{equation*}\labbel{ld} 
\lambda_ \delta=    
  (\hat \zeta_ {\bm \beta }  \x \omega)  \+ \omega ^ {\xi_0} \+
\beta  _{i_1}^\diamondsuit \+ \dots \+
\beta  _{i_{j-1}}^\diamondsuit 
\+ \delta \+  \beta  _{i_{j+1}}^\diamondsuit 
\+ \dots \+ \beta  _{i_h}^\diamondsuit  
   \end{equation*}
with $\delta$ varying among those ordinals which
are $< \beta  _{i_j}^\diamondsuit$. 

Fix some  $\delta < \beta  _{i_j}^\diamondsuit$
such that $ \hat \zeta_ {\bm \beta}+ \delta \geq 
\zeta_ {\bm \beta} $. This is possible, 
since  $\beta  _{i_j}$ is e-special, 
thus $\beta  _{i_j} \geq \zeta_ {\bm \beta}$,
actually,   $\beta  _{i_j} > \zeta_ {\bm \beta}$, since
in the case $\beta  _{i_j} = \zeta_ {\bm \beta}$
we should have chosen $ \omega ^{\xi_0}$, instead.
Let
  $\bm \alpha $ be the sequence in which
$ \beta  _{i_j} = \hat \zeta_ {\bm \beta}+ \beta  _{i_j}^\diamondsuit  $
 is changed to 
 $\alpha _{i_j}=\hat \zeta_ {\bm \beta}+ \delta $.
As above, changing the value of just one summand does not alter
$\zeta$ and $\varepsilon$, hence 
\eqref{prima} gives
$\sums _{i < \omega} \alpha_{i}=
  \lambda _ \delta $.  
By the inductive assumption and monotonicity of
$\sumt$,
  $\lambda _ \delta =
\sums _{i < \omega} \alpha_{i} \leq \sumt _{i < \omega} \alpha_{i}
\leq \sumt _{i < \omega} \beta _{i}$.
Now let $\delta$ vary, with the above
constraints $\delta < \beta  _{i_j}^\diamondsuit$
and $ \hat \zeta_ {\bm \beta}+ \delta \geq 
\zeta_ {\bm \beta} $. Since the above argument works for unboundedly many
$ \delta < \beta  _{i_j}^\diamondsuit$, we get 
$\sums _{i < \omega} \beta _{i} = 
\sup _{ \delta < \beta  _{i_j}^\diamondsuit}  \lambda _ \delta 
\leq \sumt _{i < \omega} \beta _{i}$.

(b) The case when an
S-summand with the smallest exponent is
$ \beta  _{i_j}^\heartsuit$ and we apply \eqref{seconda}
  is entirely similar, actually, slightly simpler,
since we do not need to deal with the distinction
between 
$\hat \zeta_ {\bm \beta}$ and 
 $ \zeta_ {\bm \beta}$. 

(c)  Consider now the case when
\eqref{prima} applies and 
an
S-summand with the smallest exponent
is 
$ \omega ^{\xi_0}$.
Because of the conditions required in
order to apply \eqref{prima},  
for every 
$ \delta^* < \zeta_ {\bm \beta }$,
there are infinitely many $\beta_i$
such that  $ \delta^* \leq \beta _i < \zeta_ {\bm \beta }$.
Fix some ordinal $ \delta < \omega ^{\xi_0}$, hence,
if we set
 $ \delta^* = \hat \zeta_ {\bm \beta } + \delta   $, 
then $ \delta^* < \zeta_ {\bm \beta }$.
Consider the sequence $\bm \alpha$ 
 in which all the $\beta_i$  such that 
$ \delta^*=  \hat \zeta_ {\bm \beta } + \delta \leq \beta _i < \zeta_ {\bm \beta }$ 
 are changed to 
$\hat \zeta_ {\bm \beta } + \delta$
(note that if $\beta _i \geq \zeta_ {\bm \beta }$,
we do not change the value of $\beta_i$).
Thus $\bm \alpha \leq \bm \beta $, $ \zeta_ {\bm \alpha  }=
\hat \zeta_ {\bm \beta }+ \delta $ and Clause
\eqref{seconda} applies to    $\bm \alpha  $, hence
 $\sums _{i < \omega} \alpha_{i} = (\hat \zeta_ {\bm \beta } \x \omega )
\+ ( \delta   \x \omega ) \+ \alpha _{i_1}^\heartsuit 
\+ \dots \+ \alpha _{i_h}^\heartsuit $. 
Since exactly those $\beta_i$ with  
$  \hat \zeta_ {\bm \beta } + \delta \leq 
\beta _i < \zeta_ {\bm \beta }$ are changed to 
$ \hat \zeta_ {\bm \beta } + \delta$,
the set of indexes for the e-special elements relative
to $\bm \alpha$ is equal to 
the set of indexes for the e-special elements relative
to $\bm \beta $.
For each such index, we have 
$ \hat \zeta_ {\bm \beta } +  \beta  _{i_j}^\diamondsuit  =
 \beta  _{i_j} =  \alpha   _{i_j} =
 \zeta_ {\bm \alpha  } + \alpha   _{i_j} ^\heartsuit=
 \hat \zeta_ {\bm \beta } + \delta + \alpha   _{i_j} ^\heartsuit$,
hence 
$ \beta  _{i_j}^\diamondsuit  = \delta + \alpha   _{i_j} ^\heartsuit$,
thus  $ \beta  _{i_j}^\diamondsuit  \geq \alpha   _{i_j} ^\heartsuit$.
But also
$ \beta  _{i_j}^\diamondsuit  \leq \alpha   _{i_j} ^\heartsuit$.
Indeed, write $\alpha   _{i_j} ^\heartsuit$ in normal form:
the smallest exponent of $\alpha   _{i_j} ^\heartsuit$ should be equal to  
the smallest exponent of  $\beta  _{i_j}^\diamondsuit$
in normal form, since
 $ \beta  _{i_j}^\diamondsuit  = \delta + \alpha   _{i_j} ^\heartsuit$,
$ \delta < \omega ^{\xi_0}$ and, by assumption, 
the smallest exponent of $  \omega ^{\xi_0}$
is $\leq$ than the smallest exponent of 
$\beta  _{i_j}^\diamondsuit $.
Thus $\delta$ is absorbed in the 
ordinal (not natural!) sum $\delta + \alpha   _{i_j} ^\heartsuit$,
hence 
$ \beta  _{i_j}^\diamondsuit  = \delta + \alpha   _{i_j} ^\heartsuit
= \alpha   _{i_j} ^\heartsuit$.

Thus we have  
$ \sums _{i < \omega} \alpha_{i} = (\hat \zeta_ {\bm \beta } \x \omega )
\+ ( \delta   \x \omega ) \+ \beta  _{i_1}^\diamondsuit 
\+ \dots \+ \beta  _{i_h}^\diamondsuit$. 
Since we are applying 
\eqref{prima} in evaluating $\sums _{i < \omega} \beta _{i}$,
by the assumptions,  $\xi_0$ is a limit ordinal,
hence  $\delta   \x \omega < \omega ^{\xi_0}$,
since $ \delta < \omega ^{\xi_0}$.  
This shows that
$\sums _{i < \omega} \alpha_{i} < \sums _{i < \omega} \beta _{i}$,
a fact which follows also from the 
proof of Lemma \ref{lemmo}(iv).
In particular, we can apply the inductive hypothesis, getting 
$\sums _{i < \omega} \alpha_{i} \leq 
\sumt _{i < \omega} \alpha_{i} 
\leq \sumt _{i < \omega} \beta _{i}$,
by monotonicity of $\sumt$. 
 
Now let $\delta$ vary.
By above, we have
$ \lambda _ \delta =
(\hat \zeta_ {\bm \beta } \x \omega )
\+ ( \delta   \x \omega ) \+ \beta  _{i_1}^\diamondsuit 
\+ \dots \+ \beta  _{i_h}^\diamondsuit \leq 
\sumt _{i < \omega} \beta _{i}$, for every 
$\delta < \omega ^{\xi_0}$.
As we mentioned, since   
$\xi_0$ is limit,  
$\delta   \x \omega < \omega ^{\xi_0}$, hence
all the exponents in the normal form of 
$\delta   \x \omega $ are smaller than all the exponents
in the normal forms of the other S-summands giving  
$\lambda _ \delta $. 
By a generalization of Remark \ref{lim},  
\begin{align*}  
 \sup _{\delta < \omega ^{\xi_0}} \lambda _ \delta &=
(\hat \zeta_ {\bm \beta } \x \omega )
\+ \left(\sup _{\delta < \omega ^{\xi_0}}( \delta   \x \omega ) \right)
 \+ \beta  _{i_1}^\diamondsuit 
\+ \dots \+ \beta  _{i_h}^\diamondsuit 
\\
&=
(\hat \zeta_ {\bm \beta } \x \omega )
\+  \omega ^{\xi_0}
 \+ \beta  _{i_1}^\diamondsuit 
\+ \dots \+ \beta  _{i_h}^\diamondsuit 
=
\sums _{i < \omega} \beta  _{i}.
\end{align*} 
Since 
$ \lambda _ \delta  \leq 
\sumt _{i < \omega} \beta _{i}$, for every 
$\delta < \omega ^{\xi_0}$, we get 
$\sums _{i < \omega} \beta  _{i} \leq 
\sumt _{i < \omega} \beta  _{i}$, what we had to show.

(d) We now deal with 
the case when
 there are infinitely many $\beta_i $
equal to  $\zeta_ {\bm \beta }$ (in particular,
\eqref{seconda} applies) and, again, 
an S-summand with the smallest exponent
is 
$ \omega ^{\xi_0}$.
Let $\zeta_ {\bm \beta }  =
 \omega ^ {\xi_s}  + 
\dots
+ \omega ^ {\xi_1}
+ \omega ^ {\xi_0}$ in normal form.

Fix some $n< \omega$.
Let $\bm \alpha $ be the sequence
obtained from $\bm \beta  $ by  changing to 
$ \omega ^ {\xi_s}  + \dots + \omega ^ {\xi_1}$
 all but $n$-many
$\beta_i$, among those which are equal to $\zeta_ {\bm \beta }$. 
Then $\zeta_ {\bm \alpha }=  \omega ^ {\xi_s}  + 
\dots+ \omega ^ {\xi_1} $.
If  $ \beta    _{i_j} $ is e-special
with respect to  $\bm \beta  $, then 
$ \beta    _{i_j} > \zeta_ {\bm \beta }$, since
infinitely many $\beta_i$ are equal to
$\zeta_ {\bm \beta }$; moreover, 
$\alpha   _{i_j}$ is  
e-special
with respect to  $\bm \alpha  $.
We have
$\omega ^ {\xi_s}  +  \dots + \omega ^ {\xi_1}+ \omega ^ {\xi_0}
+ \beta    _{i_j} ^\heartsuit
= \zeta_ {\bm \beta } + \beta    _{i_j} ^\heartsuit=
\beta    _{i_j} = \alpha     _{i_j}=
 \zeta_ {\bm \alpha  } + \alpha     _{i_j} ^\heartsuit=
\omega ^ {\xi_s}  +  \dots + \omega ^ {\xi_1}+ \alpha     _{i_j} ^\heartsuit$,
thus 
$\omega ^ {\xi_0}+ \beta    _{i_j} ^\heartsuit=
\alpha     _{i_j} ^\heartsuit$. 
Actually, $ \beta    _{i_j} ^\heartsuit=
\alpha     _{i_j} ^\heartsuit$, since all the exponents in the
normal form of $ \beta    _{i_j} ^\heartsuit $  are 
$> \xi_0$, by the assumptions in the present case, hence  $\omega ^ {\xi_0} $
is absorbed in the sum $\omega ^ {\xi_0}+ \beta    _{i_j} ^\heartsuit$. 

But there are more e-special elements for the sequence
$\bm \alpha $: they correspond to those 
$\beta_i =  \zeta_ {\bm \beta }$ which have \emph{not}
been changed to   $ \omega ^ {\xi_s}  + \dots + \omega ^ {\xi_1}$,
so $\alpha_i = \beta _i  = \zeta_ {\bm \beta }$ for
such indexes. In this case,
computing $\alpha   _{i} ^\heartsuit$  with respect to  $\bm \alpha $,
we get
$ \omega ^ {\xi_s}  +  \dots + \omega ^ {\xi_1} +  \alpha   _{i} ^\heartsuit = 
\zeta_ {\bm \alpha  } + \alpha   _{i} ^\heartsuit =
\alpha     _{i} =  \zeta_ {\bm \beta } =
 \omega ^ {\xi_s}  + \dots + \omega ^ {\xi_1} + \omega ^ {\xi_0}$,
thus
$\alpha     _{i} ^\heartsuit = \omega ^ {\xi_0}$.
We have $n$ indexes as above, hence
 $\sums _{i < \omega} \alpha  _{i} = 
(\zeta_ {\bm \alpha  } \x \omega ) \+ (\omega ^ {\xi_0} \x n) 
\+ \alpha   _{i_1}^\heartsuit 
\+ \dots \+ \alpha   _{i_h}^\heartsuit =
(\zeta_ {\bm \alpha  } \x \omega ) \+ (\omega ^ {\xi_0} \x n) 
\+ \beta  _{i_1}^\heartsuit 
\+ \dots \+ \beta  _{i_h}^\heartsuit $. 

Since we have already proved that $\sums$ satisfies \eqref{s3},
 $\sums _{i < \omega} \alpha   _{i} < 
\sums _{i < \omega} \beta  _{i}$, hence,
by the inductive hypothesis and monotonicity of $\sumt$,
$ \sums _{i < \omega} \alpha   _{i} = 
(\zeta_ {\bm \alpha  } \x \omega ) \+ (\omega ^ {\xi_0} \x n)  
\+ \beta  _{i_1}^\heartsuit 
\+ \dots \+ \beta  _{i_h}^\heartsuit 
=\sums _{i < \omega} \alpha  _{i} \leq 
\sumt _{i < \omega} \alpha   _{i} 
\leq \sumt _{i < \omega} \beta   _{i}$.

Now let $n$ vary. 
Again by a variation on Remark \ref{lim}, we have 
\begin{align*}
   \sums _{i < \omega} \beta   _{i}
&=
\omega ^ {\xi_s+1}  \+ \dots \+ \omega ^ {\xi_1+1} \+ \omega ^ {\xi_0+1}
\+ \beta  _{i_1}^\heartsuit 
\+ \dots \+ \beta  _{i_h}^\heartsuit
\\
& =
(\zeta_ {\bm \alpha  } \x \omega ) \+ 
( \sup _{ n< \omega}  (\omega ^ {\xi_0} \x n)) 
\+ \beta  _{i_1}^\heartsuit 
\+ \dots \+ \beta  _{i_h}^\heartsuit
\\
& = ^{\text{Remark \ref{lim}}} 
\sup _{ n< \omega} \big(
(\zeta_ {\bm \alpha  } \x \omega ) \+
 ( \omega ^ {\xi_0} \x n) 
\+ \beta  _{i_1}^\heartsuit 
\+ \dots \+ \beta  _{i_h}^\heartsuit \big) \leq 
\sumt _{i < \omega} \beta   _{i}.
\end{align*}  

(e) The remaining case to be dealt with
is when, again, 
an S-summand with the smallest exponent
is $ \omega ^{\xi_0}$, but  there are  only finitely many $\beta_i $
equal to  $\zeta_ {\bm \beta }$ and
\eqref{seconda} applies (the case when \eqref{prima} 
applies has been treated in (c)).
Then, by Remark \ref{determ}(b)(c),  
$ \omega ^{\xi_0}$ is a limit ordinal, hence
$ \xi_0 >0$.

As in (c),
for every 
$ \delta^* < \zeta_ {\bm \beta }$,
there are infinitely many $\beta_i$
such that  $ \delta^* \leq \beta _i < \zeta_ {\bm \beta }$.
Fix some ordinal $ \delta < \omega ^{\xi_0}$, hence
 $ \delta^* = \omega ^ {\xi_s}  + \dots + \omega ^ {\xi_1} + \delta
   < \zeta_ {\bm \beta }$.
Consider the sequence $\bm \alpha$ 
 in which all the $\beta_i$  such that 
$ \delta^* \leq \beta _i < \zeta_ {\bm \beta }$ 
 are changed to 
$ \delta^*$.
Thus $\bm \alpha \leq \bm \beta $, $ \zeta_ {\bm \alpha  }=
\delta^* $ and Clause
\eqref{seconda} applies to    $\bm \alpha  $, hence
 $\sums _{i < \omega} \alpha_{i} = 
( \omega ^ {\xi_s+1}  + \dots + \omega ^ {\xi_1+1} )
\+ ( \delta   \x \omega ) \+ \alpha _{i_1}^\heartsuit 
\+ \dots \+ \alpha _{i_h}^\heartsuit $. 
The set of indexes for the e-special elements relative
to $\bm \alpha$ is equal to 
the set of indexes for the e-special elements relative
to $\bm \beta $.
For each such index, we have 
$  \omega ^ {\xi_s}  + \dots + \omega ^ {\xi_1} + \omega ^ {\xi_0} 
+   \beta  _{i_j}^\heartsuit =
  \zeta_ {\bm \beta } +  \beta  _{i_j}^\heartsuit  =
 \beta  _{i_j} =  \alpha   _{i_j} =
 \zeta_ {\bm \alpha  } + \alpha   _{i_j} ^\heartsuit=
\omega ^ {\xi_s}  + \dots + \omega ^ {\xi_1} + \delta + \alpha   _{i_j} ^\heartsuit$,
hence 
$ \omega ^ {\xi_0} +   \beta  _{i_j}^\heartsuit   = \delta + \alpha   _{i_j} ^\heartsuit$.
As above, by the assumptions on $\omega ^ {\xi_0}$,
the first summands are absorbed, hence  
$   \beta  _{i_j}^\heartsuit   =  \alpha   _{i_j} ^\heartsuit$,
thus 
$\sums _{i < \omega} \alpha_{i} = 
( \omega ^ {\xi_s+1}  + \dots + \omega ^ {\xi_1+1} )
\+ ( \delta   \x \omega ) \+ \beta  _{i_1}^\heartsuit 
\+ \dots \+ \beta  _{i_h}^\heartsuit $, in particular, 
$\sums _{i < \omega} \alpha_{i} < 
( \omega ^ {\xi_s+1}  + \dots + \omega ^ {\xi_1+1} + \omega ^ {\xi_0+1} )
 \+ \beta  _{i_1}^\heartsuit \+ \dots \+ \beta  _{i_h}^\heartsuit
= \sums _{i < \omega} \beta _{i} $.
By the inductive assumption and monotonicity,
$\sums _{i < \omega} \alpha_{i}  \leq 
\sumt _{i < \omega} \alpha_{i} \leq \sumt _{i < \omega} \beta _{i}$.

As usual by now, by the analogue of Remark \ref{lim} 
and the assumptions on $ \omega ^ {\xi_0}$, 
letting $\delta$ vary
\begin{align*}
 \sums _{i < \omega} \beta _{i}
&= 
 \omega ^ {\xi_s+1}  \+ \dots \+ \omega ^ {\xi_1+1} \+ \omega ^ {\xi_0+1} 
 \+ \beta  _{i_1}^\heartsuit \+ \dots \+ \beta  _{i_h}^\heartsuit 
\\
&=
 \omega ^ {\xi_s+1}  \+ \dots \+ \omega ^ {\xi_1+1} \+
(\sup _{\delta < \omega ^{\xi_0}} (\delta \x \omega ) )  
 \+ \beta  _{i_1}^\heartsuit \+ \dots \+ \beta  _{i_h}^\heartsuit 
\\
&=
\sup _{\delta < \omega ^{\xi_0}} \big( 
 \omega ^ {\xi_s+1}  \+ \dots \+ \omega ^ {\xi_1+1} \+
 (\delta \x \omega ) 
 \+ \beta  _{i_1}^\heartsuit \+ \dots \+ \beta  _{i_h}^\heartsuit
\big)
\leq \sumt _{i < \omega} \beta _{i}
 \end{align*} 
where we have used the fact that
$\sup _{\delta < \omega ^{\xi_0}} (\delta \x \omega ) =
\omega ^ {\xi_0+1}$, since  $\xi_0$
is not limit (and not $0$).

\renewcommand\refname{Additional References}

\end{document}